\documentclass[11pt]{article}
\usepackage{amsmath,amsfonts,amssymb,amsthm}
\usepackage{latexsym, xspace, enumerate}
\usepackage[mathscr]{eucal}
\usepackage[all]{xypic}
\usepackage{hyperref}
\usepackage{amscd}
\usepackage{enumerate}
\usepackage{vmargin}
\setmarginsrb{28mm}{20mm}{28mm}{25mm}%
          {10mm}{7mm}{10mm}{10mm}

\numberwithin{equation}{section}

\newtheorem{theorem}{Theorem}[subsection]

\newtheorem{conjecture}[theorem]{Conjecture}
\newtheorem{corollary}[theorem]{Corollary}

\newtheorem{em-definition}[theorem]{Definition}
\newtheorem{em-example}[theorem]{Example}

\newtheorem{lemma}[theorem]{Lemma}

\newtheorem{problem}[theorem]{Problem}
\newtheorem{proposition}[theorem]{Proposition}
\newtheorem{em-remark}[theorem]{Remark}

\newtheorem{fact}[theorem]{Fact}
\newtheorem{question}[theorem]{Question}
\newenvironment{example}{\begin{em-example} \em }{ \end{em-example}}
\newenvironment{remark}{\begin{em-remark} \em }{ \end{em-remark}}
\newenvironment{definition}{\begin{em-definition} \em }{ \end{em-definition}}

\newcommand{\F}{{\mathcal{M}}}
\def\P{\mathbf{P}}

\def\Q{{\mathbb Q}}

\def\R{{\mathbb R}}

\def\Z{{\mathbb Z}}
\def\T{{\mathbb T}}
\def\N{{\mathbb N}}
\def\J{\mathbb J}

\def\C{{\mathcal C}}
\def\CC{{\mathcal I}}

\def\c{\mathcal C}
\def\1{\emph 1)}
\def\2{\emph 2)}
\def\3{\emph 3)}
\def\4{\emph 4)}

\def\Aut{\mathrm{Aut}}
\def\End{\mathrm{End}}

\def\U{\mathcal U}

\def\ent{\mathrm{ent}}
\def\aent{\ent^\star}

\def\Flow{\mathrm{Flow}}

\def\cov{\mbox{{\rm cov}}}
\def\pgp{\mathrm{Pol}}
\def\P{\mathbf{P}}

\def\mod{\mathrm{mod}}

\def\H{\mathfrak{H}}
\def\HB{\mathfrak H_b}
\def\abg{\mathbf{AbGrp}}
\def\af{\mathbf{Flow}}
\global\def\f{\phi}
\newcommand{\ms}[1]{\mathscr{#1}}

\title{Topological Entropy and Algebraic Entropy\\ for group endomorphisms}
\author{Dikran Dikranjan
\\{\footnotesize {\tt  dikran.dikranjan@dimi.uniud.it}} 
\\{\footnotesize Dipartimento di Matematica e Informatica,}
\\{\footnotesize Universit\`{a} di Udine,}
\\{\footnotesize Via delle Scienze, 206 - 33100 Udine, Italy} 
 \and Anna Giordano Bruno
\\{\footnotesize {\tt  anna.giordanobruno@uniud.it}} 
\\{\footnotesize Dipartimento di Matematica e Informatica,}
\\{\footnotesize Universit\`{a} di Udine,}
\\{\footnotesize Via delle Scienze, 206 - 33100 Udine, Italy}
 }

\date{
}

\begin{document}

\maketitle

\bigskip
\begin{abstract}
The notion of entropy appears in many fields and this paper is a survey about entropies in several branches of Mathematics. We are mainly concerned with the topological and the algebraic entropy in the context of continuous endomorphisms of locally compact groups, paying special attention to the case of compact and discrete groups respectively. The basic properties of these entropies, as well as many examples, are recalled. Also new entropy functions are proposed, as well as generalizations of several known definitions and results. Furthermore we give some connections with other topics in Mathematics as Mahler measure and Lehmer Problem from Number Theory, and the growth rate of groups and Milnor Problem from Geometric Group Theory. Most of the results are covered by complete proofs or references to appropriate sources. 
\end{abstract}

\newpage
\tableofcontents

\newpage
\section{Introduction}

The notion of entropy was invented by Clausius in Physics in 1865 and carried over to Information Theory by Shannon in 1948. Kolmogorov \cite{K} and Sinai \cite{Sinai} in 1958 introduced the measure theoretic entropy in Ergodic Theory. Moreover the topological entropy was defined for the first time by Adler, Konheim and McAndrew \cite{AKM} in 1965, and other notions of topological entropy were given by Bowen \cite{B} and Hood \cite{hood}. Finally, entropy was taken also in Algebraic Dynamics by Adler, Konheim and McAndrew \cite{AKM} in 1965 and Weiss \cite{We} in 1974, and then by Peters \cite{P} in 1979.

\smallskip 
In each setting entropy is a non-negative real-valued function $h$ measuring the randomness and disorder (the value $\infty$ is allowed). So the entropy $h(T)$ of a transformation $T:X\to X$ of a space $X$ measures the average uncertainty about where $T$ moves the points of $X$; the size of $h(T)$ reflects the randomness of $T$ and the degree to which $T$ disorganizes the space $X$.
The main specific cases in which entropy is defined are:
\begin{enumerate}[(a)]
\item $X$ probability space, $T$ measure preserving trasformation, $h_{mes}$ measure entropy;
\item $X$ set, $T$ selfmap, $\mathfrak h$ set theoretic entropy;
\item $X$ topological space, $T$ continuous selfmap, $h_{top}$ topological entropy;
\item $X$ metric or uniform space, $T$ uniformly continuous selfmap, $h_U$ uniform entropy;
\item $X$ discrete group, $T$ endomorphism, $h_{alg}$ algebraic entropy;
\item $X$ locally compact group, $T$ continuous endomorphism, $k$ topological entropy, $h_{alg}$ algebraic entropy.
\end{enumerate}

\bigskip
This paper covers the first author's survey talk on entropy given at the International Conference on Topology in Islamabad in July 2011. This explains
the choice to focus our attention mainly on two entropies, namely the topological entropy and the algebraic entropy, considering both of them in the setting of (topological) groups.  Nevertheless, we pay some attention also to the other known entropies and discuss some new ones, providing new results and new connections in many directions. Among the novelties we should mention two set-theoretic entropies (\S \ref{set-sec}), the e-spectra (\S \ref{tes-sec} and \S \ref{aes-sec}) providing connection to the famous Lehmer Problem, and the general setting of the algebraic entropy and the (topological) adjoint entropy, without any limitation on the domain of the endomorphism. 

\medskip
In Section \ref{prelude} we collect a number of known notions and results which do not explicitly contain entropy, but are used in the rest of the paper.
We start recalling the notion of Mahler measure, as well as Lehmer Problem, in \S \ref{Lehmer-sec}. We use this notion in several parts of the paper and we give also some equivalent forms of Lehmer Problem in terms of entropy, as we mention below in the Introduction. 
In \S \ref{bernoulli-sec} we give several known results. In particular we recall the definition of the Bernoulli shifts (left, right and two sided).
As becomes clear in the main body of the paper, these shifts are the leading examples in the theory of entropy. 

\medskip
In Section  \ref{mes-set-sec} we consider two different kinds of entropy. In \S \ref{mes-sec} we recall the notion of measure theoretic entropy in ergodic theory, 
introduced independently by Kolmogorov \cite{K} and Sinai \cite{Sinai} for measure preserving transformations of probability measure spaces, using measurable partitions of the space.
This entropy is not the main subject of our exposition and it is widely studied and understood, so we only
give the definition of this entropy and some of its basic properties (as they gave the inspiration for the concept of topological entropy having similar properties), referring to \cite{Wa,W} for a more detailed treatment. In  \S \ref{set-sec} we introduce two kinds of set-theoretic entropy, namely, the covariant entropy $\mathfrak h$ and the contravariant  entropy  $\mathfrak h^*$ for selfmaps of sets. These new entropies have interesting properties on their own; on the other hand they are strictly related respectively to the topological entropy and to the algebraic entropy, as they permit to compute easily the topological entropy and the algebraic entropy of special endomorphisms (namely, the generalized shifts; see \S \ref{genshift}).

\medskip
Section \ref{top-sec} treats the topological entropy. In \S \ref{akm-sec} we briefly recall the notion of topological entropy, due to Adler, Konheim and McAndrew \cite{AKM}, for continuous selfmaps of compact spaces, defined using counterimages of open coverings of the space. Moreover we state its basic properties, namely, Monotonicity for continuous images, Invariance under conjugation, Continuity on inverse limits, Logarithmic Law, Weak Addition Theorem and a Reduction to surjective continuous selfmaps.
We recall in \S \ref{bowen-sec} the definition due to Bowen \cite{B} of the uniform entropy for uniformly continuous selfmaps of metric spaces, and its generalization to uniformly continuous selfmaps of uniform spaces due to Hood \cite{hood}. In particular, the uniform entropy coincides with the topological entropy when they are both defined. More details on this entropy can be found in the survey paper \cite{DSV}.

Starting from this point, we focus our attention on the realm of topological groups and their continuous endomorphisms. In \S \ref{cg-sec} we state the main properties of the topological entropy for continuous endomorphisms of compact groups; these are Bernoulli normalization, Kolmogorov-Sinai Formula, Yuzvinski Formula, Addition Theorem and Uniqueness Theorem.
In \S \ref{tes-sec} we introduce the topological e-spectrum of a compact abelian group $K$, namely the set $\mathbf E_{top}(K)$ of all possible values of the topological entropy of continuous endomorphisms of $K$. Letting $\mathbf E_{top}= \bigcup \mathbf E_{top}(K)$, the union taken over all compact abelian groups $K$, the question whether the equality $\inf (\mathbf E_{top}\setminus \{0\})= 0$ is equivalent to Lehmer Problem.
In the final \S \ref{lc-sec} of this section dedicated to the topological entropy, we pass to the more general setting of continuous endomorphisms of locally compact groups using Hood's generalization of Bowen's definition of uniform entropy. Not so much is known in this general situation, so we describe it giving some examples and open problems to further improve this knowledge. Moreover a reduction is found for the computation of the uniform entropy of continuous endomorphisms of totally disconnected locally compact groups eliminating the use of the measure.

\medskip
The algebraic entropy is exposed in Section \ref{alg-sec}. The first definition of the algebraic entropy $\ent$ for endomorphisms of abelian groups was sketched briefly in \cite{AKM}. The same notion was studied later by Weiss \cite{We} and more recently by Dikranjan, Goldsmith, Salce and Zanardo \cite{DGSZ}. Since the definition uses trajectories of finite subgroups, it perfectly fits for endomorphisms of torsion abelian groups (or more generally, locally finite groups). More precisely the algebraic entropy $\ent$ of an endomorphism $\f$ of an abelian group coincides with the algebraic entropy $\ent$ of the restriction of $\f$ to the torsion part of the group. Peters \cite{P} introduced another notion of algebraic entropy for automorphisms $\f$ of abelian groups, using trajectories of finite {\em subsets} with respect to the inverse automorphism $\f^{-1}$. As noted in \cite{DG}, this definition can be appropriately modified in order to be extended to arbitrary endomorphisms $\f$ of abelian groups, by using trajectories of finite subsets with respect to $\phi$. The so defined algebraic entropy $h_{alg}$ coincides with $\ent$ on endomorphisms of torsion abelian groups.

An extension of the algebraic entropy $\ent$, called \emph{intrinsic algebraic entropy}, was recently given in \cite{DGSV} replacing the family of all finite subgroups of $G$ by a suitable larger family of subgroups depending on the endomorphism $\phi$ of $G$. Moreover in the algebraic context other entropies were defined for endomorphisms of modules by means of real-valued module invariants in \cite{SZ}; these entropies generalize the algebraic entropy $\ent$ in a direction different from the one of $h_{alg}$ considered here, they were intensively studied in the recent papers \cite{GBS,SVV,SZ,Vi,Z}.

In \S \ref{def-sec} we define the algebraic entropy $h_{alg}$ for endomorphisms of arbitrary groups, removing any hypothesis of commutativity on the groups. It turns out that the basic properties of the algebraic entropy still hold in this general setting, such as Monotonicity for subgroups and quotients, Invariance under conjugation, Logarithmic Law, Bernoulli normalization, Continuity on direct limits, Weak Addition Theorem. In \S \ref{ab-sec} we recall the fundamental properties of the algebraic entropy for endomorphisms of abelian groups: Algebraic Kolmogorov-Sinai Formula, Algebraic Yuzvinski Formula  and Addition Theorem. Then comes a Uniqueness Theorem claiming that five of these properties (Bernoulli normalization, Invariance, Continuity, Algebraic Yuzvinski Formula  and Addition Theorem) uniquely determine the algebraic entropy.  We leave also several open problems related to a possible generalization of the Addition Theorem and the Uniqueness Theorem to the non-abelian case.

In \S \ref{growth-sec} we explain how the algebraic entropy is related to the classical topic of growth rate of finitely generated groups.  This notion was introduced by Milnor \cite{M1} and many authors studied the problem posed by him \cite{M3} in this context; among them we mention Milnor \cite{M2} himself, Wolf \cite{Wolf}, Bass \cite{Bass}, Tits \cite{Tits}, Adyan \cite{Ad}, Gromov \cite{Gro} and Grigorchuk \cite{Gri1,Gri2,Gri3,Gri4}. We point out how the algebraic entropy extends in a natural way the growth rate of groups to the growth rate of group endomorphisms, allowing us to extend Milnor Problem to any group endomorphism. In \S \ref{Growth-sec} we expose the known results on algebraic entropy that solve this more general version of Milnor Problem in the abelian case.

In analogy with the case of topological entropy, in \S \ref{aes-sec} we introduce the algebraic spectrum $\mathbf E_{alg}(G)$ of all possible values of the algebraic entropy of endomorphisms of an abelian group $G$. Letting $\mathbf E_{alg}= \bigcup \mathbf E_{alg}(G)$, the union taken over all abelian groups $G$, we prove that the question whether $\inf (\mathbf E_{alg}\setminus \{0\})= 0$ is equivalent to Lehmer Problem.

Finally \S \ref{alc-sec} concerns algebraic entropy for continuous endomorphisms of locally compact groups. Peters \cite{Pet1} gave a definition of entropy for topological automorphisms of locally compact abelian groups extending the notion of entropy he introduced in the discrete case. This definition was appropriately modified by Virili \cite{V} for endomorphisms of locally compact abelian groups, and  here we avoid the hypothesis on the group to be abelian. So the notion of algebraic entropy can be extended to the general setting of continuous endomorphisms of locally compact groups; namely, the same context as the one considered for the uniform entropy. Only some properties of the algebraic entropy are known in this general context, so we leave some open problems, that indicate a possible way to improve this knowledge. Analogously to the topological case, a measure-free computation of the algebraic entropy of continuous endomorphisms of totally disconnected locally compact groups is proposed. 

\medskip
The adjoint algebraic entropy for endomorphisms of abelian groups, was recently introduced in \cite{DGV} as a counterpart of the algebraic entropy $\ent$, replacing finite subgroups by finite-index subgroups. In Section \ref{adj-sec} we propose a more general setting, removing the hypothesis on the groups to be abelian. This is done in \S \ref{adj-def}, where we verify also that the basic properties established in the abelian case, as Invariance under conjugation, Logarithmic Law, Monotonicity for subgroups and quotients, Weak Addition Theorem, remain valid in the general context as well. In \S \ref{adj-mp} we list the relevant properties of the adjoint algebraic entropy in the abelian case. In particular, we see that the Addition Theorem does not hold in full generality and we recall a fundamental dichotomy result: only the values $0$ and $\infty$ may occur as values of the adjoint algebraic entropy. We show in \S \ref{top-adj} that an appropriate recourse to topology  can avoid this dichotomy. Here we recall the notion of topological adjoint entropy for continuous endomorphisms of topological abelian groups given in \cite{G}; also in this case we remove the hypothesis that the groups have to be abelian.

\medskip
We dedicate Section \ref{conn-sec} to describe the known relations among all the entropies introduced so far. 

In \S \ref{mes-top} we see the precise relation of the topological entropy with the measure entropy when they both make sense, namely, on compact metric spaces endowed with invariant Borel probability measure. In particular they coincide for continuous surjective endomorphisms of compact groups equipped with their Haar measure. 
Then in \S \ref{top-alg} we describe how the topological entropy and the algebraic entropy are related by Pontryagin duality in the context of continuous endomorphisms of locally compact abelian groups. In the compact-discrete case, the topological entropy of a continuous endomorphism of a compact abelian group coincides with the algebraic entropy of the dual endomorphism of the dual discrete abelian group. We see also how this can be generalized for continuous endomorphisms of totally disconnected locally compact abelian groups with totally disconnected dual;  this can be achieved by using the above mentioned  possibility to avoid the use of measure in the computation of the topological entropy and the algebraic entropy. The validity of this result in full generality is an open problem.
In \S \ref{genshift} we give the relations of the covariant set-theoretic entropy with the topological entropy, and of the contravariant set-theoretic entropy with the algebraic entropy of generalized shifts introduced in \cite{AKH}.  It is clear that the topological adjoint entropy coincides with the adjoint algebraic entropy for endomorphisms of discrete groups; on the other hand in \S \ref{LAAAAST} we see that it coincides with the topological entropy for continuous endomorphisms of totally disconnected compact groups.
Using Pontryagin duality we give also the relation of the topological adjoint entropy with the algebraic entropy.

\medskip
The last Section \ref{cav-sec} briefly recalls the categorical approach to entropy from \cite{DG1}. In particular, the notion of entropy function of an abelian category is recalled. This kind of entropy functions cover the case of the algebraic entropy, while another notion of contravariant entropy function is needed to view the topological entropy and the measure entropy in categorical terms.

A different general approach to entropy is discussed in \cite{DGV1}, where entropy is defined in the category $\mathfrak S$ of normed semigroups and then pulled back to a specific category $\mathfrak X$ (of sets, topological or measure spaces, abelian groups, modules, etc.) via appropriate functors $F:\mathfrak X \to \mathfrak S$.

\subsection{Notation and terminology}

We denote by $\mathbb Z$, $\mathbb N$, $\mathbb N_+$, $\Q$ and $\R$ respectively the set of integers, the set of natural numbers (including $0$), the set of positive integers, the set of rationals and the set of reals. For $n\in\N$, let $\R_{\geq n}=\{x\in\R: x\geq n\}$ and $\R_{> n}=\{x\in\R: x> n\}$. For $m\in\mathbb N_+$, we use $\mathbb Z(m)$ for the finite cyclic group of order $m$, and for a prime $p$ we denote by $\mathbb J_p$ the $p$-adic integers and by $\Q_p$ the $p$-adic numbers.

With a slight divergence from the standard use, for a set $X$ we denote by $[X]^{<\omega}$ the set of all non-empty finite subsets of $X$.

Let $G$ be a group. In general, we denote the group $G$ multiplicatively and its identity by $e_G$; on the other hand, we denote the abelian groups additively. If $H$ is a subgroup of $G$, we indicate this by $H\leq G$, and we denote by $[G:H]$ the index of $H$ in $G$. For $F\in[G]^{<\omega}$ and $n\in\N_+$ let 
$F_{(n)}=\underbrace{F\cdot\ldots\cdot F}_n.$
For a set $X$ we denote by $G^X$ and $G^{(X)}$ respectively the direct product and the direct sum of $|X|$ many copies of $G$.
 
The subset of torsion elements of $G$ is $t(G)$ (it is a subgroup when $G$ is abelian), the center of $G$ is $Z(G)$ and the derived subgroup of $G$ is $G'$. Moreover $\End(G)$ is the set of all endomorphisms of $G$ (it is a ring  when $G$ is abelian). 
We denote by $e_G$ and $id_G$ respectively the endomorphism of $G$ which is identically $e_G$ and the identity endomorphism of $G$.

For $G$ an abelian group, $D(G)$ denotes the divisible hull of $G$. 
For a ring $R$, we denote by $R[t]$ the ring of polynomials with coefficients in $R$.

For a category $\mathfrak M$ we write $M\in\mathfrak M$ if $M$ is an object of $\mathfrak M$ and $N\subseteq M$ if $N$ is a subobject of $M$. 
If $\mathfrak M$ is an abelian category and $M\in\mathfrak M$, we denote by $0_M$ the zero morphism of $M$ and by $1_M$ the identity morphism of $M$.
Moreover, for $M_1,M_2\subseteq M$, we denote by $M_1+M_2$ the join of $M_1$ and $M_2$ and by $M_1\cap M_2$ the intersection of $M_1$ and $M_2$. If $N$ is a subobject of $M$, $M/N$ is the quotient object. 

\newpage
\section{Prelude}\label{prelude}

\subsection{Lehmer Problem}\label{Lehmer-sec}

Let $k$ be a positive integer, let $f(t)=st^k+a_{k-1}t^{k-1}+\ldots+a_0\in\mathbb C[t]$ be a non-constant polynomial with complex coefficients and let $\{\lambda_i:i=1,\ldots,k\}\subseteq\mathbb C$ be the set of all roots of $f(t)$ (we always assume the roots of a polynomial to be counted with multiplicity); in particular, $f(t)=s\cdot\prod_{i=1}^n(t-\lambda_i)$. The Mahler measure of $f(t)$ was defined independently by Lehmer \cite{Lehmer} and Mahler \cite{Mahler} in two different equivalent forms.
Following Lehmer \cite{Lehmer} (see also \cite{EWard}), the \emph{Mahler measure} of $f(t)$ is $M(f(t))=|s|\cdot \prod_{|\lambda_i|>1} |\lambda_i|.$ The \emph{(logarithmic) Mahler measure} of $f(t)$ is $${m}(f(t))=\log M(f(t))=\log|s|+\sum_{|\lambda_i|>1}\log|\lambda_i|.$$
For a survey on the Mahler measure of algebraic numbers see \cite{Smyth} (see also \cite{BDM}, \cite{EWard}, \cite{Hi} and \cite{M}).

\medskip
For an algebraic number $\alpha\in\mathbb C$, the \emph{Mahler measure} $m(\alpha)$ of $\alpha$ is the Mahler measure of the minimal polynomial of $\alpha$.
Lehmer \cite{Lehmer}, with the aim of generating large primes, associated to any monic polynomial $f(t)\in \Z[t]$ with roots $\lambda_1, \ldots, \lambda_k$ the sequence of integers $$\Delta_n(f(t))=\prod_{i=1}^k|1-\lambda_i^n|.$$
The idea comes from Mersenne primes generated by the polynomial $f(t)=t-2$. Lehmer was using the polynomial $f(t)= t^3-t-1$. This is the non-reciprocal polynomial with the smallest positive Mahler measure \cite{Smyth}. The polynomial $$g(t)=x^{10}+x^9-x^7-x^6-x^5-x^4-x^3+x+1$$ is the reciprocal polynomial with the smallest known positive Mahler measure, that is, $m(g(t))=\log \lambda$, where $\lambda=1.17628\ldots$ is the Lehmer number \cite{Hi}. Still in \cite{Hi} it is noted that $\lambda$ is the largest real root of $g(t)$ to date, and it is the only one of its algebraic conjugates outside the unit circle (i.e., $\lambda$ is a \emph{Salem number}). If there exists a polynomial $p(t)$ with positive Mahler measure smaller than this, then $\deg p(t)\geq 55$ \cite{MRQ}.

\begin{remark}
Let $f(t)\in\Z[t]\setminus\{0\}$ be monic with roots $\lambda_1,\ldots,\lambda_k$. Then
$$m(f(t))=\lim_{n\to\infty} \frac{\log |\Delta_n(f(t))|}{n};$$ 
so $m(f(t))$ measures the (exponential) growth of the sequence $\{\Delta_n(f(t))\}_{n\in\N}$, as 
$$\lim_{n\to\infty} \frac{\log |\Delta_n(f(t))|}{n}=\lim_{n\to \infty} \frac{\sum_{i=1}^k\log|1-\lambda_i^n|}{n}= \sum_{i=1}^k \lim_{n\to\infty} \frac{\log|1-\lambda_i^n|}{n}=\sum_{|\lambda_i|>1}  \log|\lambda_i|.$$
\end{remark}

The case of zero Mahler measure is completely determined by  the following theorem due to Kronecker:

\begin{theorem}[Kronecker Theorem]\label{Kr}\emph{\cite{Kr}}
Let $f(t)\in\Z[t]$ be a monic polynomial with roots $\lambda_1, \ldots, \lambda_k$ in $\mathbb C$. If $\lambda_1$ is not a root of unity, then $|\lambda_i|>1$ for at least one $i\in\{1,\ldots,k\}$. 
\end{theorem}

Indeed, if $f(t)=st^k+ a_{k-1}t^{k-1}+\ldots+a_0\in\Z[t]\setminus \{0\}$, $\deg f(t)=k$ and all roots $\lambda_1,\ldots,\lambda_k$ of $f(t)$ are roots of unity, then there exists $n\in\N_+$ such that every $\lambda_i$ is a root of $t^n-1$. By Gauss Lemma $f(t)$ divides $t^n-1$ in $\Z[t]$. Therefore $s=1$, hence $m(f(t))=0$.
Suppose now that $m(f(t))=0$; we can assume without loss of generality that $f(t)$ is monic. Moreover all $|\lambda_i|\leq 1$ for every $i\in\{1,\ldots,k\}$. By Theorem \ref{Kr} each $\lambda_i$ is a root of unity.

\begin{corollary}
Let $f(t)\in\Z[t]\setminus\{0\}$ be primitive. Then $m(f(t))=0$ if and only if $f(t)$ is cyclotomic (i.e., all the roots of $f(t)$ are roots of unity).
Consequently, if $\alpha$ is an algebraic integer, then $m(\alpha)=0$ if and only if $\alpha$ is a root of unity.
\end{corollary}

On the other hand the general problem posed by Lehmer in 1933 is still open:

\begin{problem}[Lehmer Problem]\emph{\cite{Lehmer}}\label{L-pb}
Is the lower bound of all positive Mahler measures still positive? 
Or, conversely, given any $\delta>0$, is there any algebraic integer whose Mahler measure is strictly between $0$ and $\delta$? 
\end{problem}

This problem is equivalent replacing ``algebraic integer'' with ``algebraic number''. Moreover this is equivalent to ask whether $\inf\{m(f(t)):f(t)\in\Z[t]\ \text{primitive}\}=0$.
It suffices to consider monic polynomials. Indeed, if $\deg(f(t))=k$ and $f(t)=st^k+a_{k-1}t^{k-1}+\ldots+a_0$ with $|s|>1$, then $m(f(t))\geq\log|s|\geq\log2>0$.

Denote 
\begin{equation}\label{LNumber}
\mathfrak L=\inf\{m(\alpha):\alpha\ \text{algebraic integer}\};
\end{equation}
Lehmer asked whether $\mathfrak L$ is zero or strictly positive.

\subsection{Basic results}\label{bernoulli-sec}

\subsubsection{Pontryagin duality and Bernoulli shifts}

Let $G$ be a topological abelian group. The {\em Pontryagin dual} $\widehat{G}$ of $G$ is the group of all continuous homomorphisms $G\to\mathbb T$, endowed with the compact-open topology. If $\phi:G\to G$ is a continuous endomorphism, its dual endomorphism
\begin{center}
$\widehat\phi:\widehat{G}\to \widehat{G}$ is defined by $\widehat\phi(\chi)=\chi\circ\phi$ for every $\chi\in\widehat{G}$. 
\end{center}
The Pontryagin dual of a locally compact abelian group is locally compact as well, and the Pontryagin dual of a discrete (compact) abelian group is always compact (respectively, discrete) \cite{HR,Pontr}. The map $\omega_G: G\to \widehat{\widehat{G}}$ defined by $\omega_G(g)(\chi)=\chi(g)$ for every $g\in G$ and $\chi\in\widehat{G}$ is a topological isomorphism.

\bigskip
Let $K$ be a non-empty set.
\begin{itemize}
\item[(a)] The \emph{two-sided Bernoulli shift} $\overline{\beta}_K$ of $K^{\Z}$ is defined by 
$$\overline\beta_K((x_n)_{n\in\Z})=(x_{n-1})_{n\in\Z}, \mbox{ for } (x_n)_{n\in\Z}\in K^{\Z}.$$
\item[(b)] The \emph{left Bernoulli shift} ${}_K\beta$ of $K^{\N}$ is defined by 
$$_K\beta(x_0,x_1,x_2,\ldots)=(x_1,x_2,x_3,\ldots).$$
\end{itemize}

In case $K$ is a topological space and the product is equipped with the Tichonoff topology, these Bernoulli shifts are continuous ($\overline\beta_K$ is a homeomorphism).  
If $K$ is a group, then the Bernoulli shifts are group endomorphisms ($\overline\beta_K$ is an automorphism). 
Moreover in this case we can define another shift of $K^\N$ as follows (it is again an endomorphism of $K^\N$).
\begin{itemize}
\item[(c)] The \emph{right  Bernoulli shift} $\beta_K$ of the group $K^\N$ is defined by 
$$\beta_K(x_0,x_1,x_2,\ldots)=(e_K,x_0,x_1,\ldots).$$
\end{itemize}

\smallskip
The left Bernoulli shift $_K\beta$ and the two-sided Bernoulli shift $\overline\beta_K$ are relevant for both ergodic theory and topological dynamics, while the right Bernoulli shift $\beta_K$ restricted to the direct sum $\bigoplus_\N K$, when $K$ is a group, is fundamental for the algebraic entropy. 

We denote the Bernoulli shifts restricted to the direct sum respectively by $\overline\beta_K^\oplus$, $\beta_K^\oplus$ and ${}_K\beta^\oplus$. Applying Pontryagin duality we have the following relations. 

\begin{fact}\label{dualbeta}
For a finite abelian group $K$,
\begin{itemize}
\item[(a)] $\widehat{\beta_K^\oplus}={}_K\beta$,
\item[(b)]$\widehat{{}_K\beta^\oplus}=\beta_K$, and
\item[(c)]$\widehat{\overline\beta_K^\oplus}=\overline\beta_K^{-1}$.
\end{itemize}
\end{fact}

\subsubsection{Mahler measure and endomorphisms}

Let $n\in\N_+$. If $\phi:\Z^n\to \Z^n$ is an endomorphism, then $\phi$ is $\Z$-linear. Then the action of $\phi$ on $\Z^n$ is given by an $n\times n$ matrix $A_\phi$ with integer coefficients. We call characteristic polynomial $p_\phi(t)$ of $\phi$ over $\Z$, the (primitive) characteristic polynomial of $A_\phi$ over $\Z$.
The same occurs replacing $\Z$ with $\Q$. Indeed, if $\phi:\Q^n\to \Q^n$ is an endomorphism, then $\phi$ is $\Q$-linear. Then the action of $\phi$ on $\Q^n$ is given by an $n\times n$ matrix $A_\phi$ with rational coefficients. We call characteristic polynomial $p_\phi(t)$ of $\phi$ over $\Z$, the (primitive) characteristic polynomial of $A_\phi$ over $\Z$.

Taking the dual situations with respect to the two above, we see that they are very similar. In detail, let $\psi:\T^n\to\T^n$ be a continuous endomorphism.  Then $\phi=\widehat\psi:\Z^n\to \Z^n$ is an endomorphism, with its matrix $A_\phi$. It is possible to see that the action of $\psi$ on $\T^n$ is given by the transposed matrix ${}^tA_\phi$, that we call $A_\psi$. The characteristic polynomial $p_\psi(t)$ of $\psi$ is the characteristic polynomial of $A_\psi$, which coincides with the characteristic polynomial $p_\phi(t)$ of $A_\phi$.
Let now $\psi:\widehat\Q^n\to\widehat \Q^n$ be a continuous endomorphism.  Then $\phi=\widehat\psi:\Q^n\to \Q^n$ is an endomorphism, with its matrix $A_\phi$. It is possible to see that the action of $\psi$ on $\widehat\Q^n$ is given by the transposed matrix ${}^tA_\phi$, that we call $A_\psi$. Now, the characteristic polynomial $p_\psi(t)$ of $\psi$ is the characteristic polynomial of $A_\psi$, which coincides with the characteristic polynomial $p_\phi(t)$ of $A_\phi$.

\begin{definition}
For an endomorphism $\phi:\Q^n\to \Q^n$ (respectively, for a continuous endomorphism $\psi:\widehat\Q^n\to\widehat\Q^n$)
we call \emph{Mahler measure} of $\phi$ (respectively, $\psi$) the value $m(\phi)=m(p_\phi(t))$ (respectively, $m(\psi)=m(p_\psi(t))$).
\end{definition}

Let $\alpha\in\mathbb C$ be an algebraic number. Then the finite field extension $\Q(\alpha)$ of $\Q$ is isomorphic to $\Q^n$ as an abelian group, where $n\in\N_+$ is the degree of the extension. Moreover the endomorphism $\varphi_\alpha:\Q(\alpha)\to\Q(\alpha)$, given by the multiplication by $\alpha$, is associated to the companion matrix with respect to the base $(1,\alpha,\ldots,\alpha^{n})$. Then the characteristic polynomial of $\varphi_\alpha$ is precisely the minimum polynomial of $\alpha$ over $\Q$. Hence $m(\varphi_\alpha)=m(\alpha)$, that is, the Mahler measure of the endomorphism $\varphi_\alpha$ is precisely the Mahler measure of $\alpha$.

Let now $\beta\in\mathbb C$ be a transcendent number. Then the simple ring extension $\Q[\beta]$ (isomorphic to the polynomial ring $\Q[t]$) does not coincide with the field $\Q(\beta)$; 
$\Q[\beta]$ is an infinite dimensional vector space over $\Q$, so $\Q[\beta]$ is isomorphic to $\Q^{(\N)}$ as an abelian group. Take $(1,\beta,\ldots,\beta^n,\ldots)$ as a basis of $\Q[\beta]$. The endomorphism $\varphi_\beta$ of $\Q[\beta]$, given by the multiplication by $\beta$, is precisely (conjugated to) the right Bernoulli shift of $\Q^{(\N)}$. To get the two-sided Bernoulli shift one can take the ring $\Q[\beta, \beta^{-1}]$ and identify it with $\Q^{(\Z)}$ as a $\Q$-linear space, using as a base the sequence $(\ldots, \beta^{-n}, \ldots, \beta^{-1}, 1, \beta, \ldots, \beta^{n}, \ldots)$. Under this isomorphism,  the two-sided Bernoulli shift of $\Q^{(\N)}$ is conjugated to the automorphism obtained by the multiplication by $\beta$ in $\Q[\beta, \beta^{-1}]$. 
\subsubsection{Flows}

Consider an arbitrary category $\mathfrak X$.

\begin{definition}
A \emph{flow} in ${\mathfrak X}$ is a pair $(X,\phi)$, where $X$ is an object in $\mathfrak X$ and $\phi: X \to X$ an endomorphism in ${\mathfrak X}$. 
\end{definition} 

The category $\af_{\mathfrak X}$ of flows of $\mathfrak X$ has as objects all flows in ${\mathfrak X}$, and a morphism in $\af_{\mathfrak X}$ between two flows $(X,\phi)$ and $(Y,\psi)$ is a morphism $u: X \to Y$ in ${\mathfrak X}$ such that the diagram
\begin{equation}\label{casc-mor}
\begin{CD}
X @>\phi>> X\\ 
@V u VV  @VV u V\\ 
Y @>>\psi> Y
\end{CD}
\end{equation}
in ${\mathfrak X}$ commutes. Two flows $(X,\phi)$ and $(Y,\psi)$ are isomorphic in $\af_{\mathfrak X}$ if the morphism $u: X \to Y$ in \eqref{casc-mor}
is an isomorphism in ${\mathfrak X}$.

\medskip
Let $\mathbf{Mes}$, $\mathbf{Set}$, $\mathbf{Comp}$, $\mathbf{Grp}$ be respectively the categories of probability spaces and measure preserving transformations, sets and maps, compact spaces and continuous maps, groups and homomorphisms.

\medskip
The known entropies are defined on the endomorphisms $\mathrm{Hom}(X,X)$ of a fixed object $X$ in a specific category $\mathfrak X$. 
So an entropy function can be considered as a map $h:\Flow_\mathfrak X\to \R_{\geq 0}\cup\{\infty\}$, defined on flows $(X,\phi)$. If the domain $X$ of $\phi$ is clear from the context, we can denote the flow $(X,\phi)$ only by $\phi$.

\newpage
\section{Measure theoretic entropy and set-theoretic entropy} \label{mes-set-sec}

\subsection{Measure theoretic entropy} \label{mes-sec}

We recall that a \emph{measure space} is a triple $(X,{\mathfrak B},\mu)$, where $X$ is a set, $\mathfrak B$ is a $\sigma$-algebra over $X$ (the elements od $\mathfrak B$ are called \emph{measurable sets}) and $\mu:\mathfrak B\to \R_{\geq0}\cup\{\infty\}$ is a measure. Moreover $(X,{\mathfrak B},\mu)$ is called \emph{probability space} if $\mu(X)=1$. When there is no possibility of confusion we denote $(X,{\mathfrak B},\mu)$ simply by $X$.
The selfmap $\psi:X\to X$ is a \emph{measure preserving trasformation} if $\mu(\psi^{-1}(B))=\mu(B)$ for every $B\in\mathfrak B$.
 
\smallskip
For a measure space $(X,{\mathfrak B},\mu)$ and a measurable partition $\xi=\{A_i:i=1,\ldots,k\}$ of $X$  (i.e., the partition $\xi$ is contained in $\mathfrak B$), define the \emph{entropy} of $\xi$ by $$H(\xi)=-\sum_{i=1}^k \mu(A_k)\log \mu(A_k).$$

For two partitions $\xi, \eta$ of $X$, let $\xi \vee \eta=\{U\cap V: U\in \xi, V\in \eta\}$ and define $\xi_1\vee \xi_2\vee \ldots \vee \xi_n$ analogously for partitions $\xi_1,\xi_2,\ldots,\xi_n$ of $X$.
For a measure preserving transformation $\psi:X\to X$ and a measurable partition $\xi=\{A_i:i=1,\ldots,k\}$ of $X$, let $\psi^{-j}(\xi)=$ $\{\psi^{-j}(A_i):i=1,\ldots,k\}$. The sequence $\left\{\frac{H(\bigvee_{j=0}^{n-1}\psi^{-j}(\xi))}{n}\right\}_{n\in\N_+}$ is decreasing.

\medskip
Kolmogorov \cite{K} and Sinai \cite{Sinai} gave the following

\begin{definition}
Let $X$ be a measure space and $\psi:X\to X$ a measure preserving transformation.
\begin{itemize}
\item[(a)] For a measurable partition $\xi$ of $X$, the \emph{measure entropy of $\psi$ with respect to $\xi$} is 
$$H_{mes}(\psi,\xi)=\lim_{n\to\infty}\frac{H(\bigvee_{j=0}^{n-1}\psi^{-j}(\xi))}{n}.$$
\item[(b)] The {\em measure entropy} of $\psi$ is $$h_{mes}(\psi)=\sup\{ H_{mes}(\psi,\xi): \xi\ \text{measurable partition of}\ X\}.$$
\end{itemize}
\end{definition} 
  
Sometimes, when the measure $\mu$ of the measure space $X$ is explicitly given and we need to emphasize the specific measure $\mu$, we write also $h_\mu$ in place of $h_{mes}$. 

\begin{example}
For $k\in \N_+$ and $\varphi_k:\T\to \T$ given by the multiplication by $k$, $h_{mes}(\varphi_k)=\log k$. This equality can be deduced from Kolmogorov-Sinai Formula \ref{KSF} and Theorem \ref{A-S} below.
\end{example}

\begin{example}[Coin tossing space]\cite{W}
For a fixed $n\in\N_+$ and for every $i\in \Z$ let $X_i=\Z(n)=\{0,1,\ldots,n-1\}$ and $X=\prod_{i\in\Z}X_i$.
\begin{itemize}
\item[(a)] Then $X$ is a compact topological space with basic open (cylindric) sets $$C=\prod_{i=-\infty}^{n-1}X_i \times\prod_{i=n}^{m} Y_i \times\prod_{i=m+1}^{\infty}X_i$$ where $n<m$ in $\Z$ and $Y_i$ is a non-empty subset of $X_i$ for every $i=n,n+1, \ldots, m$.
\item[(b)] Let $i\in\Z$ and $\mu_i=(p_0,p_1,\ldots,p_{n-1})$ be a fixed probability vector, that is, $p_i\geq 0$ for every $i=0,1,\ldots,n-1$ and $\sum_{i=0}^{n-1} p_i=1$. Then $\mu_i$ is a measure on $X_i$. Hence a Borel measure $\mu$ on $X$ is defined by $\mu(C)=\prod_{i=n}^{m}\mu_i(Y_i)$ \cite[Theorem 4.26]{Wa}. 
\item[(c)] The two-sided Bernoulli shift $\overline\beta_{\Z(n)}:X\to X$ is an invertible measure preserving transformation of $(X,\mu)$. Then $h_{mes}(\overline\beta_{\Z(n)})=-\sum_{i=0}^{n-1}p_i\log p_i$.
\item[(d)] (Bernoulli normalization) Taking in item (c) $p_i=1/n$ for every $i=0,1,\ldots,n-1$, 
one obtains the Haar measure on $X$. The equality in (c) gives $h_{mes}(\overline\beta_{\Z(n)})=h_{mes}({}_{\Z(n)}\beta)=\log n$.
\end{itemize}
\end{example}

This measure entropy is a conjugacy invariant \cite[Theorem 4.11]{Wa}, we recall below some other basic properties of this 
entropy. 

\begin{lemma}[Monotonicity for factors]
Let $X$ be a measure space and $\psi:X\to X$ a measure preserving transformation. If $\overline\psi$ is a factor of $\psi$, then $h_{mes}(\overline\psi)\leq h_{mes}(\psi)$.
\end{lemma}

\begin{theorem}[Logarithmic Law]
Let $X$ be a probability space and $\psi:X\to X$ a measure preserving transformation. Then $h_{mes}(\psi)=k h_{mes}(\psi^k)$ for every $k\in\N_+$. If $\psi$ is invertible, then $h_{mes}(\psi^k)=|k| h_{mes}(\psi)$ for every $k\in\Z$.
\end{theorem}

\begin{theorem}[Weak Addition Theorem]
Let $X_1,X_2$ be probability spaces and $\psi_i:X_i\to X_i$ a measure preserving transformation for $i=1,2$. Then $h_{mes}(\psi_1\times \psi_2)=h_{mes}(\psi_1)+h_{mes}(\psi_2)$.
\end{theorem}

An important construction from the point of view of the measure entropy is the following. It was introduced by Pinsker for the first time.
For a measure preserving  transformation $\phi$ of a measure space $(X, {\mathcal B}, \mu)$ the {\em Pinsker $\sigma$-algebra} $\mathfrak P(\phi)$ of $\phi$ is the greatest $\sigma$-subalgebra of ${\mathcal B}$ such that $\phi$ restricted to $(X, \mathfrak P(\phi),\mu\restriction_{{\mathcal B}})$ has zero measure entropy. 

Note that $id_X: (X, {\mathfrak B},\mu) \to (X, {\mathfrak P}(\phi),\mu\restriction_{{\mathcal B}})$ is measure preserving, so $(X, {\mathfrak P}(\phi),\mu\restriction_\mathcal B)$ is a factor of $(X, {\mathfrak B},\mu)$ (see \cite{Wa}).

\medskip
See Walters's book \cite{Wa} for a complete treatment of the measure entropy.

\subsection{Set-theoretic entropies}\label{set-sec}
 
We recall first some natural concepts related to selfmaps of sets. Let $X$ be a set and $\lambda: X \to X$ a selfmap. For $Y\subseteq X$ we say that 
$Y\subseteq X$ is {\em $\lambda$-invariant} (respectively, {\em inversely $\lambda$-invariant}) if $\lambda(Y) \subseteq Y$ (respectively, $\lambda^{-1}(Y) \subseteq Y$). 

\begin{definition}
Let $X$ be a set and  $\lambda: X \to X$ a selfmap. Call $x\in X$
 \begin{itemize}  
  \item[(a)]  {\em periodic} if $\lambda^n(x) = x$ for some $n\in\N_+$;
  \item[(b)]  {\em quasi-periodic} if $\lambda^n(x) = \lambda^m(x)$ for some $m\ne n$ in $\N$;
  \item[(c)] {\em wandering} if $x$ is not quasi-periodic.  
\end{itemize} 
 \end{definition}  

Clearly a point $x\in X$ is wandering  precisely when all $\lambda^n(x)$ with $n\in\N$ are pairwise distinct.  Let 
\begin{itemize}
  \item[(a)] $\mathrm{Per}(\lambda,X) =\{x\in X: x \mbox{ is a periodic point of }\lambda\}$;
  \item[(b)] $\mathrm{QPer}(\lambda,X) =\{x\in X: x \mbox{ is a quasi-periodic point of } \lambda\}$;
  \item[(c)] $\mathrm{Wan}(\lambda,X) = \{ x \in X: x\mbox{ is a wandering point of }\lambda\}$;
  \item[(d)] $\mathrm{Fix}(\lambda,X)=\{x\in X: \lambda(x) = x\}$ are the fixed points of $\lambda$.
\end{itemize}
Then: 
 \begin{itemize}  
  \item[(a)] $\mathrm{QPer}(\lambda,X) = \bigcup_{n\in\N} \lambda^{-n}(\mathrm{Per}(\lambda,X))$; 
  \item[(b)] $\mathrm{QPer}(\lambda,X)$ and $\mathrm{Wan}(\lambda,X)$ are both $\lambda$-invariant and inversely $\lambda$-invariant, while $\mathrm{Per}(\lambda,X)$ is only $\lambda$-invariant;  
  \item[(c)] $X = \mathrm{QPer}(\lambda,X) \cup \mathrm{Wan}(\lambda,X)$ is a partition of $X$.
\end{itemize} 
  
\begin{remark}\label{Rem} 
Let $X$ be a set and $\lambda:X\to X$ a selfmap. For $x,y\in X$  let $x\sim y$ when there exist $n,m \in \N$ such that $\lambda^n(x) = \lambda^m(y)$. Then $\sim$ is an equivalence relation on $X$. Let 
\begin{equation}\label{(P)}
X = \bigcup_{j\in I} E_j 
\end{equation}
be the partition of $X$ related to $\sim$. Then each class $E_j$ is both $\lambda$-invariant and inversely $\lambda$-invariant. 
 Let $Y= X/\sim$ and let $q: X \to Y$ be the quotient map associated to the partition \eqref{(P)}.  
\end{remark}

We order $X$ by letting $x \leq_\lambda y$ (we write simply $x\leq y$ when there is no danger of confusion) if $\lambda^n(y) = x$ for some $n \in \N$. Then $\leq_\lambda $ is a preorder on $X$, which is an order precisely when $\lambda$ has no periodic points beyond the fixed points. The antichains of $(X, \leq_\lambda )$ are called also \emph{$\lambda$-antichains} in case we need to emphasize the dependence on $\lambda$.

\begin{definition} 
Let $X$ be a set, $\lambda:X\to X$ a finitely many-to-one surjective selfmap and $E$ a subset of $X$.  Let 
$$\uparrow\! E=\{y\in X: (\exists x \in E)\ x \leq  y\}= \bigcup_{n\in\N}\lambda^{-n}(E),$$
and let $\F(E)=E\setminus\bigcup_{n\in\N_+}\lambda^{-n}(E)$ be \emph{the set of minimal points of $E$}.
\end{definition}

Clearly, $\uparrow\! E$ is inversely $\lambda$-invariant and $E\subseteq \uparrow\! E$. 

From a dynamical point of view, $x\in E$ is minimal precisely when it is  a \emph{non-recurrent point of $E$} (i.e.,  $\lambda^n(x)\in E$ for no $n\in\N_+$), so $\F(E)$ is the set of non-recurrent points of $E$. The set $E$ is a $\lambda$-antichain if $E=\F(E)$.   

\medskip
We recall now the following definition of subadditive sequence, since in general the sequences in the definitions of the various entropies turn out to have this property.

\begin{definition}\label{subadditive}
A sequence $\{a_n\}_{n\in\N}$ of non negative real numbers is {\em subadditive} if  $a_{n+m}\leq a_n  +a_m$ for every $n,m\in \N$. 
\end{definition}

The following known result is applied several times in the paper, to prove the existence of the limits defining the various entropies.

\begin{lemma}[Fekete Lemma]\label{fekete}\emph{\cite{Fek}, \cite[Theorem 4.9]{Wa}}
For a subadditive sequence $\{a_n\}_{n\in\N}$ of non negative real numbers, the sequence $\{\frac{a_n}{n}\}_{n\in\N}$ converges and $$\lim_{n\to\infty}\frac{a_n}{n}=\inf_{n\in\N}\frac{a_n}{n}.$$
\end{lemma}

\subsubsection{Covariant set-theoretic entropy}\label{stent-sec}

Let $X$ be a set and $\lambda: X \to X$  a selfmap. For a finite subset $D$ of $X$ and $n\in\N_+$ the {\em $n$-th $\lambda$-trajectory of $D$} is 
$$\mathfrak T_n(\lambda,D) = D\cup\lambda(D)\cup\cdots\cup\lambda^{n-1}(D),$$ 
while the \emph{$\lambda$-trajectory of $D$} (often named also \emph{(positive) orbit of $D$ under} $\lambda$ by some authors) is
$$\mathfrak T(\lambda,D)= \bigcup_{n\in\N}\lambda^n(D) = \bigcup_{n\in\N_+} \mathfrak  T_n(\lambda,D).$$ 

\smallskip
It is easy to see that $\mathfrak T(\lambda, \{x\}) $ is finite for a point $x\in X$ precisely when $x$ is quasi-periodic (i.e., $\mathfrak T(\lambda,\{x\}$ is infinite if and only if $x$ is wandering). 

\medskip
For the proof of the the first item of the following definition Fekete Lemma \ref{fekete} applies.

\begin{definition}  
Let $X$ be a set and $\lambda: X \to X$ a selfmap.
\begin{itemize}  
\item[(a)] For every finite subset $D$ of $X$ the limit 
$$\mathfrak h (\lambda, D)=\lim_{n\to\infty} \frac{|\mathfrak T_n(\lambda,D) |}{n}$$
exists and satisfies $\mathfrak h(\lambda, D)\leq |D|$ (see \cite[Lemma 2.5]{FD}). It is the {\em set-theoretic entropy} of $\lambda$ with respect to $D$. 
  \item[(b)]  {\rm \cite{FD}} The number $$\mathfrak h (\lambda) = \sup\left\{ \mathfrak h (\lambda, D) :D \in[X]^{<\omega}\right\}$$ is the {\em set-theoretic entropy} of $\lambda$. 
\end{itemize}
\end{definition}

\medskip
Obviously the correspondence $[X]^{<\omega}\to \N\cup\{\infty\}$ given by  $D \mapsto  \mathfrak h (\lambda, D)$ is monotone increasing and $\mathfrak h (\lambda, D_1 \cup D_2)\leq \mathfrak h (\lambda, D_1) + \mathfrak h (\lambda, D_2)$ for every $D_1,D_2\in[X]^{<\omega}$. 
 
We see in the sequel that $\mathfrak h (\lambda, D)$ is always a non-negative integer and we have seen that it is bounded by $|D|$. 

\medskip
The next example is the ``backbone" of set-theoretic entropy: 

\begin{example}[Shift normalization]  \label{Ber_CovarSet} 
Define the \emph{right shift} $\rho: \N \to \N$ by $\rho(n) = n+1$. 
Obviously, ${\mathfrak h}(\rho,\{0\}) = {\mathfrak h}(\rho,\{0,1,\ldots , n\}) = 1$ for every $n\in \N$. Since every finite subset $D$ of $\N$ is contained in $\{0,1,\ldots , n\}$ for some $n\in\N$,  this proves that  $ {\mathfrak h}(\rho) = 1$. 
\end{example}

\begin{example}\label{example1}  
Let $X$ be a set and $\lambda: X \to X$ a selfmap.
\begin{itemize}  
 \item[(a)] Then $\mathfrak h (\lambda, \{x\}) > 0$ for $x\in X$ if and only if $x$ is a wandering point (as $\mathfrak T(\lambda, \{x\}) $ is infinite precisely when $x$ is wandering); in such a case $ \mathfrak h (\lambda, \{x\}) = 1$. 
 \item[(b)] If $D$ is a finite subset of $X$, then:
  \begin{itemize}  
  \item[(b$_1$)] $ \mathfrak h(\lambda, D)>0$ if and only if $D$ contains a wandering point $x$; in such a case  $ \mathfrak h(\lambda, D)\geq\mathfrak h(\lambda, \{x\})=1$;
  \item[(b$_2$)]$\mathfrak h(\lambda,D)=  \mathfrak h(\lambda,\lambda(D))$. 
\end{itemize} 
\end{itemize}
\end{example}
  
Next we list the properties of the set-theoretic entropy proved in \cite{FD}. 

\begin{lemma}\label{Properies0}   Let $X$ be a set and $\lambda: X \to X$ a selfmap. 
\begin{itemize}  
 \item[(a)] \emph{(Monotonicity for subsets)} If $Y$ is a $\lambda$-invariant subset of $X$, then $\mathfrak h(\lambda) \geq \mathfrak h(\lambda\restriction_Y)$. 

If $Y$ is an inversely $\lambda$-invariant subset of $X$, then $\mathfrak h(\lambda) \geq \mathfrak h(\lambda\restriction_{X\setminus Y})$. 
 \item[(b)] \emph{(Monotonicity for factors)} If $\psi:Y\to Y$ a selfmap and $\alpha:X\to Y$ is a surjective map such that $\alpha \circ \lambda = \psi \circ \alpha$, then $\mathfrak h(\lambda) \geq \mathfrak h (\psi)$. 
\end{itemize}
\end{lemma}

The next set of properties of  the set-theoretic entropy are separated in the following theorem since they completely characterize this entropy. 

\begin{theorem}\label{Properies}  \emph{\cite{FD}}
Let $X$ be a set and $\lambda: X \to X$ a selfmap. 
\begin{itemize}  
\item[(a)] \emph{(Invariance under conjugation)} If $Y$ is another set and $\alpha:X\to Y$ is a bijection, then ${\mathfrak h}(\lambda)={\mathfrak h}(\alpha\circ\lambda\circ\alpha^{-1})$.
\item[(b)] \emph{(Continuity)} If $\{Y_i: i \in I\}$ is a directed system of $\lambda$-invariant subsets of $X$ with $ X = \bigcup_{i\in I} Y_i$, then $\mathfrak h(\lambda) = \sup_{i\in I} \mathfrak h(\lambda \restriction_{Y_i})$.  
\item[(c)] \emph{(Addition Theorem)} If $X=Y_1\cup Y_2$ with $Y_1, Y_2\subseteq X$ disjoint and $\lambda$-invariant, then $\mathfrak h(\lambda)=\mathfrak h(\lambda\restriction_{Y_1})+\mathfrak h(\lambda\restriction_{Y_2})$.
\item[(d)] \emph{(Reduction to the image)} For the restriction $\lambda\restriction_{\lambda(X)}: \lambda(X) \to \lambda(X)$, $\mathfrak h(\lambda\restriction_{\lambda(X)})=\mathfrak h(\lambda)$.
\item[(e)] \emph{(Logarithmic Law)}
\begin{itemize}
\item[(e$_1$)]   $\mathfrak h (\lambda^k) = k \mathfrak h(\lambda)$ for all $k \in \N$; 
\item[(e$_2$)]    if $\lambda$ is a bijection, then $\mathfrak h(\lambda^{-1}) = \mathfrak h(\lambda)$, so $\mathfrak h(\lambda^k) = |k| \mathfrak h(\lambda)$ for all $k \in \Z$. 
\end{itemize}
\end{itemize}
\emph{(Uniqueness)} Moreover, if $h:\lambda\mapsto h(\lambda)$ is a function defined on selfmaps $\lambda: X \to X$ with values in $\R_{\geq0} \cup \{\infty\}$ satisfying (a)--(e) and $h(\rho)=1$, then $h$ coincides with $\mathfrak h$.  
\end{theorem}

In other words the Shift normalization $\mathfrak h(\rho) = 1$ and the five basic properties stated in Theorem \ref{Properies} completely determine $\mathfrak h$. In Remark \ref{Rem3Feb} we discuss the independence of these properties. 

One can apply item (c) of the above theorem to compute the set-theoretic entropy by considering two types of selfmaps $\lambda: X \to X$:  
\begin{itemize}  
\item[(a)] maps without quasi-periodic points, that is, maps for which every point is wandering (i.e., $X = W(\lambda,X)$);
\item[(b)] maps for which every point is quasi-periodic (i.e., $X = Q(\lambda,X)$).
\end{itemize}

In the next lemma we see in particular that, for a selfmap $\lambda$ of a set $X$, the subset $Q(\lambda,X)$ is the \emph{Pinsker subset} of $X$ with respect to $\lambda$, namely, the greatest $\lambda$-invariant subset $P$ of $X$ such that $\mathfrak h(\lambda\restriction_P)=0$.

\begin{lemma}\label{setPinsker}
Let $X$ be a set and $\lambda: X \to X$ a selfmap. Then $\mathfrak h(\lambda)=0$ if and only if $X=\mathrm{QPer}(\lambda,X)$.
In particular, $\mathrm{QPer}(\lambda,X)$ is the greatest $\lambda$-invariant subset of $X$ such that $\mathfrak h(\lambda\restriction_{\mathrm{QPer}(\lambda,X)})=0$.
\end{lemma}
\begin{proof} 
If $\mathrm{QPer}(\lambda,X)=X$, for every non empty finite subset $D$ of $X$, the trajectory $\mathfrak  T(\lambda,D)$ of $D$ under $\lambda$ is finite. Hence $\mathfrak h(\lambda)=0$. If $x\in X\setminus \mathrm{QPer}(\lambda,X)$, then $\mathfrak h(\lambda,\{x\})=1$. In particular, $\mathfrak h(\lambda)>0$.
\end{proof}

From now on we can concentrate on maps without quasi-periodic points, i.e., of type (a).  When $X = \mathrm{Wan}(\lambda,X)$,  the partition \eqref{(P)} related to $\sim$ has all equivalence classes $E_j$ infinite and both $\lambda$-invariant and inversely $\lambda$-invariant. Hence $\mathfrak h(\lambda\restriction_{E_j})\geq 1$ for each $j\in I$ by Example \ref{example1}, and an application of (c) and (d) of Theorem \ref{Properies} gives 
\begin{equation}\label{(88)}
\mathfrak h(\lambda) = \sum_{j\in I}\mathfrak h(\lambda\restriction_{E_j}),
\end{equation}
where the sum is intended to be $\infty$ when the sum has infinitely many terms. 
Hence \eqref{(88)} gives $\mathfrak h(\lambda) = |I|$ in case $I$ is finite, otherwise $\mathfrak h(\lambda) = \infty$ in view of Example \ref{example1}. We can resume this in the following: 

\begin{corollary}\label{Coro2}
Let $X$ be a set and $\lambda:X\to X$ a selfmap. Then $\mathfrak h(\lambda)$ coincides with the number of equivalence classes in \eqref{(P)} that meet $W(\lambda,X)$, i.e.,  $\mathfrak h(\lambda)= |q(W(\lambda,X))|$. 
\end{corollary}

To conclude, note that a class $E_j$ of the partition \eqref{(P)} meets $\mathrm{Wan}(\lambda,X)$ if and only if it is contained in $\mathrm{Wan}(\lambda,X)$, so Corollary \ref{Coro2} can be stated also as follows:  {\rm $\mathfrak h(\lambda)$ coincides with the number of equivalence classes $E_j$ contained in $\mathrm{Wan}(\lambda,X)$.}

\medskip
Since the partition \eqref{(P)} for a bijection $\lambda$ coincides with the relative partition for $\lambda^{-1}$ and $\mathrm{Wan}(\lambda^{-1},X)=\mathrm{Wan}(\lambda,X)$, from Corollary \ref{Coro2} one can easily obtain: 

\begin{corollary}\label{Coro3}
Let $X$ be a set and $\lambda:X\to X$ a bijection. Then $\mathfrak h(\lambda^{-1})= \mathfrak h(\lambda)$. Moreover $\mathfrak h(\lambda^k)=|k|\mathfrak h(\lambda)$ for every $k\in\Z$.
\end{corollary}

\begin{remark}\label{Rem3Feb}
\begin{itemize}
\item[(a)] It is easy to see that the axioms (a)--(d) in the Uniqueness Theorem \ref{Properies} alone imply also Monotonicity for (invariant) subsets. 
\item[(b)] Let us see that they do not imply
\begin{itemize}  
\item[(e$^*$)] (Monotonicity for factors)
\end{itemize}
Indeed, as one can easily see from the proof of Corollary \ref{Coro2} (or the Uniqueness Theorem \ref{Properies} in \cite{FD}), an entropy function
$h$ satisfying (a)--(d) is completely determined by its values on the periodic selfmaps $\sigma_n: X_n\to X_n$, where $n\in\N_+$, 
$X_n =\{1, \ldots, n -1\}$ and $\sigma_n$ is the cycle $(12\dots n)$ in the permutation group of $X_n$. Let $a_n=h(\sigma_n)$ for every $n\in\N_+$. 
Then the sequence $\{a_n\}_{n\in\N_+}$ determines $h$ in the obvious way. Since each $\sigma_n$ is a factor of the right shift $\rho$ and $\sigma_m$ is a factor of $\sigma_n$ if and only if $n|m$ in $\N$, (e$^*$) becomes equivalent to 
\begin{itemize}  
\item[(e$^{**}$)] $a_n \leq 1$ for all $n\in \N_+$ and the assignment $n \mapsto a_n$ is monotone, where $\N_+$ is ordered with respect to the relation $n|m$. 
\end{itemize}
Hence one can easily violate (e$^{**}$) in many ways keeping the axioms (a)--(d). 
\item[(c)] Obviously, the Logarithmic Law (e) implies (e$^{**}$), since the it gives $2a_n = a_n$, so $a_n = 0$ for every $n\in \N_+$.  
\item[(d)] The next natural question is whether we can replace (e) by the weaker axiom (e$^{*}$) in the Uniqueness Theorem \ref{Properies}, that is, whether (a), (b), (c), (d) and (e$^{*}$) imply (e) (i.e., uniqueness). 
An easy negative example can be obtained by taking $h(\lambda)$ to be the total number of orbits of $\lambda$ (i.e., $a_n= 1$ for all $n\in \N_+$). Generic counterexamples can be obtained by picking the sequence $\{a_n\}_{n\in\N_+}$ so that (e$^{**}$) holds true. 
\end{itemize}
\end{remark}  

Now we recall a pair of notions introduced in \cite{G0} that allow us to give another form of the set-theoretic entropy. 

\begin{definition}{\cite{G0}} 
Let $X$ be a set and $\lambda: X \to X$ a selfmap.
\begin{itemize}  
  \item[(a)] An {\em infinite orbit} of  $\lambda$ in $X$ is an infinite $\lambda$-orbit $O_\lambda(x)=\{\lambda^n(x): n\in \N\}$ of a point $x\in X$.  Note that an infinite orbit in $X$ is nothing else but the trajectory $\mathfrak T(\lambda,\{x\})$ with respect to $\lambda$ of a wandering point of $\lambda$ in $X$. 
  \item[(b)]  The maximum number $\mathfrak o(\lambda)$ of pairwise disjoint infinite orbits of $\lambda$ is called  {\em infinite orbit number} of $\lambda$.
\end{itemize}
\end{definition}  

Obviously, $\mathfrak o(\lambda)$ coincides with the size of a maximum collection of pairwise disjoint trajectories of wandering points of $\lambda$ in $X$ (if such a cardinality is infinite, we just take $\infty$). Therefore $\mathfrak o(\lambda)$ tells us, roughly speaking, how many pairwise disjoint copies of the right shift $\rho$ we can detect within the dynamical system $\lambda: X \to X$. 
The following theorem can be deduced from Corollary \ref{Coro2} or from the final part (i.e., Uniqueness) of Theorem \ref{Properies}.

\begin{theorem}\label{setE}\emph{\cite{FD}} 
Let $X$ be a set and $\lambda: X \to X$ a selfmap. Then $\mathfrak h (\lambda)=\mathfrak o(\lambda)$.
\end{theorem}

\subsubsection{Contravariant set-theoretic entropy}\label{cstent-sec}

In analogy with the covariant set-theoretic entropy described in the previous subsection, we introduce here another notion of entropy for selfmaps, using counterimages in place of images.

\medskip
Let $X$ be a set and $\lambda: X \to X$ a surjective selfmap of $X$. For a finite subset $D$ of $X$ and $n\in\N_+$, the \emph{$n$-th $\lambda$-cotrajectory of $D$} is
$$\mathfrak T_n^*(\lambda,D) = D\cup\lambda^{-1}(D)\cup\ldots\cup\lambda^{-n+1}(D),$$
while the \emph{$\lambda$-cotrajectoy of $D$} is
$$\mathfrak T^*(\lambda,D)=\bigcup_{n\in\N}\lambda^{-n}(D)=\bigcup_{n\in\N_+}\mathfrak T_n^*(\lambda,D).$$ 
Clearly, $\mathfrak T^*(\lambda,D)=\uparrow\! D$.

\medskip
In the next remark we see that for a finitely many-to-one surjective selfmap $\lambda:X\to X$, the sequence $\{\mathfrak T_n^*(\lambda,D)\}_{n\in\N_+}$ appearing in Definition \ref{def-h*} of the contravariant set-theoretic entropy need not be subadditive. So to prove that the $\limsup$ in Definition \ref{def-h*} is actually a limit it is not possible to apply Fekete Lemma \ref{fekete}. On the other hand, we prove in Theorem \ref{limexists} that the limit exists.

\begin{remark}
Let $\lambda:\N\to \N$ be a selfmap defined by $\lambda(1)=\lambda(0)=0$, $\lambda(2n+2)=2n$ and $\lambda(2n+3)=2n+1$ for every $n\in\N$.
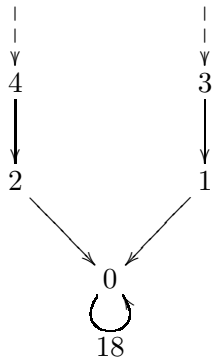
\begin{equation*}
\xymatrix{
\ar@{-->}[d] & & \ar@{-->}[d] \\
4\ar[d] & & 3\ar[d] \\
2\ar[dr] & & 1\ar[dl] \\
& 0\ar@(dl,dr)[] &
}
\end{equation*}

\bigskip
\noindent Then $\mathfrak T_2^*(\lambda,\{0\})=\{0,1,2\}$ and so $|\mathfrak T_2^*(\lambda,\{0\})|=3$, while $\mathfrak T_1^*(\lambda,\{0\})=\{0\}$ and hence $|\mathfrak T^*_1(\lambda,\{0\})|+|\mathfrak T^*_1(\lambda,\{0\})|=2<3$.
\end{remark}

\begin{definition}\label{def-h*}
Let $X$ be a set and $\lambda: X \to X$ a finitely many-to-one surjective selfmap.
\begin{itemize}  
  \item[(a)] For every finite subset $D$ of $X$, the number $$\mathfrak h^*(\lambda, D)=\limsup_{n\to\infty} \frac{|\mathfrak T_n^*(\lambda,D)|}{n}$$ is the {\em contravariant set-theoretic entropy} of $\lambda$ with respect to $D$. 
  \item[(b)]  The number $$\mathfrak h^*(\lambda) = \sup\left\{ \mathfrak h^*(\lambda, D) :D \in[X]^{<\omega}\right\}$$ is the {\em contravariant set-theoretic entropy} of $\lambda$. 
\end{itemize}
\end{definition}

We extend the definition to arbitrary finitely many-to-one selfmaps as follows. For a selfmap $\lambda:X\to X$ of a set $X$, the \emph{surjective core} $\mathrm{sc}(\lambda)$ exists; it is the greatest $\lambda$-invariant subset of $X$ such that the restriction $\lambda\restriction_{\mathrm{sc}(\lambda)}: \mathrm{sc}(\lambda) \to \mathrm{sc}(\lambda)$ is surjective (note that the surjective core can be empty). If $\lambda$ is finitely many-to-one, then $\mathrm{sc}(\lambda)=\bigcap_{n\in\N}\lambda^n(X)$.

\begin{definition}
Let $X$ be a set and $\lambda:X\to X$ a finitely many-to-one selfmap. The \emph{contravariant set-theoretic entropy} of $\lambda$ is $\mathfrak h^*(\lambda)=\mathfrak h^*(\lambda\restriction_{\mathrm{sc}(\lambda)})$.
\end{definition}

Obviously the correspondence $[X]^{<\omega}\to \N\cup\{\infty\}$ given by  $D \mapsto  \mathfrak h^* (\lambda, D)$ is monotone increasing. 

\begin{remark}\label{h=h*}  Let $X$ be a set and $\lambda:X\to X$ a bijection. Then 
\begin{itemize}
\item[(a)] $\mathfrak h^*(\lambda)=\mathfrak h(\lambda^{-1})$.
\item[(b)] Moreover $\mathfrak h^*(\lambda^{-1})=\mathfrak h^*(\lambda)$; indeed, in view of item (a) and Corollary \ref{Coro3}, $
\mathfrak h^*(\lambda^{-1})=\mathfrak h(\lambda)= \mathfrak h(\lambda^{-1})= \mathfrak h^*(\lambda).$  
\end{itemize}
\end{remark}

Here comes the leading example: 

\begin{example}[Shift normalization]\label{Ber_ContravarSet} 
\begin{itemize}
  \item[(a)] Define the \emph{left shift} $\varsigma: \N \to \N$ by $\varsigma(n) = n-1$ for every $n\in\N_+$ and $\varsigma(0)=0$.  Obviously ${\mathfrak h^*}(\varsigma,\{0\}) = {\mathfrak h^*}(\varsigma,\{0,1,\ldots , n\}) = 1$ for every $n\in \N$. Since every finite subset $D$ of $\N$ is contained in $\{0,1,\ldots , n\}$ for some $n\in\N$,  this proves that  $ {\mathfrak h^*}(\varsigma) = 1$.
  \item[(b)] For the right shift $\rho: \N \to \N$ defined in Example \ref{Ber_CovarSet} one has $ {\mathfrak h^*}(\rho) = 0$.
\end{itemize}
\end{example}

\begin{lemma}
Let $X$ be a set and $\lambda: X \to X$ a selfmap. 
\begin{itemize}  
\item[(a)] \emph{(Monotonicity for subsets)} If $Y$ is an  $\lambda$-invariant subset of $X$, then $\mathfrak h^*(\lambda) \geq \mathfrak h^*(\lambda\restriction_Y)$. 

If $Y$ is an inversely $\lambda$-invariant subset of $X$, then $\mathfrak h^*(\lambda) \geq \mathfrak h(\lambda\restriction_{X\setminus Y})$. 

\item[(b)] \emph{(Monotonicity for factors)} If $\psi:Y\to Y$ a selfmap and $\alpha:Y\to X$ is a surjective map such that $\alpha \circ \psi = \lambda \circ \alpha$, then $\mathfrak h^*(\lambda) \leq \mathfrak h^* (\psi)$. 

\item[(c)] \emph{(Invariance under conjugation)} If $Y$ is another set and $\alpha:X\to Y$ is a bijection, then ${\mathfrak h^*}(\lambda)={\mathfrak h^*}(\alpha\circ\lambda\circ\alpha^{-1})$.
\item[(d)] \emph{(Addition Theorem)} If $X=Y_1\cup Y_2$ with $Y_1, Y_2$ disjoint, $\lambda^{-1}$-invariant and $\lambda$-invariant, then $\mathfrak h^*(\lambda)=\mathfrak h^*(\lambda\restriction_{Y_1})+\mathfrak h^*(\lambda\restriction_{Y_2})$.
\end{itemize}
\end{lemma}

Here we give an example showing that the obvious counterpart of the continuity axiom valid for $\mathfrak h$ fails for $\mathfrak h^*$.

\begin{example}\label{Anti-Cont} 
Let $\varrho: \Z\to \Z$ be the \emph{two sided shift} defined by $\varrho(n) = n+1$ for every $n\in\Z$. Then for every $n\in \Z$ the subset $X_n = \{m\in \Z: m\geq n\}$ is $\varrho$-invariant, $\mathfrak h^* (\varrho\restriction _{X_n}) =0$ by Example \ref{Ber_ContravarSet}(b) and  $\Z = \bigcup_{n\in\Z} X_n$. Nevertheless $\mathfrak h^* (\varrho) = 1 > \sup_{n\in\Z} \mathfrak h^* (\varrho\restriction _{X_n}) =0$. 
\end{example}

The next example shows that in the category $\mathfrak S$ of sets and finitely many-to-one surjective maps the inverse limit may not exist.

\begin{example}
Let $\{(Y_n,f_n):n\in\N\}$ be the inverse system of finite sets underlying an infinite binary tree. In other words, 
$Y_n$ is the set of all functions $n \to 2 = \{0,1\}$ and $f_n:Y_{n+1}\to Y_n$ is the restriction
(i.e.,  $Y_0=\{r\}$, $Y_n$ is the set of all binary $n$-tuples for $n>0$, $f_0:Y_1\to Y_0$ is the constant map $r$
and $f_n(y0)=f_n(y1)=y$ for every $y\in Y_n$ and every $n\in\N$).
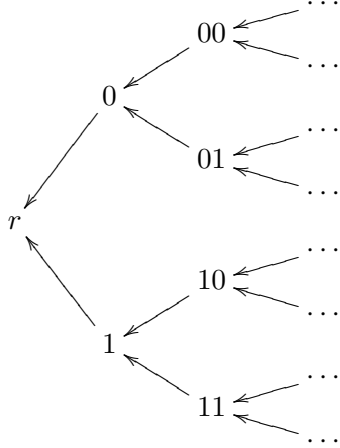
\begin{equation*}
\xymatrix@R-2pc{
&&& \cdots\ar[ld]\\
 & & 00\ar[ldd] \\
&&& \cdots\ar[lu]\\
& 0 \ar[ldddd] & \\
&&&  \cdots\ar[ld]\\
& & 01 \ar[luu]\\
&&& \cdots\ar[lu]\\
r & & \\
&&& \cdots\ar[ld]\\
 & & 10\ar[ldd] \\
&&& \cdots\ar[lu]\\
& 1\ar[luuuu] & \\
&&& \cdots\ar[ld]\\
 & & 11\ar[luu]\\
&&& \cdots\ar[lu]
}
\end{equation*}
If the inverse limit $Z$ of this inverse system would exist in the category $\mathfrak S$, then $Z$ should be infinite, as the projections $Z\to Y_n$ would have to be surjective. Since a surjective map from an infinite set to a finite set is never finitely many-to one, 
such a $Z$ does not exist.
\end{example}

We give now several reductions for the computation of the contravariant set-theoretic entropy. 

\medskip
For a finitely many-to-one selfmap $\lambda:X\to X$ let $$P_*(\lambda,X)= \uparrow\!\{x\in P(\lambda,X): |\uparrow\!x|<\infty\}.$$ Obviously, $P_*(\lambda,X)$ is both $\lambda$-invariant and inversely  $\lambda$-invariant. Moreover 
\begin{equation}\label{P*}
\mathfrak h^*(\lambda) = \mathfrak h^*(\lambda\restriction_{X\setminus P_*(\lambda,X)} ).  
\end{equation}

\begin{remark}
Let $X$ be a set and $\lambda:X\to X$ a finitely many-to-one selfmap. Consider the partition $X=\bigcup_{j\in I}E_j$ from \eqref{(P)}.
Then $$P_*(\lambda,X)=\bigcup\{E_j:j\in I, E_j\ \text{finite}\}.$$
Moreover $\mathrm{sc}(\lambda) \cap P_*(\lambda,X)=\{x\in \mathrm{Per}(\lambda,X):\uparrow\! x\ \text{is finite}\}$.
\end{remark}

\begin{lemma}\label{sabato} 
Let $X$ be a set and $\lambda:X\to X$ a finitely many-to-one surjective selfmap. If $x\in X$, then $\mathfrak h^*(\lambda,\{x\}) > 0$ if and only if $\uparrow\! x$ is infinite. 
\end{lemma}

For a set $X$ and a finitely many-to-one selfmap $\lambda:X\to X$, let ${\bf P}(\lambda,X)$ be the greatest $\lambda$-invariant and inversely $\lambda$-invariant subset of $X$ such that $\mathfrak h^*(\lambda \restriction_{{\bf P}(\lambda,X) })=0$. The following theorem shows in particular the existence of such subset.

\begin{theorem}\label{setPinsker*}
Let $X$ be a set and $\lambda: X \to X$ a finitely many-to-one selfmap. Then:
\begin{itemize}
\item[(a)] $\mathfrak h^*(\lambda)=0$ if and only if $\mathrm{sc}(\lambda)=\mathrm{Per}(\lambda,X)$;
\item[(b)] ${\bf P}(\lambda,X) = X \setminus \bigcup\{\uparrow\! x: x\in X, \uparrow\! x \mbox{ infinite}\} = \{x\in X: \uparrow\! \lambda^n(x) \mbox{ is finite for all }n\in \N\}$.
\end{itemize}
\end{theorem}
\begin{proof} 
(a) If $\mathrm{sc}(\lambda)=\mathrm{Per}(\lambda,X)$, then for every non-empty finite subset $D$ of $\mathrm{sc}(\lambda)$, the cotrajectory $\mathfrak T^*(\lambda,D)$ of $D$ under $\lambda$ is finite. Hence $\mathfrak h^*(\lambda)=0$. Now assume that $\mathfrak h^*(\lambda)=0$. Since always $\mathrm{Per}(\lambda,X)\subseteq \mathrm{sc}(X)$, it remains to check that $\mathrm{Per}(\lambda,X)\supseteq \mathrm{sc}(\lambda)$. Assume that there exists $x\in \mathrm{sc}(\lambda)$ with $x\not \in \mathrm{Per}(\lambda,X)$. Then $\uparrow\! x$ is infinite, so $\mathfrak h^*(\lambda,\{x\}) > 0$ by Lemma \ref{sabato}, a contradiction. 

\smallskip
(b) Let for convenience
$$U = \bigcup\{\uparrow\! x: x\in X, \uparrow\! x\ \text{infinite}\}\ \text{and}\ V = \{x\in X: \uparrow\! \lambda^n(x) \mbox{ is finite for all }n\in \N\}.$$
Note that $y \in U$ precisely when there exists $n\in\N$ such that $\uparrow\! \lambda^n(y)$ is infinite, that is, $y \not \in V$.
This proves that $y \in U$ precisely when $y \not \in V$, i.e., the desired equality $X \setminus U = V$ holds.

The set $U$ is inversely $\lambda$-invariant, so that $X \setminus U$ is $\lambda$-invariant. On the other hand, $V$ is inversely $\lambda$-invariant.
Now assume that $Y\subseteq X $ is $\lambda$-invariant, and inversely $\f$-invariant, and that $\mathfrak h^*(\lambda \restriction_{Y})=0$. Pick $y\in Y$. Then $z = \lambda^n(y) \in Y$ and $\uparrow\! \lambda^n(y) \subseteq Y$ for every $n\in \N$. Hence $\mathfrak h^*(\lambda, \{\lambda^n(y)\}) = 0$. By Lemma \ref{sabato}, $\uparrow\!\lambda^n(y)$ is finite for every $n\in\N$. Therefore $y\in V$, and hence $Y \subseteq V$. This proves that $\mathbf P(\lambda,X)= X\setminus U=V$.
\end{proof}

\begin{remark} 
\begin{itemize}
\item[(a)] Example \ref{Anti-Cont} shows that, if in the definition of ${\bf P}(\lambda,X)$ we ask this set to be {\em only} $\lambda$-invariant, then such a  subset need not exist even in the case of the bijection $\varrho: \Z\to \Z$. 
\item[(b)] Note that $P_*(\lambda,X)\subseteq \mathbf P(\lambda,X)$. Moreover $\mathbf P(\lambda,X)\setminus P_*(\lambda,X)$ coincides with the union of those infinite classes $E_j$ in \eqref{(P)} that miss $\mathrm{sc}(\lambda)$ (this must be compared with item (a)). 
\end{itemize}
\end{remark}
 
In view of item (b) of the above remark and \eqref{P*}, we may assume without loss of generality that $P_*(\lambda,X)=\emptyset$ in the sequel. 
 
\medskip
We start from the following reduction to finite subsets containing no periodic point.

\begin{lemma}\label{red-notper}
Let $X$ be a set and $\lambda:X\to X$ a finitely many-to-one surjective selfmap. For every $D\in[X]^{<\omega}$ with $\mathfrak h^*(\lambda,D)>0$  there exists $E\in[X]^{<\omega}$ such that $E\subseteq \uparrow\! D$, $E\subseteq X\setminus \mathrm{Per}(\lambda,X)$ and there exists $n_0\in\N_+$ with $$\mathfrak T^*_{n-n_0}(\lambda,E)  \subseteq \mathfrak T^*_{n}(\lambda, D)\subseteq \mathfrak T^*_{n}(\lambda,E) \cup D_E\ \text{for all}\ n \in\N,n\geq n_0,$$
 where $D_E=\bigcup_{n\in\N_+}\lambda^n(E)\cap \uparrow\! D$. Therefore (as $D_E$ is finite), $\mathfrak h^*(\lambda,D)=\mathfrak h^*(\lambda,E)$.
\end{lemma}

\begin{lemma}\label{no-rec} 
Let $X$ be a set, $\lambda:X\to X$ a finitely many-to-one surjective selfmap and $E\in[X]^{<\omega}$ with $E\cap \mathrm{Per}(\lambda,X)=\emptyset$.  Then:
\begin{itemize}
\item[(a)] $\uparrow\! E=\uparrow\!\F(E)$ and so $\F(E)\neq\emptyset$; 
\item[(b)] $E\subseteq \mathfrak T^*_m(\lambda,\F(E))$ for some $m\in\N_+$, hence $\mathfrak T^*_n(\lambda,\F(E)) \subseteq \mathfrak T^*_n(\lambda, E)\subseteq \mathfrak T^*_{n+m}(\lambda, \F(E))$ for every $n,m\in\N_+$. In particular $\mathfrak h^*(\lambda,\F(E))=\mathfrak h^*(\lambda,E)$.
\end{itemize}
\end{lemma}
\begin{proof} 
(a) Since obviously $E \subseteq\uparrow\!\F(E)$, we deduce that $\uparrow\! E=\uparrow\!\F(E)$ and $\F(E)\neq\emptyset$. 

\smallskip
(b) Let $m,n\in\N_+$. From the inclusion $E \subseteq \uparrow \! \F(E)$ and the fact that $E$ is finite, we have that $E\subseteq \mathfrak T^*_m(\lambda,\F(E))$ for some $m\in\N_+$. Hence 
\begin{equation}\label{Catena}
\mathfrak T^*_n(\lambda,\F(E)) \subseteq \mathfrak T^*_n(\lambda, E)\subseteq \mathfrak  T^*_n(\lambda, \mathfrak T^*_m(\lambda, \F(E)))\subseteq  \mathfrak T^*_{n+m} (\lambda, \F(E)).
\end{equation}
The last assertion follows from these inclusions and the definition of  $\mathfrak h^*$. 
\end{proof}

In this terminology Lemma \ref{red-notper} and Lemma \ref{no-rec} show that it suffices to consider finite antichains 
without periodic points for the computation of the contravariant set-theoretic entropy, in other words
\begin{equation}\label{antieq}
\mathfrak h^*(\lambda)=\sup\{\mathfrak h^*(\lambda,E): E\in[X]^{<\omega}, E\cap P(\lambda,X)=\emptyset, E\ \text{antichain}\}.
\end{equation}

\medskip
In the next lemma we collect the properties of finite antichains and in particular in item (d) we isolate a special property.

\begin{lemma}\label{h*:E->x} 
Let $X$ be a set, $\lambda:X\to X$ a finitely many-to-one surjective selfmap and $E\in[X]^{<\omega}$ an antichain with $E\cap \mathrm{Per}(\lambda,X)=\emptyset$. Then:
\begin{itemize}
 \item[(a)] for every $n\neq m$ in $\N_+$, $\lambda^{-n}(E)$ and $\lambda^{-m}(E)$ are disjoint antichains; in particular $\lambda^{-n}(E)\cap \mathfrak T_n^*(\lambda, E)=\emptyset$ and  $\uparrow \!  x \cap \uparrow \! y=\emptyset$ whenever $x,y\in E$ and $x\neq y$;
    \item[(b)] one has a partition $\mathfrak T_n^*(\lambda,E)=\stackrel{\circ}{\bigcup}_{x\in E} \mathfrak T_n^*(\lambda,\{x\})$ for every $n\in\N_+$;
   \item[(c)]  the sequence $\left\{\frac{|\mathfrak T^*_n(\lambda,E)|}{n}\right\}_{n\in\N_+}$ is increasing, hence convergent (or divergent to $\infty$); in particular
     \begin{itemize} 
        \item[(c$_1$)]  $|\mathfrak T_n^*(\lambda,E)|\geq  n|E|$ for every $n\in\N_+$;
        \item[(c$_2$)]  $\mathfrak h^*(\lambda,E)\geq |E|$;
        \item[(c$_3$)]  $\mathfrak h^*(\lambda, E)=\sum_{x\in E}\mathfrak h^*(\lambda, \{x\})$;
     \end{itemize}
      \item[(d)] for the finite antichain $E$, the following conditions are equivalent:  
     \begin{itemize} 
        \item[(d$_1$)]  $|\lambda^{-n}(E)|=|E|$ for every $n\in\N_+$;
        \item[(d$_2$)]  $|\mathfrak T_n^*(\lambda,E)|=n|E|$ for every $n\in\N_+$;
        \item[(d$_3$)]  $\mathfrak h^*(\lambda, E)=|E|$.
     \end{itemize}
\end{itemize}
\end{lemma}
\begin{proof} 
(a) is easy to check and (b) follows immediately from (a). 

\smallskip
(c) Item (b) gives $$|\mathfrak T^*_{n+1}(\lambda,E)| = \sum_{i=0}^n |\lambda^{-i}(E)|= |\mathfrak T^*_n(\lambda,E)| + |\lambda^{-n}(E)|.$$ Hence 
$$
(n+1)  |\mathfrak T^*_n(\lambda,E)| = n  |\mathfrak T^*_n(\lambda,E)| +  |\mathfrak T^*_n(\lambda,E)| \leq n( |\mathfrak T^*_n(\lambda,E)| +  |\lambda^{-n}(E)|) = n|\mathfrak T^*_{n+1}(\lambda,E)|,
$$
since 
\begin{equation}\label{Formula1}
 |\lambda^{-i}(E)| \leq  |\lambda^{-j}(E)|\ \text{for every}\ i,j\in\N, i < j.
\end{equation}
This proves the first assertion. 

(c$_1$) follows again from \eqref{Formula1}; indeed, it implies that $|E|\leq|\lambda^{-i}(E)|$ for every $i\in\N$, so $|\mathfrak T^*_n(\lambda,E)| = \sum_{i=0}^{n-1} |\lambda^{-i}(E)|\geq n|E|$.

(c$_2$) follows from (c$_1$). 

(c$_3$) follows from (b) and the fact that $\left\{\frac{|\mathfrak T^*_n(\lambda,E)|}{n}\right\}_{n\in\N_+}$, as well as $\left\{\frac{|\mathfrak T^*_n(\lambda,\{x\})|}{n}\right\}_{n\in\N_+}$, converges.

\smallskip
(d) follows from (a) and (c). 
 \end{proof}

\begin{definition} 
Let $X$ be a set, $\lambda:X\to X$ a finitely many-to-one surjective selfmap and $E,A\in[X]^{<\omega}$.
\begin{itemize} 
  \item[(a)] Call $E$ a {\em stratifiable antichain} if $E$ is an antichain satisfying the equivalent properties of Lemma \ref{h*:E->x}(d). 
  \item[(b)] Call $A$ a {\em stratifiable antichain of $E$} if $A$ is a stratifiable antichain and there exists $n_0\in\N_+$ such that 
\begin{equation}\label{33}
\mathfrak T^*_{n-n_0}(\lambda,A)  \subseteq \mathfrak T^*_{n}(\lambda, E)\subseteq \mathfrak T^*_{n}(\lambda,A) \cup D_A\ \text{for all}\ n \in \N,
n\geq n_0, 
\end{equation}
 where $D_A=\bigcup_{n\in\N_+}\lambda^n(A)\cap \uparrow\! E$. 
 \end{itemize} 
\end{definition}

Note that the first inclusion in item (b) yields $A \subseteq \mathfrak T^*_{n_0}(\lambda, E)$, so that $D_A=\bigcup_{n=0}^{n_0}\lambda^{n}(A)\cap \uparrow\! E$ is finite. 

\begin{corollary}\label{max-a-chain} Let $X$ be a set and $\lambda:X\to X$ a finitely many-to-one surjective selfmap. If $E\in[X]^{<\omega}$ and $E\cap P(\lambda,X)=\emptyset$, then $\mathfrak h^*(\lambda, E) < \infty$ if and only if there exists a stratifiable antichain $A$ of $E$; in such a case  $\mathfrak h^*(\lambda, E) = |A|=\mathfrak h^*(\lambda, A)$. 
\end{corollary}

\begin{proof} If $E$ admits a stratifiable antichain $A$, then obviously $\mathfrak h^*(\lambda, E) =\mathfrak h^*(\lambda, A) = |A| < \infty$. 

Suppose that $\mathfrak h^*(\lambda, E) =m < \infty$. Then $\mathfrak h^*(\lambda, E) =\mathfrak  h^*(\lambda, \F(E)) $ for the antichain $\F(E)$ by Lemma \ref{no-rec}(b). By Lemma \ref{h*:E->x}(c), $|\F(E)| \leq m$. Let $a_i=|\lambda^{-i}(\F(E))|$ for $i\in\N$. In view of \eqref{Formula1} the sequence $\{a_i\}_{i\in\N}$ is an increasing sequence of integers with $a_0\leq m$. By Lemma \ref {h*:E->x}(b), 
$$\mathfrak h^*(\lambda, E)= \mathfrak h^*(\lambda, \F(E)) = \lim_{n\to\infty} \frac{1}{n}\sum_{i=0}^{n-1}a_i = m,$$ 
hence we conclude that there exists $n_0\in\N$ such that $a_i = m$ for all $i\geq n_0$. Let $A = \lambda^{-n_0}(\F(E))$; in particular, $A\subseteq\uparrow\!E$ and $A$ is an antichain by Lemma \ref{h*:E->x}(a). By the choice of $n_0$, $A$ satisfies the equivalent conditions of item (d) of Lemma \ref{h*:E->x}. Thus $A$ is a stratifiable antichain of $E$. 
\end{proof}

At this stage we can show that, for a set $X$, a finitely many-to-one surjective selfmap $\lambda:X\to X$ and any $E\in[X]^{<\omega}$, the sequence $\left\{\frac{|\mathfrak T^*_n(\lambda,E)|}{n}\right\}_{n\in\N}$ converges; in particular, the $\limsup$ defining the contravariant entropy $\mathfrak h^*(\lambda,E)$ in Definition \ref{def-h*} is a limit.

\begin{proposition}\label{limexists} 
Let $X$ be a set and $\lambda: X \to X$ a finitely many-to-one surjective selfmap. For every $E\in[X]^{<\omega}$, the sequence $\left\{\frac{|\mathfrak T_n^*(\lambda,E)|}{n}\right\}_{n\in\N}$ converges.
\end{proposition}
\begin{proof} 
By Lemma \ref{red-notper} and Lemma \ref{no-rec}(b) we can assume that $E\cap P(\lambda,X)=\emptyset$ and that $E$ is an antichain.

If $\mathfrak h^*(\lambda,E)$ is infinite, the assertion easily follows from Lemma \ref{h*:E->x}(c) as the sequence  $\left\{\frac{|\mathfrak T_n^*(\lambda,E)|}{n}\right\}_{n\in\N_+}$ is increasing. 

Assume that $\mathfrak h^*(\lambda,E)$ is finite. By Corollary \ref{max-a-chain} there exists a stratifiable antichain $A\in[X]^{<\omega}$ of $E$, that is, $\mathfrak T_n^*(\lambda,A)=n|A|$ for every $n\in\N_+$, and there exists $n_0\in\N_+$ such that \eqref{33} holds. Therefore 
$$(n-n_0)|A|= |\mathfrak T_{n-n_0}^*(\lambda,A)|\leq |\mathfrak T_{n}^*(\lambda,E)|\leq |\mathfrak T_{n}^*(\lambda,A)|+|D_A|=n|A|+|D_A|,$$ 
and hence the sequence $\left\{\frac{|\mathfrak T_n^*(\lambda,E)|}{n}\right\}_{n\in\N_+}$ converges to $|A|=\mathfrak h^*(\lambda,E)$, as $D_A$ is finite.
\end{proof}

\begin{definition} Let $X$ be a set, $\lambda:X\to X$ a finitely many-to-one surjective selfmap and $E$ a subset of $X$. Then $y\in X$ is a \emph{ramification point of $\lambda$ over $E$} if $y\in\uparrow \! E$ and $|\lambda^{-1}(y)|>1$.
\end{definition}

The motivation to introduce this notion comes from the fact that $\mathfrak h^*(\lambda,E)=\infty$ for a finite subset $E$ of $X$ precisely when $\lambda$ has infinitely many ramification points over $E$, as the next proposition shows.

\begin{proposition}\label{ram-rec}
Let $X$ be a set, $\lambda:X\to X$ a finitely many-to-one surjective selfmap and $E\subseteq X\setminus P(\lambda,X)$ finite. The following conditions are equivalent:
\begin{itemize}
 \item[(a)] $\lambda$ has infinitely many ramification points over $E$; 
 \item[(b)] for every $m\in\N_+$ there exists a finite antichain $F_m\subseteq \uparrow \! E$ such that $|F_m|\geq m$;
 \item[(c)] $\mathfrak h^*(\lambda,E)=\infty$. 
 \end{itemize}
\end{proposition}

\begin{proof}  (a)$\Rightarrow$(b) Denote by $Y$ the set of all  ramification points over $E$. Assume for a contradiction that there exists $m\in \N_+$ such that every antichain $F$ over $E$ has  $ < m$ elements. Since $Y$ is infinite and every antichain in $Y$ has  $ < m$ elements, there exists a  chain $C = \{y_{1}\leq  \ldots \leq y_{m+1}\}$ in $Y$. Suppose $\lambda(z_i) = y_i$ for every $i$ and $z_i \ne y_{i+1}$ for all $i \leq m$. Then $\{z_1,\ldots,z_m\}\subseteq\uparrow\!E$ is an antichain, so we are done. 

\smallskip
(b)$\Rightarrow$(c) is trivial. 

\smallskip
(c)$\Rightarrow$(a) Assume that $\lambda$ has exactly $m$ ramification points $y_{1},\ldots , y_{m}$ over $E$. 
Since $\mathfrak h^*(\lambda,E)= \mathfrak h^*(\lambda,\F(E))$ by Lemma \ref{no-rec}, we can replace $E$ by $\F(M)$, so that 
we can assume in the sequel that $E$ is an antichain. For $i = 1,\ldots, m$ let $t_i$ be the smallest natural such that $\lambda^{t_i}(y_i) \in E$. Let $k = \max\{t_i:i = 1,\ldots, m \} + 1$ and $A = \lambda^{-k}(E)$. Then $A\subseteq \uparrow \! E$ is an antichain by Lemma \ref{h*:E->x}(a). Moreover $|\lambda^{-n}(A)| = |A|$ for every $n\in \N $ as $\mathfrak T^*_{k} (\lambda, E)$ contains all ramification points $y_{1},\ldots , y_{m}$ over $E$. So $\mathfrak h^*(\lambda,E)<\infty$ by Corollary \ref{max-a-chain}, a contradiction. 
\end{proof}

Now we recall a pair of notions introduced in \cite{AADGH} that help us in computing the contravariant set-theoretic entropy. 

\begin{definition}\cite{AADGH} Let $X$ be a set and $\lambda: X \to X$ a selfmap.
\begin{itemize}  
\item[(a)] A \emph{string} of  $\lambda$ in $X$ is an infinite sequence $S=\{s_n\}_{n\in\N}$ of pairwise distinct elements of $X$ such that $\lambda(s_n)=s_{n-1}$ for every $n\in\N_+$.
\item[(b)] The maximum number $ \mathfrak s(\lambda)$ of pairwise disjoint strings of $\lambda$ is called {\em string number} of $\lambda$.
\end{itemize}
\end{definition}  

Therefore $\mathfrak s(\lambda)$ tells us, roughly speaking, how many pairwise copies of the left shift $\varsigma$ we can detect within the dynamical system $\lambda: X \to X$. 

In the next lemma we intend $width(X)$ defined as $\infty$ whenever $\sup \{|A|: A \mbox{ finite antichain in }X\}$ is infinite. 

\begin{lemma} \label{h*=w}
Let $X$ be a set, $\lambda:X\to X$ a finitely many-to-one surjective selfmap. 
Then $$\mathfrak h^*(\lambda) = width(X\setminus P_*(\lambda,X),\leq_\lambda).$$ 
\end{lemma}
\begin{proof} 
In view of \eqref{P*} we can assume that $P_*(\lambda,X) =\emptyset$. 

Take a finite antichain $A$ and note that $\mathfrak h^*(\f, A) \geq |A|$ by Lemma \ref{h*:E->x}(c). This proves that  $ \mathfrak  h^*(\lambda) \geq width(X) $.

If  $\mathfrak  h^*(\lambda) =\infty$, then by Proposition \ref{ram-rec} $X$ has arbitrarily large antichains, so $width(X) = \infty$ as well. Suppose now that $\mathfrak h^*(\lambda)=m<\infty$. To prove that $width(X)\geq\mathfrak h^*(\lambda)$ apply Corollary \ref{max-a-chain} to get an antichain of size $|A|=m$. 
\end{proof}

\begin{theorem}\label{ConsetE}  
Let $X$ be a set and $\lambda: X \to X$ a finitely many-to-one selfmap. Then $\mathfrak h^*(\lambda)= \mathfrak s(\lambda)$. 
\end{theorem}
\begin{proof} 
We can assume that $\lambda$ is surjective, in view of the definition of $\mathfrak h^*$ and since $ \mathfrak s(\lambda)= \mathfrak s(\lambda\restriction_{\mathrm{sc}(\lambda)})$ by \cite{DGV}. Moreover, due to \eqref{P*} and the obvious counterpart for $\mathfrak s(\lambda)$, we can assume without loss of generality that $P_*(\lambda,X)=\emptyset$, so that $\mathfrak h^*(\lambda)=width(X)$ by Lemma \ref{h*=w}.  Hence it suffices to prove that $ \mathfrak s(\lambda) = width(X,\leq)$. 

Indeed, fix a family of pairwise disjoint strings $\{S_i: i \in I\} $ and for every $i \in I$ pick $a_i \in S_i$. Then   $A=  \{a_i: i \in I\}$ is an antichain in $X$.  Conversely, for every  antichain  $ A$ in $X$ the  family $\{\uparrow\! a: a \in A\}$ is pairwise disjoint and each $ \uparrow\! a$   contains a string (by the surjectivity of $\f$). This proves that $ \mathfrak s(\f) = width(X)$.
\end{proof}

Now we can establish the Logarithmic Law for $\mathfrak h^*$. 

\begin{proposition}[Logarithmic Law]\label{LL*} 
Let $X$ be a set and $\lambda: X \to X$ a finitely many-to-one selfmap. Then 
$\mathfrak h^*(\lambda^k) = |k| \mathfrak h^*(\lambda)$ for every $k\in\Z$.
\end{proposition}
\begin{proof}  
Fix $k\in \N_+$. Making use of Theorem \ref{ConsetE}  we can prove $\mathfrak h^*(\lambda^k) = k \mathfrak h^*(\lambda)$. 
To this end it suffices to note that every string of $\lambda$ gives rise, in an obvious way, to a family of $k$ paiwise disjoint strings of $\lambda^k$.
To see that the formula holds for every $k\in\Z$, apply Remark \ref{h=h*}(b).
\end{proof}

\begin{remark} 
We give the definition of contravariant set-theoretic entropy $\mathfrak h^*$ for a finitely many-to-one selfmap $\lambda: X \to X$ of a set $X$
in two steps: first for surjective selfmaps and then in the general case of not necessarily surjective $\lambda$, by letting $\mathfrak h^*(\lambda) =\mathfrak h^*(\lambda\restriction_{\mathrm{sc}(\lambda)})$. It is possible to define $\mathfrak h^*$ in a single step taking in item (a) of Definition \ref{def-h*} a  {\em reduced $\lambda$-cotrajectory} instead of the usual one, by replacing the inverse images $\lambda^{-i}(D)$ in the definition of the $\lambda$-cotrajectory $\mathfrak T^*_n(\lambda, D)$ by {\em reduced inverse images} $\lambda^{-i}(D)\cap \mathrm{sc}(\lambda)$. Taking this reduced  $\lambda$-cotrajectory $\mathfrak T_n^*(\lambda,D\cap \mathrm{sc}(\lambda))=\bigcup_{i=0}^{n-1}(\lambda^{-i}(D)\cap \mathrm{sc}(\lambda))$ in place of $\mathfrak T^*_n(\lambda, D)$, one obtains
\begin{equation}\label{RedInvImage}
\mathfrak h^*(\lambda,D)=\limsup_{n\to\infty}\frac{|\mathfrak T_n^*(\lambda,D\cap \mathrm{sc}(\lambda))|}{n}.
\end{equation}
Our choice to define contravariant set-theoretic entropy by using reduced inverse images (i.e., restricting to the surjective core) is justified 
by the applications. Indeed, denote by $\widetilde{\mathfrak h}^*$ the  entropy defined by the formula  item (a) of Definition \ref{def-h*}, i.e.,  
using in \eqref{RedInvImage} the inverse images instead of the reduced ones. Here is a simple example of a finitely many-to-one selfmap $\lambda:X\to X$, such that $\widetilde{\mathfrak h}^*(\lambda)=\infty$, while ${\mathfrak h}^*(\lambda)=\mathfrak s(\lambda)=1$. For every $n\in\N_+$, take a set $X_n$ such that $|X_n|=2n$ (one can choose here also $|X_n|=n$) and define $X$ as the disjoint union $X=\N\stackrel{\circ}{\cup}\stackrel{\circ}{\bigcup}_{n\in\N_+}X_n$. Let $\lambda:X\to X$ be defined by $\lambda\restriction_\N=\varsigma$ (the left shift) and $\lambda(X_n)=n$ for every $n\in\N_+$. Then $|\lambda^{-1}(n)|=|X_n|+1=2n+1$ for every $n\in\N_+$. Therefore $\widetilde{\mathfrak h}^*(\lambda)=\infty$. On the other hand $\mathrm{sc}(\lambda)=\N$ and so $\mathfrak h^*(\lambda)=\mathfrak h^*(\varsigma)=1$.
The diagram representing $\lambda$ is the following.
\begin{equation*}
\xymatrix{
& & \ar@{-->}[d] \\
\ar[drr] & \ar[dr] & 3 \ar[d] & \ar[dl] & \ar[dll] \\
& \ar[dr] & 2 \ar[d] & \ar[dl] &  \\
& & 1 \ar[d] & & \\
& & 0 \ar@(dl,dr)[] & &
}
\end{equation*}
\bigskip
\end{remark}

\begin{remark} The string number of group endomorphisms was deeply studied in \cite{DGV} and \cite{GV}.
In particular, we would like to mention the main result of \cite{DGV}, which is a Dichotomy Theorem, showing that, for $G$ an abelian group and $\phi:G\to G$ an endomorphism, $\mathfrak s(\phi)$ is either zero or infinity.
\end{remark}

\begin{remark} Proposition \ref{ram-rec} implies that, if $\mathfrak h^*(\lambda)< \infty$ for a finitely many-to-one surjective selfmap $\lambda:X\to X$ of a set $X$, then for every $x\in X$ the selfmap $\lambda$ may have only finitely many  ramification points over $x$. Hence, if $\lambda$ has all its fibers of the same size (e.g., if $X$ is a group and $\lambda$ is a group endomorphism),  then $\mathfrak h^*$ may take only two values,  either $\mathfrak h^*(\lambda)=0$ or $\mathfrak h^*(\lambda)=\infty$ (this should be compared via Theorem \ref{ConsetE} with \cite{DGV}).
\end{remark}

\newpage
\section{Topological entropy}\label{top-sec}

\subsection{Topological entropy on compact spaces}\label{akm-sec}

The topological entropy was introduced by Adler, Konheim and McAndrew \cite{AKM} as follows.

For a compact space $X$, let ${\cov(X)}$ be the family of all open covers of $X$. For $\U\in \cov(X)$ the \emph{entropy} of $\U$ is $$H(\U)=\log N(\U),$$ where $N(\U)=\min\{|\mathcal V|:$  $\mathcal V$ is a finite subcover of $\U\}$ (when necessary, we write $N_X$ instead of simply $N$).
For $m\in\N_+$ and open covers $\U_1, \ldots ,\U_m\in\cov(X)$ let $$\U_1 \vee \ldots \vee \U_m=\left\{\bigcap _{i=1}^m U_i: U_i\in \U_i\right\}.$$
For a continuous selfmap $\psi:X\to X$, $\U\in \cov (X)$ and $n\in\N_+$, let $\psi^{-n}(\U)=\{\psi^{-n}(U): U\in \U\}$. Then $\psi^{-n}(\U_1\vee \ldots \vee\U_m)=\psi^{-n}(\U_1)\vee \ldots \vee \psi^{-n}(\U_m)$ for every $n,m\in\N_+$.
 
\medskip
The second part of the following lemma follows from Fekete Lemma \ref{fekete}.

\begin{lemma} 
Let $X$ be a compact space, $\U\in\cov(X)$ and for every $n\in\N_+$ let
$$c_n=H(\U\vee \psi^{-1}(\U) \vee \ldots \vee \psi^{-n+1}(\U)).$$ Then the sequence $\{c_n\}_{n\in\N_+}$ is subadditive, so the sequence $\{\frac{c_n}{n}\}_{n\in\N_+}$ converges and $$\lim_{n\to \infty}\frac{c_n}{n}=\inf_{n\in\N} \frac{c_n}{n}.$$
\end{lemma}

\begin{definition}{\cite{AKM}} 
Let $X$ be a compact space and $\psi:X\to X$ a continuous selfmap.
\begin{itemize}
\item[(a)] For $\U\in\cov(X)$ the {\em topological entropy of $\psi$ with respect to $\U$} is $$H_{top}(\psi,\U)=\lim_{n\to\infty} \frac{H(\U\vee \psi^{-1}(\U) \vee \ldots \vee \psi^{-n+1}(\U))}{n}.$$
\item[(b)] The {\em topological entropy} of $\psi$ is $$h_{top}(\psi)=\sup\{H_{top}(\psi,\U):\U \in \cov(X)\}.$$
\end{itemize}
\end{definition}

We give now the basic properties of the topological entropy for continuous selfmaps of compact spaces. 

\begin{lemma}\label{monot}
Let $X$ be a compact space and $\psi:X\to X$ a continuous selfmap. Let $Y$ be another compact space, $\phi :Y \to Y$ a continuous selfmap, and assume that there exists a continuous map $\alpha:X \to Y$, with $\alpha\psi = \phi\alpha$.
\begin{itemize}
\item[(a)] (Monotonicity for continuous images) If $\alpha$ is surjective, then $h_{top}(\psi) \geq h_{top}(\phi).$
\item[(b)] (Monotonicity for invariant subspaces) If $\alpha$ is injective, then $h_{top}(\psi)\leq h_{top}(\phi)$.
\end{itemize} 
\end{lemma}

An easy consequence of monotonicity for continuous images is the {\em invariance} of $h_{top}$ under topological conjugation. 

\begin{theorem}[Invariance under conjugation]
Let $X$ be a compact space and $\psi:X\to X$ a continuous selfmap. If $\alpha:X\to Y$ is a homeomorphism between compact spaces, then $h_{top}(\alpha\phi\alpha^{-1})=h_{top}(\phi)$.
\end{theorem}

The next theorem allows us to reduce to \emph{surjective} continuous selfmaps of compact spaces in the computation of the topological entropy.
For a compact space $X$ and a continuous selfmap $\psi:X\to X$, let $$E_\psi(X)=\bigcap_{n\in\N}\psi^n(X);$$

\begin{theorem}[Reduction to surjective selfmaps] \label{red-to-sur}
Let $X$ be a compact space and $\psi:X\to X$ be a continuous selfmap.  Then $E_\psi(X)$ is closed and $\psi$-invariant, the map $\psi\restriction_{E_\psi(X)}:E_\psi(X)\to E_\psi(X)$ is surjective and $h_{top}(\psi)=h_{top}(\psi\restriction_{E_\psi(X)})$.
\end{theorem}

\begin{theorem}[Logarithmic Law]  
Let $X$ be a compact space and $\psi:X\to X$ a continuous selfmap. Then:
\begin{itemize}
\item[(a)] $h_{top}(\psi^{k})=kh_{top}(\psi)$ for every $k\in  \N$;
\item[(b)]  if $\psi:X\to X$ is a homeomorphism, then $h_{top}(\psi^{-1})=h_{top}(\psi)$, so $h_{top}(\psi^{k})=|k|h_{top}(\psi)$ for every $k\in \Z$.
\end{itemize}
\end{theorem}

Let $(X_i,\varphi_{i,j},I)$ be an inverse system of compact spaces $X_i$ and let $X=\varprojlim X_j$. This means that $I$ is a directed set, for every $j\leq i$ in $I$, $\varphi_{i,j}:X_i\to X_j$ is a continuous selfmap such that $\varphi_{i,j}\circ\varphi_{k,i}=\varphi_{k,j}$ for every $k\leq j\leq i$ in $I$, and $\varphi_{i,i}=id_{X_i}$. For every $i\in I$ there exists a continuous map $\varphi_i:X\to X_i$ such that $\varphi_{i,j}\circ \varphi_i=\varphi_j$ if $j\leq i $ in $I$.
If for every $i\in I$ a continuous selfmap $\psi_i:X_i\to X_i$ is given such that $\psi_i\circ\varphi_{i,j}=\varphi_{i,j}\psi_j$ with $j\leq i$ in $I$, then there exists a unique continuous map $\psi:X\to X$ such that $\varphi_i\circ \psi=\psi_i\circ \varphi_i$ for every $i\in I$, that is, $\psi=\varprojlim \psi_i$.

\begin{proposition}[Continuity]\emph{\cite[Proposition 2.1]{S}}\label{cont-sp}
Let $(X_i,\varphi_{ij},I)$ be an inverse system of compact spaces $X_i$ with $\varphi_{ij},$ surjective for every $i,j\in I$, and let $\psi_i:X_i\to X_i$ be continuous maps such that $X=\varprojlim X_i$ and $\psi=\varprojlim \psi_i$. Then $h_{top}(\psi)=\sup_{i\in I} h_{top}(\psi_i)$.
\end{proposition}

\begin{theorem}[Weak Addition Theorem]\label{WAT-top}
Let $X_1,X_2$ be compact spaces and $\psi_i:X_i\to X_i$ a continuous selfmap for $i=1,2$. Then $$h_{top}(\psi_1\times \psi_2)= h_{top}(\psi_1) + h_{top}(\psi_2).$$
\end{theorem}

 The Weak Addition Theorem concerns finite products in the category of compact spaces. We give now the following result covering its counterpart for finite coproducts.

\begin{proposition}
Let $X$ be a compact space and $\psi:X\to X$ a continuous selfmap. Let $X_1$, $X_2$ be compact $\psi$-invariant subspaces of $X$ such that $X=X_1\cup X_2$. Then $$h_{top}(\psi)=\max\{ h_{top}(\psi\restriction_{X_1}), h_{top}(\psi\restriction_{X_2})\}.$$
\end{proposition}
\begin{proof}
Let $\psi_i=\psi\restriction_{X_i}$ for $i=1,2$.
By the monotonicity from Lemma \ref{monot}(b) we have the inequality $h_{top}(\psi)\geq\max\{ h_{top}(\psi_1), h_{top}(\psi_2)\}$. To prove the converse inequality, we first fix a notation, namely for $\mathcal U\in\cov(X)$, let $\mathcal U_i=\mathcal U\restriction_{X_i}\in\cov(X_i)$ for $i=1,2$. Then, for $i=1,2$, $\U,\mathcal V\in\cov(X)$ and $n\in\N$, $$(\U\vee \mathcal V)_i=\U_i\vee\mathcal V_i\ \text{and}\ \psi^{-n}(\U)_i=\psi_i^{-n}(\U_i);$$
moreover $$N_X(U)\leq N_{X_1}(U_1)+N_{X_2}(U_2).$$
Therefore for every $n\in\N_+$ we have
\begin{equation*}
\begin{split}
N_X(\U\vee \psi^{-1}(\U) \vee \ldots \vee \psi^{-n+1}(\U))\leq N_{X_1}(\U_1\vee \psi_1^{-1}(\U_1) \vee \ldots \vee \psi_1^{-n+1}(\U_1))+\\
N_{X_2}(\U_2\vee \psi_2^{-1}(\U_2) \vee \ldots \vee \psi_2^{-n+1}(\U_2)).
\end{split}
\end{equation*}
Hence $H_{top}(\psi,\U)\leq \max\{H_{top}(\psi_1,\U_1),H_{top}(\psi_2,\U_2)\}$; in particular $h_{top}(\psi)\leq\max\{ h_{top}(\psi_1), h_{top}(\psi_2)\},$ and this concludes the proof.
\end{proof}

Let $(X,\psi)$ be a \emph{topological flow}, that is, a homeomorphism $\psi:X\to X$ of a compact Hausdorff space  $X$. A \emph{factor} $(\pi,(Y,\phi))$ of $(X,\psi)$ is a topological flow $(Y,\phi)$ together with a continuous surjective map $\pi:X\to Y$ such that $\pi\circ \psi=\phi\circ\pi$. In \cite{BL} (see also \cite{KerLi}) it is proved that a topological flow $(X,\psi)$ admits a largest factor with zero topological entropy, called \emph{topological Pinsker factor} in analogy with the Pinsker $\sigma$-algebra for the measure theoretic entropy recalled above.

\subsection{Topological entropy on metric spaces and uniform spaces}\label{bowen-sec}

The first step in this direction was done by Bowen \cite{B}, who defined entropy for uniformly continuous selfmaps $\psi: X \to X$ of a metric space $(X,d)$.
To recall his definition, we need the following notion. For $x\in X$, $n \in \N_+$ and $\varepsilon >0$ define the {\em Bowen's ball }
\begin{equation}\label{BB}
D_n(x, \varepsilon, \psi) = \bigcap_{k=0}^{n-1}\psi^{-1}(B_\varepsilon(\psi^k(x))),
\end{equation}
where $B_\varepsilon(x)$ denotes the ball centered at $x$ of radius $\varepsilon$, given by the metric $d$. 

\smallskip
Let $\ms{K}(X)$ be the family of all compact subsets of $X$. For every $K\in\ms{K}(X)$, $n\in \N_+$ and $\varepsilon >0$ let  
$$
r\sb{n}(\varepsilon,K,\psi)=\min\{|F|:F \subseteq X\ \text{and}\ K \subseteq \bigcup_{x\in F}D_n(x, \varepsilon, \psi)\}.
$$

Moreover a subset $F\subseteq X$ is said to be an $\left (n,\varepsilon \right)$-\emph{separated set} \emph{with respect to} $\psi$, if $D_n(x, \varepsilon, \psi)\cap D_n(y, \varepsilon, \psi)=\emptyset$ for each pair of distinct points $x,y\in F$. 

\medskip
For $\varepsilon>0$, $K\in \ms{K}(X)$ and $n\in\N_+$, set 
$$s\sb{n}(\varepsilon,K,\psi)=\max\{|F|:F\subseteq K\text{ and } F \text{ is } (n,\varepsilon)\text{-separated with respect to $\psi$}\}.$$
The numbers $r\sb{n}(V,K,\varepsilon)$ and $s\sb{n}(V,K,\varepsilon)$ are finite and well defined as $K$ is compact. Now define
$$
r(\varepsilon,K,\psi)= \limsup_{n\rightarrow \infty }\frac{\log r_n\left(\varepsilon,K,\psi\right)}{n}\ \text{and}\  s(\varepsilon,K,\psi)= \limsup_{n\rightarrow \infty }\frac{\log s_n\left(\varepsilon,K,\psi\right)}{n}
$$
for every $\varepsilon>0$ and $K\in \ms{K}(X)$. Furthermore, let
$$
h_r(K,\psi)=\sup\{r(\varepsilon,K,\psi):\varepsilon>0\}\  \text{and}\  h_s(K,\psi)=\sup\{s(\varepsilon,K,\psi):\varepsilon>0\}.
$$
Then $h\sb{r}(K,\psi)=h\sb{s}(K,\psi)$ by \cite{B}, so we obtain the notion of \emph{uniform entropy} $h_U(\psi)$ of $\psi$:
\begin{equation}
h_{U}(\psi)=\sup\{h\sb{r}(K,\psi):K\in \ms{K}(X)\}= \sup\{h\sb{s}(K,\psi):K\in \ms{K}(X)\}.
\end{equation}

When the metric space $(X,d)$ is compact, every continuous selfmap  $\psi: X \to X$ is uniformly continuous, so one can define both $h_U(\psi)$ and $h_{top}(\psi)$. These two values always coincide \cite[Theorem 7.8]{Wa}. In particular the values of $h_U$ on the three Bernoulli shifts coincide with the ones of $h_{top}$. 

\begin{example}\label{computeBerny}
Let us use the above fact to compute $h_U({}_K\beta)= h_{top}({}_K\beta)=\log |K| $ of the left Bernoulli shift of $G = K^\N$, where $K$ is a finite (discrete) set of size $m > 1$ and $K^\N$ carries the compact product topology. To this end we use the metric $d$ on $G$ defined by $d(x,y) = 1/\max\{n\in\N: x_0=y_0, \ldots,  x_n =y_n \}$ for distinct $x = (x_n)$, $y=(y_n)$ in $G$.  To ease the computation, one can identify $K$ with the cyclic group $\Z(m)$, so that $G$ becomes also a group and $d$ is an invariant metric.  Then $U_m =B_{1/m}(0)$ is an open subgroup of $G$ and $D_n(0,1/m,_K\! \beta)= U_{m + n}$, so $r_n(1/m,G, \ _K\! \beta) = |K|^{n+m}$ for every $n\in\N_+$. Therefore $r(1/m, G, \ _K\! \beta) = h_U(_K\! \beta)= \log |K|$.
\end{example}

Contracting selfmaps are obviously uniformly continuous. We see now that they always have uniform entropy zero. 

\begin{proposition}\label{non-expanding} 
Let $(X, d)$ be a metric space and $\psi$ a contracting selfmap. Then $h_U(\psi)=0$.
\end{proposition}
\begin{proof} 
Since $\psi$ is contracting, $d(\psi^n(x), \psi^n(y))\leq d(x,y)$ for every $n \in \N$ and every pair of points $x,y \in X$, hence $D_n(x,\varepsilon,\psi) = B_\varepsilon(x)$ for every $n\in\N_+$. Thus $r_n(\varepsilon, K,\psi)$ and $s_n(\varepsilon, K,\psi)$ do not depend on $n\in\N_+$, so $r(\varepsilon, K,\psi) = s(\varepsilon, K,\psi) = 0$, and hence $h_U(\psi)=0$ as well. 
\end{proof}

This proposition should be compared with Remark \ref{Base}(c), where a stronger property is established in the case of endomorphisms of locally compact groups. 

\begin{example}\label{non-expandingEx} 
Let $K$ be a compact totally disconnected abelian group, such that for every prime $p$ the topologically $p$-primary component $K_p$ of $K$ is a finitely generated $\J_p$-module. Then every continuous endomorphism $\psi: K \to K$ has $h_{top}(\psi) = 0$. Indeed, as $K$ is a metric space, we can apply the above proposition if we prove that $\psi$ is contracting. Since a base of the 
neighborhoods of 0 in $K$ is the family of subgroups $\{nK: n \in \N_+\}$, it suffices to notice that $\psi(nK)\subseteq nK$ for every $n \in \N_+$.
\end{example}

The notion of uniform entropy was extended in a natural way by Hood \cite{hood} to the case of uniformly continuous selfmaps of uniform spaces. 
We give below only the definitions without going deeper in this direction, since we intend to apply this notion of entropy mainly for special uniformly continuous selfmaps of locally compact spaces, namely, continuous endomorphisms of locally compact groups (that are always uniformly continuous). This special case is considered in full detail in \S \ref{lc-sec}.  

To extend the definition of uniform entropy $h_U$ to the case of a uniformly continuous selfmap $\psi: X \to X$ of a uniform space $(X,\U)$, fix an entourage $V$ of the uniform structure, $x\in X$ and $n\in \N_+$. Define the counterpart of the Bowen's ball with respect to $U$ by $D_n(x,V,\psi)= \bigcap_{k=0}^{n-1}\psi^{-1}(V(\psi^k(x)))$. Using these open sets in place of $D_n(x, \varepsilon, \psi)$, one can define the finite  numbers $r\sb{n}(V,K,\psi)$ and $s\sb{n}(V,K,\psi)$ for every compact subset $K$ of $X$ and for every entourage $V\in\U$ as above:
$$
r(V,K,\psi)= \limsup_{n\rightarrow \infty }\frac{\log r_n\left(V,K,\psi\right)}{n}\ \text{and}\ s(V,K,\psi)= \limsup_{n\rightarrow \infty }\frac{\log s_n\left(V,K,\psi\right)}{n}. 
$$
Then 
$$
h_r(K,\psi)=\sup\{r(V,K,\psi):V\in \U\}\ \text{and}\  h_s(K,\psi)=\sup\{s(V,K,\psi):V\in \U\}
$$
coincide by \cite[Lemma 1]{hood}, so one obtains the notion of \emph{uniform entropy} $h_U(\psi)$ of $\psi$:

\begin{equation}
h\sb{U}(\psi)=\sup\{h\sb{r}(K,\psi):K\in \ms{K}(X)\}= \sup\{h\sb{s}(K,\psi):K\in \ms{K}(X)\}.
\end{equation}

It is not difficult to prove that $h_U$ is monotone for continuous images, invariant under topological conjugation, and satisfies the Logarithmic Law. It can be shown that this extended notion of entropy, in the case of a compact space and its unique compatible uniform structure, coincides with the topological entropy $h_{top}$ (see \cite{DSV} for details).

\medskip
One can use the above setting to define topological entropy for continuous selfmaps of non-compact Tichonoff spaces as follows. First instances of this construction were already discussed in \cite{DSV}. Here we adopt the more general idea to make use of functorial uniformities on a Tichonoff space (see \cite{B1}) for the definition of these entropies. This new functorial point of view distinguishes our approach from that in \cite{DSV}.  
 
Let $\mathbf{Unif}$ denote the category of uniform spaces and uniformly continuous maps. A {\em functorial uniformity} on the category $\mathbf{Tich}$ of Tichonoff spaces is nothing else but a functor $F: \mathbf{Tich} \to \mathbf{Unif}$ such that the uniform space $FX$ has the same underling set as the topological space $X$ and the uniformly continuous map $Ff$, for a continuous map $f: X \to Y$ in  $\mathbf{Tich}$, is the same set-map as $f$ (using a more rigorous categorial language, $F$ commutes with the forgetful functors  $\mathbf{Tich}\to  \mathbf{Set}$ and  $\mathbf{Unif} \to  \mathbf{Set}$).
 
For every functor $F$ as above one can define a (topological) entropy $h_F$ in  $\mathbf{Tich}$ by letting $h_F(f) = h_U(Ff)$ for every selfmap  $f: X\to X$ in  $\mathbf{Tich}$. In this way $h_F$ inherits for free the nice properties of $h_U$ (e.g., Monotonicity for continuous images, Invariance under conjugation and Logarithmic Law). 
 
\begin{example}\label{newen}  
Here are several examples of functorial uniformities. 
\begin{itemize}
\item[(a)] For $ X \in \mathbf{Tich}$ let  $\mathcal F X$ denote $X$ equipped with the universal (fine) uniform structure on $X$. Since every continuous map from $X$ to a uniform space is uniformly continuous for the universal uniform structure $\mathcal{F}$ (see, for instance, \cite[15G (5)]{GJ}), the assignment $X \mapsto \mathcal F X$ is a functorial uniformity. The entropy $h _{\mathcal F}$ was defined in \cite[Definition 2.21]{DSV}.
\item[(b)] For $ X \in \mathbf{Tich}$ let  $\mathcal T X$ denote $X$ equipped with the finest totally bounded uniform structure on $X$ (generated by all finite open covers of $X$). The assignment $X \mapsto \mathcal T X$ is a functorial uniformity (actually, this functor is a composition of $\mathcal F$ and the functor $\mathbf{Unif} \to  \mathbf{Unif} $ of the finest totally bounded uniformity for uniform spaces). The entropy $h _{\mathcal T}$ coincides with the entropy $h\sb{T\sb{f}}$ introduced by Hofer \cite{hof}.
 \item[(c)] Further examples of functorial uniformities can be found in \cite{B1}.
\end{itemize}
Another resource to define entropy for topological spaces can be the use of quasi-uniformities in place of uniformities (a rich background on functorial
quasi-uniformities can be found in Brummer's papers \cite{B-1,BQ,B0}). To this end one has to define first entropy in the category of quasi-uniform spaces. 
\end{example}

An alternative way to define entropy in $\mathbf{Tich}$ was proposed by Hofer \cite{hof} as follows. He took the functor $\mathbf{Tich} \to \mathbf{Comp}$
of the Stone-\v{C}ech compactification $X\mapsto \beta X$. Then he defined his entropy for a selfmap $f: X\to X$ in $\mathbf{Tich} $ as $h_{top}(\beta f)$, where $\beta f$ is the continuous extension of $f$ to $\beta X$.

\medskip
In concluding this subsection we mention the approach to the uniform entropy $h_U$ by means of covers proposed in \cite{hof}, similar to the one in the definition of the topological entropy $h_{top}$. The equivalence of these two approaches was pointed out  in \cite{DSV}.

\subsection{Topological entropy on compact groups}\label{cg-sec}

This subsection is dedicated to the topological entropy in the category $\mathbf{CompGrp}$ of compact groups and their endomorphisms. Yet we start by a pair of well known results of Bowen  \cite{B} in the locally compact case. He proved that an endomorphism $\psi$ of a Lie group $G$ and its derivative at the identity have the same topological entropy. We are going to combine it with the following theorem from \cite{B} which allows one to compute the topological entropy of a linear map of $\R^n$ in terms of its eigenvalues. 

\begin{theorem}\label{BF}
Let $n\in\N_+$ and $\psi:\R^n\to \R^n$ a continuous endomorphism. Then 
$$h_U(\psi)=\sum_{|\lambda_i|>1}\log|\lambda_i|,$$ 
where $\{\lambda_i:i=1,\ldots,n\}$ is the family of all eigenvalues of $\psi$.
\end{theorem}
 
We enunciate now the Kolmogorov-Sinai Formula giving the value of the topological entropy of a continuous automorphism of $\T^n$ as its Mahler measure. It can be deduced easily from the above mentioned two results of Bowen applied to the torus $\T^n$: 

\begin{theorem}[Kolmogorov-Sinai Formula]\label{KSF} 
Let $n\in\N_+$ and $\psi:\T^n\to\T^n$ a topological automorphism. Then $$h_{top}(\psi)=m(\psi).$$
\end{theorem}

The following fundamental computation of the topological entropy of topological automorphisms of $\widehat\Q^n$ was given in \cite{Y}. A more transparent proof, emphasizing better the topological aspects, can be found in \cite{LW}.

\begin{theorem}[Yuzvinski Formula]\label{YF}
Let $n\in\N_+$ and $\psi:\widehat\Q^n\to\widehat \Q^n$ a topological automorphism. Then $$h_{top}(\psi)=m(\psi).$$
\end{theorem}

The Kolmogorov-Sinai Formula can be deduced from the Yuzvinski Formula and the following Addition Theorem \ref{AT-top}, as we see in Remark \ref{YF->KSF}. 

\begin{theorem}[Addition Theorem]\emph{\cite{B,Y}} \label{AT-top}
Let $K$ be a compact group, $\psi:K\to K$ a continuous endomorphism, $N$ a closed $\psi$-invariant normal subgroup of $G$ and $\overline\psi:K/N\to K/N$ the continuous endomorphism induced by $\psi$. Then $$h_{top}(\psi)=h_{top}(\psi\restriction_N) + h_{top}(\overline\psi).$$
\end{theorem}

\begin{remark}\label{YF->KSF}
The Kolmogorov-Sinai Formula \ref{KSF} follows from Yuzvinski Formula \ref{YF} and Addition Theorem \ref{AT-top}. Indeed,  every $\psi:\T^n\to \T^n$ can be lifted to an endomorphism $\widetilde \psi: \widehat\Q^n \to \widehat\Q^n$ along the projection $q: \widehat\Q^n \to \T^n$. 

Moreover $h_{top}(\widetilde\psi\restriction_{\ker q})=0$ by Proposition \ref{non-expanding}. In fact, $\ker q=(\prod_p\J_p)^n$. A local base of the linear topology of this group is given by the family of characteristic subgroups $\{W_m:m\in\N_+\}$, where, enumerating the primes as $\{p_n:n\in\N_+\}$, for every $m\in\N_+$ we let $$W_m=(p_1^{k_1}\dots p_m^{k_m})\ker q=(p_1^{k_1}\J_{p_1}\times\ldots\times p_m^{k_m}\J_{p_m})^n\times(\prod_{l>m}\J_{p_l})^n.$$
Since each $W_m$ is a fully invariant subgroup of $\ker q$, Example \ref{non-expandingEx} (see also Proposition \ref{non-expanding}) can be applied to give $h_{top}(\widetilde\psi\restriction_{\ker q})=0$. The Addition Theorem \ref{AT-top} yields  
$$
h_{top}(\psi) = h_{top}(\psi \restriction_N)+h_{top}(\psi) = h_{top}(\widetilde \psi)=m(\widetilde \psi) =m (\psi).
$$
\end{remark}
  
\begin{example}[Bernoulli normalization] 
Let $n\in\N$. Then $h_{top}(\overline\beta_{\Z(n)})=h_{top}({}_{\Z(n)}\beta)=\log n=h_{mes}(\overline\beta_{\Z(n)})=h_{mes}({}_{\Z(n)}\beta)$, according to Example \ref{computeBerny}.
\end{example}

The next result, showing continuity for inverse limits of the topological entropy in the case of continuous endomorphisms of compact groups, is a particular case of Proposition \ref{cont-sp} for continuous selfmaps of compact spaces.
  
\begin{corollary}[Continuity] \emph{\cite{P}}\label{invlim}
Let $K$ be a compact group and $\psi:K\to K$ a continuous endomorphism. If $\{N_i:i\in I\}$ is a directed system of closed $\psi$-invariant normal subgroups of $K$ with $\bigcap_{i\in I} N_i=0$, then $K\cong \varprojlim K/N_i$ and $h_{top}(\psi)=\sup_{i\in I} h_{top}(\overline \psi_i)$, where $\overline \psi_i:K/N_i\to K/N_i$ is the continuous endomorphism induced by $\psi$.
\end{corollary}

\begin{theorem}[Uniqueness Theorem]\emph{\cite{S}}\label{UT-top}
The topological entropy $h_{top}:\mathrm{Flow}_\mathbf{CompGrp}\to \R_{\geq0}\cup\{\infty\}$ is the unique function such that:
\begin{itemize}
\item[(a)] $h_{top}$ is monotone for restriction to closed invariant subgroups;
\item[(b)] if $(K,\psi)\in\mathrm{Flow}_\mathbf{Comp}$, then $h_{top}(\psi)=h_{top}(\psi\restriction_{E_\psi(K)})$;
\item[(c)] $h_{top}$ satisfies the Logarithmic Law;
\item[(d)] if $\psi$ is an inner automorphism of $K\in\mathbf{CompGrp}$, then $h_{top}(\psi)=0$;
\item[(e)] $h_{top}$ is continuous on inverse limits;
\item[(f)] $h_{top}$ satisfies the Addition Theorem;
\item[(g)] for $F$ a finite group, $h_{top}(\overline\beta_F)=\log|F|$;
\item[(h)] $h_{top}$ satisfies the Yuzvinski Formula.
\end{itemize}
\end{theorem}

Items (a), (b) and (c) are valid for the general case of compact spaces (see \S \ref{akm-sec}), (d) is proved in Proposition \ref{inner} below (where a stronger property is established), (e), (f) and (h) were given above. The computation of (g) is in Example \ref{computeBerny}.  

\begin{remark} 
As mentioned in  Remark \ref{YF->KSF}, the Kolmogorov-Sinai Formula \ref{KSF} can be deduced from the Yuzvinski Formula \ref{YF} in the presence of the Addition Theorem \ref{AT-top}. It was proved by Stoyanov \cite{S} that the Yuzvinski Formula cannot be replaced by Kolmogorov-Sinai Formula in the Uniqueness Theorem \ref{UT-top}.
\end{remark}

\subsection{The topological e-spectrum of compact abelian groups}\label{tes-sec}

In this subsection we study the set of all possible values of the topological entropy of all continuous endomorphisms of a compact abelian group. 

\begin{definition}
For a compact abelian group $K$, the \emph{topological $e$-spectrum} of $K$ is
$$
{\bf E}_{top}(K)=\{h_{top}(\psi): \psi\in \End(K),\ \mbox{continuous}\} \subseteq \R_{\geq0}\cup \{\infty\}.
$$
\end{definition}  

The next example follows from Example \ref{exdgsz} and the Bridge Theorem \ref{BT} given in the following sections.

\begin{example}
\begin{itemize} 
\item[(a)] If $K$ is a totally disconnected compact abelian group, then $${\bf E}_{top}(K)\subseteq \{\log n: n\in \N_+\}\cup \{\infty\},$$ 
so ${\bf E}_{top}(K)\cap \R_{\geq0}$ is discrete in $\R_{\geq0}$ and $\inf ({\bf E}_{top}(K)\setminus \{0\})\geq\log 2$.
\item[(b)] If $K$ is an abelian pro-$p$-group, then ${\bf E}_{top}(K)\subseteq \{n\log p: n\in \N\}\cup \{\infty\},$  so ${\bf E}_{top}(K)\cap \R_{\geq0}$ is {\em uniformly} discrete in $\R_{\geq0}$.
\end{itemize}
\end{example}

Let $${\bf E}_{top}= \bigcup\{{\bf E}_{top}(K):K \mbox{ compact abelian group}\}.$$

\begin{problem}\label{inftop=0} 
Is $\inf({\bf E}_{top}\setminus \{0\})=0$?
\end{problem}

Note that $(\R_{\geq0}\cup\{\infty\},+)$ is a monoid, with the convention that $r+\infty=\infty$ for every $r\in\R_{\geq0}\cup\{\infty\}$.

\begin{lemma} 
The set ${\bf E}_{top}$ is a submonoid of $\R_{\geq0}\cup\{\infty\}$.
\end{lemma}
\begin{proof}
Clearly $0\in\mathbf E_{top}$ and the Weak Addition Theorem \ref{WAT-top} implies that ${\bf E}_{top}$ is a subsemigroup of $\R_{\geq0}\cup\{\infty\}$.
\end{proof}

The following fact is known. It can be deduced from its algebraic counterpart, namely, Theorem \ref{LD} below, and from the Bridge Theorem \ref{BT}.

\begin{fact}\label{LD-top}
The following conditions are equivalent:
\begin{itemize}
\item[(a)] ${\bf E}_{top}= \R_{\geq0}\cup\{\infty\}$;
\item[(b)] $\inf ({\bf E}_{top}\setminus \{0\})=0$;
\item[(c)] $\inf (\bigcup_{n\in\N_+}{\bf E}_{top}(\T^n)\setminus \{0\})=0$.
\end{itemize}
If $\inf ({\bf E}_{top}\setminus \{0\})>0$, then ${\bf E}_{top}$ is countable.
\end{fact}

As a consequence of Fact \ref{LD-top} and the Kolmogorov-Sinai Formula \ref{KSF}, we conclude that $$\mathfrak L=\inf ({\bf E}_{top}\setminus \{0\})$$ (with
$\mathfrak L$ as in (\ref{LNumber})); hence Problem \ref{inftop=0} is equivalent to Lehmer Problem \ref{L-pb}.

\medskip
We introduce two classes of compact abelian groups related to the topological e-spectrum, namely,
$${\mathfrak E}_0=\{\text{compact abelian groups}\ K\ \text{with}\ {\bf E}_{top}(K)=\{0\}\}$$
and
$${\mathfrak E}_{<\infty}=\{\text{compact abelian groups}\ K\ \text{with}\ \infty\not \in {\bf E}_{top}(K)\}.$$
Obviously, ${\mathfrak E}_0\subseteq {\mathfrak E}_{<\infty}$ and $\widehat\Q^n\in {\mathfrak E}_{<\infty}$ for every $n\in \N$ by Yuzvinski Formula \ref{YF}.

\smallskip
We recall the following known results.

\begin{theorem}     
\begin{itemize}
\item[(a)] \emph{\cite{ADS}} If $K$ is a compact abelian group and $K\in {\mathfrak E}_0$, then $w(K)\leq { \mathfrak c}$ and $\dim K=0$ (i.e., $K$ is totally disconnected).
\item[(b)] There are $2^{ \mathfrak c}$ pairwise non-isomorphic compact abelian groups $K\in {\mathfrak E}_0$ with $w(K)={ \mathfrak c}$.
\end{itemize}
\end{theorem}

Item (b) of the above theorem can be deduced from \cite{DGSZ} and the Bridge Theorem \ref{BT}.

The next example shows that for non-abelian compact groups the property in item (a) of the above theorem may fail.

\begin{example}{\cite{ADS}}
Let $K=SO_3(\R)$. Then $K$ is compact, connected, with $\dim K =3$ and every non-trivial endomorphism $\psi$ of $K$ is an internal automorphism, so $h_{top}(\psi)=0$.
\end{example}

The topological entropy can detect finite-dimensionality of connected compact abelian groups; indeed, the following result holds.

\begin{theorem}\emph{\cite{ADS}}
Let $K$ be a compact abelian group.
\begin{itemize}
\item[(a)] If $K\in {\mathfrak E}_{<\infty}$, then $\dim K<\infty$.
\item[(b)] If $\dim K<\infty$ and $K$ is connected, then $K\in {\mathfrak E}_{<\infty}$; moreover, $K\in\mathfrak E_0$ if and only if $K=0$.
\end{itemize}
\end{theorem}

The next example shows that the property in item (a) of the above theorem may fail for non-abelian connected compact groups.

\begin{example}{\cite{ADS}}
Let $K=\prod_{n=1}^\infty SO_n(\R)$, (more generally, $K=\prod_{n=1}^\infty L_n$, where $L_n$ are pairwise non-isomorphic connected compact simple Lie groups). Then $h_{top}(\psi)=0$ for every continuous endomorphism $\psi$ of $K$, yet $\dim K=\infty$.
\end{example}

In contrast with the connected case (see item (b) of Theorem \ref{Cartagena}), the class $ {\mathfrak E}_0$ may contain a plenty of 
totally disconnected compact abelian groups. Actually, the  totally disconnected compact abelian groups in $\mathfrak E_{<\infty}$ are already in 
the smaller class $\mathfrak E_{0}$, as item (a) of the next theorem says. This means that for a  totally disconnected compact abelian group $K$ one has the dichotomy 
$$
\mbox{either}\ {\bf E}_{top}(K)= \{0\}\ (\mbox{i.e., }K\in {\mathfrak E}_0)\ \mbox{or}\ \{0,\infty\} \subseteq  {\bf E}_{top}(K).
$$ 
Item (b) of the theorem characterizes the totally disconnected compact abelian groups $K$ which fail to satisfy the intermediate ``smallness" 
property ${\bf E}_{top}(K) = \{0,\infty\}$ for the topological e-spectrum: 

\begin{theorem}\label{Cartagena}{\em \cite{ADS}}
Let $K$ be an infinite totally disconnected compact abelian group. 
Then 
\begin{itemize}
\item[(a)] $K\in {\mathfrak E}_0$ if and only if $K\in  {\mathfrak E}_{<\infty}$.
\item[(b)]   If $K\not\in {\mathfrak E}_0$,  $($i.e., $\{0,\infty\} \subseteq  {\bf E}_{top}(K))$,  then the following conditions are equivalent: 
\begin{itemize}
\item[(b$_1$)]  ${\bf E}_{top}(K)\ne  \{0,\infty\}$; 
\item[(b$_2$)]  $K$ has a direct summand of the form $\Z_{p^n}^\N$ for some $n\in \N$ and some prime $p$; 
\item[(b$_3$)]  ${\bf E}_{top}(K)\supseteq \{m\log p: m\in \N\}$ for some prime $p$;
\item[(b$_4$)]  some of the Ulm-Kaplanski invariants $\alpha_{n,p}(\widehat{K})$ of its Pontryagin dual $\widehat{K}$ is infinite. 
\end{itemize}
\end{itemize}
\end{theorem}
The proof heavily relies on the Bridge Theorem \ref{BT} and properties of the algebraic entropy $\ent$ established in \cite{DGSZ}.
This is why we defer it to \S \ref{LAAAAST}, dedicated to the Bridge Theorem \ref{BT}  and its applications. 

\begin{example} 
According to item (b$_4$) of the above theorem, a typical example of a  totally disconnected compact abelian group satisfying ${\bf E}_{top}(K)= \{0,\infty\}$ is the group $K = \prod_{n=1}^\infty\Z(p^n)^{a_n}$ for a prime $p$ and for any sequence $\{a_n\}_{n\in\N}$ of natural numbers. 
\end{example}

\subsection{Topological entropy on locally compact groups}\label{lc-sec}

Let $(X,\mathcal U)$ be a locally compact uniform space endowed with a Borel measure $\mu$ and let $\psi: X \to X$ be a uniformly continuous selfmap. In the case when $X$ is metrizable, Bowen introduced in \cite[Section 2]{B} another entropy function $k$, assuming that $\mu$ satisfies a compatibility condition with respect to $\psi$. Hood \cite{hood} extended the definition of $k$ to uniformly continuous selfmaps of uniform locally compact spaces as follows. 
Call the Borel measure $\mu$ on $X$ {\em $\psi$-homogeneous}, if
\begin{enumerate}[(a)]
\item $\mu(K) < \infty$ for all compact sets $K$ in $X$;
\item $\mu(K_0) > 0$ for some compact subset $K_0$ of $X$;
\item for each $U \in \mathcal U$ there exist $V \in \mathcal U$ and $c>0 $ such that
$$
\mu(D_n(y, V , \psi)) \leq c \mu(D_n(x, U, \psi)),
$$ 
for all $n\in \N_+$ and all $x,y \in X$,
\end{enumerate}
where $D_n(x, U, \psi)$ is Bowen's ball defined in \eqref{BB}. 

It follows from (c) that for $x\in X$ the value
\begin{equation}\label{ccc}
k(\psi,U)=\limsup_{n\to\infty}-\frac{\log\mu(D_n(x, U , \psi))}{n}
\end{equation}
does not depend on the particular choice of $x$. Moreover $k(\psi,U) \leq k(\psi,V)$ whenever $U\supseteq V$. Hence we can define the {\em rate of decay} of the measure of Bowen's ball as
$$
k(\mu, \psi)=\lim_{U\in \mathcal U}k(\psi,U).
$$
Bowen proved that 
\begin{equation}\label{Bowen_Formula}
h_s(K,\psi) = k(\mu, \psi)
\end{equation}
for every compact subset $K$ of $X$ with $\mu(K)>0$, in particular, $h_U(\psi) = k(\mu, \psi)$ in view of (b). 

Note that the more precise equality \eqref{Bowen_Formula} implies that \emph{$h_s(K,\psi) $ does not depend on the compact set $K$ as far as $\mu(K) > 0$}. This allows for a simpler computation of $h_U$ by choosing an appropriate compact set with $\mu(K) > 0$.

\smallskip
In what follows we discuss the entropy $k$ for a locally compact group $G$, endowed with a right Haar measure $\mu$ and considered in its right uniform structure. Let $\C(G)$ be the family of all compact neighborhoods of $e_G$ in $G$.
For $n\in\N_+$, the {\em $n$-th $\psi$-cotrajectory} of a fixed $U\in \C(G)$ is
\begin{equation}\label{cotraj}
C_n(\psi,U)=U\cap \psi^{-1}(U)\cap \dots\cap \psi^{-n+1}(U).
\end{equation}
For $x\in G$, $n\in \N_+$ and $U \in \C(G)$, one has $D_n(x,U,\psi)=C_n(\psi,U)x$ (here we write $U$ in $D_n(x,U,\psi)$ in place of the entourage $W_U=\{(y,x)\in G \times G: yx^{-1}\in U)\}$ in the right uniformity, determined by $U$). Therefore the right invariance of $\mu$ yields that $\mu$ is $\psi$-homogeneous for every continuous endomorphism $\psi$ of $G$ and $\mu(D_n(x,U,\psi))= \mu (C_n(\psi,U))$. Now $k(\psi,U)$, with $U\in\C(G)$ (generalizing \eqref{ccc} from the metric case) is defined by
\begin{equation}\label{h_mu(U)} 
k(\psi,U)=\limsup_{n\to\infty}-\frac{\log\mu(C_n(\psi,U))}{n},
\end{equation}
and
\begin{equation}\label{h_mu} 
k(\psi)=\sup\{k(\psi,U):U\in \C(G)\}.
\end{equation}

As we mentioned, a reformulation in terms of locally compact groups of the results from \cite{hood} gives, for a locally compact group $G$ and a continuous endomorphism $\psi:G\to G$,
$$k(\psi)=h_U(\psi).$$
In view of this equality we use the symbol $k$ also to denote $h_U$, and we call it topological entropy.

\medskip
The following easy remark simplifies the computation of the supremum \eqref{h_mu} in many cases.

\begin{remark}\label{Base}
Let $G$ be a locally compact group, $\mu$ a right Haar measure on $G$ and $\psi:G\to G$ a continuous endomorphism.
\begin{itemize}
   \item[(a)] To see that the value of $k(\psi,U)$ computed in \eqref{h_mu(U)} does not depend on the choice of the right Haar measure $\mu$ on $G$ it suffices to recall that if $\mu'$ is another right Haar measure on $G$, then there exists a constant $c>0$ such that $\mu'= c\mu$, so the limit in (\ref{h_mu(U)}) does not change. 
   \item[(b)] For $U,U'\in \C(G)$, with $U \subseteq U' $, one has $C_n(\psi,U) \subseteq C_n(\psi,U')$ for every $n\in\N_+$; so $k(\psi,U)\geq k(\psi,U')$. In this sense $k(\psi,U)$ is monotone with respect to $U$. Therefore in \eqref{h_mu} one can take the neighborhoods $U$ from a base of $\C(G)$.  
   \item[(c)] Using the formula (\ref{h_mu(U)}) one can obtain a short proof of the fact that a continuous endomorphism  $\psi$ of $G$, such that $\psi(U) \subseteq U$ for all members 
   $U$ of a base of $\C(G)$ has zero topological entropy (compare with Proposition \ref{non-expanding}). Indeed, the hypothesis $\psi(U) \subseteq U$ leads to $C_n(\psi,U)= U$, so obviously $k(\psi,U) = 0$. 
\end{itemize}
\end{remark}

It is well known that when the locally compact group $G$ is totally disconnected, then $\C(G)$ has a base of clopen compact subgroups (see \cite{HR}). Therefore in view of Remark \ref{Base}(b) one can take in \eqref{h_mu} clopen compact subgroups $U$ in this case. More precisely, for a topological group $G$, let $\mathcal B(G)$ be the family of all open compact subgroups of $G$. Then (b) entails the following useful property: 

\begin{lemma}\label{B(G)}
Let $G$ be a totally disconnected locally compact group. Then 
$$k(\psi)=\sup\{k(\psi,U):U\in \mathcal B(G)\}.$$
\end{lemma}

This lemma allows one to eliminate the Haar measure $\mu$ completely and see that in \eqref{h_mu(U)} one has actually a {\em limit}:

\begin{proposition}\label{no-mu}
Let $G$ be a totally disconnected locally compact group and let $U\in\mathcal B(G)$. Then 
$$k(\psi,U)=\lim_{n\to\infty}\frac{\log [U: C_n(\psi,U)]}{n}.$$
If $G$ is compact, then $$k(\psi,U)=\lim_{n\to\infty}\frac{\log [G: C_n(\psi,U)]}{n}.$$
\end{proposition}
\begin{proof}
Let $u= \mu(U)>0$ and $n\in\N_+$. The open subgroup $C_n(\psi,U)$ of the compact subgroup $U$ of $G$ has finite index $[U:C_n(\psi,U)]$, so $\mu(C_n(\psi,U)) ={u}[U:C_n(\psi,U)]^{-1}$. Therefore 
$$k(\psi,U)=\limsup_{n\to\infty}\frac{\log [U: C_n(\psi,U) ]-\log u }{n}=\limsup_{n\to\infty}\frac{\log [U: C_n(\psi,U) ] }{n},$$
as $\lim_{n\to\infty} \frac{\log u}{n} = 0$. 

Assume now that $G$ is compact. Then $C_n(\psi,U)$ has finite index in $G$. From $[G: C_n(\psi,U)]= [G:U][U: C_n(\psi,U) ]$ for every $n\in\N_+$, and the equality $\lim_{n\to\infty}\frac{\log [G:U]}{n}=0$ one can easily deduce
$$
\lim_{n\to\infty}\frac{\log [G: C_n(\psi,U)]}{n}= \limsup_{n\to\infty}\frac{\log [G:U][U: C_n(\psi,U) ]}{n}=\limsup_{n\to\infty}\frac{\log [U: C_n(\psi,U) ]}{n}=k(\psi,U).$$
As both sequences $\log [G:C_n(\psi,U) ]=\log [G:U][U: C_n(\psi,U) ]$ and $\log[U: C_n(\psi,U)]$ are subadditive, by Fekete Lemma \ref{fekete} the two $\limsup$ are actually $\lim$.
\end{proof} 

A topological group $G$ is said to be {\em SIN} (small invariant neighborhoods), if $G$ has a local base at $e_G$ consisting of invariant neighborhoods $U$ (i.e., $U^x= U$ for all $x\in G$).  Compact groups are known to be SIN. 

\begin{proposition}\label{inner}
Let $G$ be a SIN locally compact group. Every conjugation $\f_a: x\mapsto x^a$ of $G$ has zero topological entropy. In particular every conjugation of a compact group has zero topological entropy. 
\end{proposition}
\begin{proof} 
By Lemma \ref{B(G)} it suffices to note that $k(\f_a,U) = 0$ for every compact invariant neighborhood $U$ of $e_G$ in $G$. 
\end{proof}

The next examples show that this property fails in locally compact groups that are not SIN (see items (a), (c) and (d)), but may be available even if the group is not SIN (see item (b)). 

\begin{example}\label{exlctop}
\begin{itemize}
\item[(a)] Let $p$ be a prime and let $\Q_p$ denote the (locally compact) field of $p$-adic integers. Denote by $\xi$ the automorphism of $
\Q_p$ defined by $\xi(x) = \frac{x}{p}$ for every $x\in\Q_p$. Let $G$ be the semi-direct product of $\langle \xi \rangle$ and $\Q_p$ with respect to the action of $\langle \xi \rangle$ on $\Q_p$ determined by $\xi$. Denote by $N$ the open normal subgroup $\{0\} \times \Q_p$ of $G$. Clearly, the conjugation $\f_\xi$ by $\xi$ in $G$ has $N$ as $\xi$-invariant subgroup and $\f_\xi\restriction_N$ coincides (up to the isomorphism $\eta: N \cong \Q_p$) with the automorphism $\xi$ of $\Q_p$. Therefore the $(n+1)$-th $\phi_\xi$-trajectory of the compact open subgroup $U= \eta^{-1}(\mathbb J_p)$ coincides exactly with $p^nU$, hence $k(\f_\xi,U)= \log p$. Therefore $k(\f_\xi) = \log p > 0$. 
\item[(b)] The group $G$ of item (a) is center-free. Now consider a nilpotent group, namely the Heisenberg group 
$$\mathbb H_{\Q_p}=\begin{pmatrix} 
1 & \Q_p & \Q_p \\
0 & 1 & \Q_p \\
0 & 0 & 1
\end{pmatrix}.$$
over the field $\Q_p$ of $p$-adic integers. 
Consider the compact open subgroup $U=\begin{pmatrix} 
1 & \J_p & \J_p \\
0 & 1 & \J_p \\
0 & 0 & 1
\end{pmatrix}$ of $\mathbb H_{\Q_p}$ and fix the element $\alpha_n =\begin{pmatrix} 1 & p^{-n} & 0 \\
0 & 1 & 0 \\
0 & 0 & 1
\end{pmatrix}\in \mathbb H_{\Q_p}$ for $n \in \N$. Then $U\cap \f_{\alpha_n}(U)  =\begin{pmatrix} 
1 & \J_p & \J_p \\
0 & 1 & p^n\J_p \\
0 & 0 & 1
\end{pmatrix}$. Therefore 
$$
\bigcap_{n=1}^\infty \f_{\alpha_n} (U) = \begin{pmatrix} 
1 & \J_p & \J_p \\
0 & 1 & 0 \\
0 & 0 & 1
\end{pmatrix}
$$ 
has empty interior, so $G$ is not SIN. We see now that in spite of this, the topological entropy of the conjugations in $\mathbb H_{\Q_p}$ are zero, i.e., Proposition \ref{inner} cannot be inverted even for totally disconnected locally compact nilpotent groups of class two. 

Let $\gamma = \begin{pmatrix} 1 & t & r \\
0 & 1 & s \\
0 & 0 & 1
\end{pmatrix}\in \mathbb H_{\Q_p}$. Then there exists $m\in \N$ such that $p^mt, p^ms \in \J_p$, consequently $p^mjt, p^mjs \in \J_p$ for every $j\in \Z$. Hence 
for every 
$$\xi =  \begin{pmatrix} 
1 & a & b \\
0 & 1 & c \\
0 & 0 & 1
\end{pmatrix}
\in V =\begin{pmatrix}
1 & p^m \J_p & \J_p \\
0 & 1 & p^m\J_p \\
0 & 0 & 1
\end{pmatrix}
$$
one has 
$$\f_\gamma(\xi) =  \begin{pmatrix} 
1 & a & b + tc - sa \\
0 & 1 & c \\
0 & 0 & 1
\end{pmatrix} \in U.
$$
Analogously $(\f_\gamma)^j(V)\subseteq U$ for all $j \in\N$. Hence  $V \subseteq C_n(\f_\gamma, U)$, and consequently  
$[U:C_n(\f_\gamma, U)]\leq [U:V]=p^m$ for all $n \in \N_+$. 
Therefore  $k( \f_\gamma) = 0$. 

\item[(c)] Let 
$$
G=\begin{pmatrix}
\R_{>0} & \R  \\
0  & 1
\end{pmatrix}
$$  
be the group of all real-entry $2 \times 2$ matrices of the form $\alpha=\begin{pmatrix} 
x & y \\
0 & 1
\end{pmatrix}$ with $x>0$, and let $\xi =\begin{pmatrix}
a & 0  \\
0  & 1 \end{pmatrix}$, with $a> 1$. Let us see that $\f_\xi$ has  topological entropy $\log a$. To this end fix a compact neighborhood $U$ of the identity of the form $U=\begin{pmatrix}
I & J  \\
0  & 1
\end{pmatrix},$ where $I = [\frac{1}{2}, \frac{3}{2}]$ is a compact neighborhood of 1 in $\R_{>0}$ and $J= [-1, 1]$ is a compact symmetric neighborhood of $0$ in $\R$. Since $\f_\xi(U) = \begin{pmatrix}
I & a^{-1}J  \\
0  & 1 \end{pmatrix},$ we can easily deduce that $C_n(\f_\xi, U) = \begin{pmatrix}
I & a^{-n+1}J  \\
0  & 1
\end{pmatrix}.$ 
Write briefly $C_n= C_n(\f_\xi, U) $ and let $f_{C_n}(x,y)$ be the characteristic function of the (rectangular) subset $C_n \subseteq \R^2$. 
Denoting by $\mu_r$ and  $\mu_l$ the right and left Haar measures in $G$ (see \cite[(15.17)(g)]{HR}), one has
\begin{equation}\label{intergrale1}
\mu_r(C_n) = \int_\infty^\infty\int_\infty^\infty \frac{f_{C_n}(x,y)}{x^2}dx\ dy = 2a^{-n+1}\int_{1/2}^{3/2}\frac{dx}{x^2}=\frac{8a^{-n+1}}{3},
\end{equation}
and 
\begin{equation}\label{intergrale2}
\mu_l(C_n) = \int_\infty^\infty\int_\infty^\infty \frac{f_{C_n}(x,y)}{|x|}dx\ dy = 2a^{-n+1}\int_{1/2}^{3/2}\frac{dx}{|x|}=2a^{-n+1}\log 3.
\end{equation}
Using (\ref{h_mu(U)}) and (\ref{intergrale1}), we compute 
$$
k(\f_\xi,U) = \lim_{n\to\infty}  -\frac{\log (\frac{8a^{-n+1}}{3})}{n}=\lim_{n\to\infty}  \frac{(n-1) \log a - \log 8 + \log 3}{n} = \log a.
$$
It is easy to see that using (\ref{intergrale2}) instead of (\ref{intergrale1}) gives the same limit $ \log a$. We can deduce now that $k( \f_\xi) = \log a > 0$.  This gives an example of a connected locally compact group with conjugations of positive topological entropy.     
\end{itemize}
\end{example}

Motivated by item (b) above, one can ask the following

\begin{question}\label{conj1}
Is it true that all conjugations of locally compact nilpotent groups have zero topological entropy? 
\end{question}

For a locally compact field $K$ let $K^*$ be the multiplicative group of all non-zero elements of $K$. Provide the matrix group
$G_K=\begin{pmatrix}
K^{*} & K  \\
0  & 1
\end{pmatrix}$
with the topology inherited from $K^4$. Then $G_K$ is a closed subgroup of the locally compact group $GL_2(K)$. 
With a slight modification of the arguments of (a) and (c), one can prove that both groups $G_{\Q_p}$ and  $G_{\R}$
have a conjugation of a positive topological entropy. This motivates our next: 

\begin{question}\label{conj2} Let $K$ be a non-discrete  locally compact field. 
\begin{itemize}
\item[(a)] Is it true that the group $G_K$ has a conjugation with positive topological entropy?
\item[(b)] Is it true that all conjugations of $G_K$ have finite topological entropy?
\item[(c)] Are (a) and/or (b) true for the larger groups $GL_n(K)$ with $n\geq 2$?
\item[(d)] Does there exist a locally compact group with a  conjugation of infinite topological entropy? Can such a group be chosen connected (totally disconnected)?
\end{itemize}
\end{question}

As the Addition Theorem is one of the fundamental properties of every entropy function, we start considering its weaker form for the topological entropy in this context of continuous endomorphisms of locally compact groups; it holds with the additional hypothesis on the group to be totally disconnected.

\begin{theorem}[Weak Addition Theorem]\label{neeeew} 
Let $G_1,G_2$ be totally disconnected locally compact groups and $\psi_i:G_i\to G_i$ continuous endomorphisms for $i=1,2$. Then $k(\psi_1\times\psi_2)=k(\psi_1)+k(\psi_2)$.
\end{theorem}
\begin{proof} {\sl We give only a sketch of a proof.}
Consider first the case when $G_1,G_2$ are metrizable. As mentioned before, the $\limsup$ defining $k(\psi_i,U_i)$ for an open compact subgroup $U_i$ of $G_i$ is a limit. Then $k(\psi_1 \times \psi_2,U_1\times U_2) = k(\psi_1,U_1) + k(\psi_2,U_2)$. Since both $k(\psi_i.-)$, $i=1,2$, are monotone, 
letting $U_1$ and $U_2$ run over a countable decreasing chain of neighborhoods of the neutral element in $G_i$ we are done. 
To prove the non-metrizable case, it suffices to present the groups as inverse limits of metrizable ones.
\end{proof}

\begin{question}\label{newQ} 
Does the Weak Addition Theorem \ref{neeeew} extend to all locally compact groups?   
\end{question}

\newpage
\section{Algebraic entropy}\label{alg-sec}

\subsection{Definition and basic properties}\label{def-sec}

The notion of algebraic entropy was sketched for the first time by Adler, Konheim and McAndrew \cite{AKM} for endomorphisms of abelian groups, even if the natural setting for their notion is that of endomorphisms of torsion abelian groups. This algebraic entropy, called $\ent$, was later studied by Weiss \cite{We}, who gave its basic properties and the relation with the topological entropy using the Pontryagin duality. Recently, it was deeply investigated in \cite{DGSZ}.

\smallskip
A different definition of algebraic entropy was given by Peters \cite{P} for automorphisms of abelian groups. It was appropriately modified in \cite{DG} to fit for all endomorphisms of abelian groups. The latter definition of algebraic entropy coincides with the one by Weiss on endomorphisms of torsion abelian groups and can be extended to endomorphisms of arbitrary groups, as we do in this section.

\smallskip
Let $G$ be a group and $\phi:G\to G$ an endomorphism. For a finite subset $F$ of $G$ and $n\in\N_+$, the \emph{$n$-th $\phi$-trajectory} of $F$ is
$$T_n(\phi,F)=F\cdot \phi(F)\cdot\ldots\cdot\phi^{n-1}(F).$$
The \emph{algebraic entropy of $\phi$ with respect to $F$} is $$H_{alg}(\phi,F)=\limsup_{n\to\infty}\frac{\log|T_n(\phi,F)|}{n},$$ and the \emph{algebraic entropy} of $\phi$ is $$h_{alg}(\phi)=\sup\{H_{alg}(\phi,F):F\in[G]^{<\omega}\}.$$

First of all we show that the $\limsup$ defining $H_{alg}(\phi,F)$ is a limit. We start proving that $\{\log|T_n(\phi,F)|\}_{n\in\N_+}$ is a subadditive sequence in order to apply Fekete Lemma \ref{fekete}.

\begin{lemma}\label{subadd}
Let $G$ be a group, $\phi:G\to G$ an endomorphism and $F\in[G]^{<\omega}$. For $n\in\N_+$, let $c_n=\log|T_n(\phi,F)|$. Then $c_{n+m}\leq c_n+c_m$ for every $n,m\in\N_+$. 

Therefore the limit $H_{alg}(\phi,F)=\lim_{n\to\infty}\frac{\log|T_n(\phi,F)|}{n}$ exists and $H_{alg}(\phi,F)=\inf_{n\in\N_+}\frac{\log|T_n(\phi,F)|}{n}$.
\end{lemma}
\begin{proof}
By definition 
\begin{align*}
T_{n+m}(\phi,F) &=F\cdot\phi(F)\cdot\ldots\cdot\phi^{n-1}(F)\cdot\phi^n(F)\cdot\ldots\cdot\phi^{n+m-1}(F)\\
&=T_n(\phi,F)\cdot\phi^n(T_m(\phi,F)).
\end{align*}
Consequently,
\begin{align*}
c_{n+m} &=\log|T_{n+m}(\phi,F)|\\
&\leq \log(|T_n(\phi,F)||T_m(\phi,F)|)\\
&=\log|T_n(\phi,F)|+\log|T_m(\phi,F)|\\
&=c_n+c_m.
\end{align*}
This proves that the sequence $\{c_n\}_{n\in\N_+}$ is subadditive. By Fekete Lemma \ref{fekete} the sequence $\{\frac{c_n}{n}\}_{n\in\N_+}$ has limit and $\lim_{n\to\infty}\frac{c_n}{n}=\inf_{n\in\N_+}\frac{c_n}{n}$.
\end{proof}

\begin{remark}
For a group $G$ and $\phi:G\to G$ an endomorphism, a finite subset $F$ of $G$ and $n\in\N_+$, one could define the $n$-th $\phi$-trajectory of $F$ also as $$T_n^{\#}(\phi,F)=\phi^{n-1}(F)\cdot\ldots\cdot\phi(F)\cdot F,$$ changing the order of the factors with respect to the above definition.
With these trajectories it is possible to define another entropy function letting $$H_{alg}^{\#}(\phi,F)=\limsup_{n\to\infty}\frac{\log|T_n^{\#}(\phi,F)|}{n},$$ and $$h_{alg}^{\#}(\phi)=\sup\{H_{alg}^{\#}(\phi,F):F\in[G]^{<\omega}\}.$$
In the same way as for the algebraic entropy, one can see that the $\limsup$ defining $H_{alg}^{\#}(\phi,F)$ is a limit, as $\{\log|T_n^{\#}(\phi,F)|\}_{n\in\N_+}$ is a subadditive sequence and so Fekete Lemma \ref{fekete} applies.

Wo do not go further in this direction, but the first question one should answer before studying this entropy function is the following.
Do $h_{alg}^{\#}$ and $h_{alg}$ coincide? Obviously they coincide on the identity map and on abelian groups.
\end{remark}

\begin{example}
As one would expect, it is clear by the definition that $h_{alg}(e_G)=0$  for every group $G$, where $e_G:G\to G$ denotes the endomorphism of $G$ which is identically $e_G$.
\end{example}

The next lemma shows that the function $H_{alg}(\phi,-)$ is monotone, so the computation of the algebraic entropy can be carried out by using only finite subsets containing the identity of the group.

\begin{lemma}\label{0inF}
Let $G$ be a group and $\phi:G\to G$ an endomorphism. Then:
\begin{itemize}
\item[(a)] $H_{alg}(\phi,F)\leq H_{alg}(\phi,F')$ for every $F,F'\in[G]^{<\omega}$ with $F\subseteq F'$.
\item[(b)] Hence $$h_{alg}(\phi)=\sup\{H_{alg}(\phi,F):F\in[G]^{<\omega}, e_G\in F\}.$$
\end{itemize}
\end{lemma}

\begin{remark}\label{locfin}
\begin{itemize}
\item[(a)] A group $G$ is locally finite if every finite subset of $G$ generates a finite subgroup (i.e., every finite subset of $G$ is contained in a finite subgroup of $G$). Every locally finite group is obviously torsion, while the converse holds true under the hypothesis that the group is abelian (but the solution of Burnside's problem shows that even groups of finite exponent fail to be locally finite). 
\item[(b)] Let $G$ be a locally finite group and $\phi:G\to G$ an endomorphism. The monotonicity of the function $H_{alg}(\phi,-)$ from Lemma \ref{0inF}(a) implies that in this case $h_{alg}(\phi)=\sup\{H_{alg}(\phi,F):F\ \text{finite subgroup of}\ G\}$.
\item[(c)] If $G$ is a locally finite group, then $h_{alg}(id_G)=0$. Moreover, if $\phi:G\to G$ is a locally quasi-periodic endomorphism (i.e., for every $x\in G$ there exist $n\neq m$ in $G$ such that $\phi^n(x)=\phi^m(x)$), then $h_{alg}(\phi)=0$.
\item[(d)] The definition of algebraic entropy by Adler, Konheim and McAndrew \cite{AKM} was given for endomorphisms $\phi$ of abelian groups $G$ exactly in this way, that is, using only finite subgroups: 
$$
\ent(\phi)=\sup\{H_{alg}(\phi,F):F\ \text{finite subgroup of}\ G\}.
$$
Clearly, $\ent(\phi)=\ent(\phi\restriction_{t(G)})$; therefore, as we underlined above, this algebraic entropy works for endomorphisms of torsion abelian groups. Moreover the equality $\ent(\phi)=h_{alg}(\phi\restriction_{t(G)})$ shows that $h_{alg}$ can be seen as a generalization of $\ent$ out of the realm of torsion abelian groups.
\end{itemize}
\end{remark}

We state now all the typical basic properties considered for the algebraic entropy. The proofs are the same given for the abelian case in \cite{DG}, as they work in the same way in the general case. 

\begin{lemma}[Monotonicity for subgroups and quotients]\label{restriction_quotient}
Let $G$ be a group, $\phi:G\to G$ an endomorphism and $H$ a $\phi$-invariant subgroup of $G$.
\begin{itemize}
\item[(a)] Then $h_{alg}(\f)\geq h_{alg}(\f\restriction_H)$.
\item[(b)] If $H$ is normal and $\overline\phi:G/H\to G/H$ is the endomorphism induced by $\phi$, then $h_{alg}(\f)\geq h_{alg}(\overline{\f})$.
\end{itemize}
\end{lemma}
\begin{proof}
(a) For every $F\in[H]^{<\omega}$, obviously $H_{alg}(\phi\restriction_H,F)=H_{alg}(\phi,F)$, so $H_{alg}(\phi\restriction_H, F)\leq h(\phi)$. Hence $h_{alg}(\phi\restriction_H)\leq h_{alg}({\phi})$.

\smallskip
(b) Now assume that $H$ is normal and $F\in[G/H]^{<\omega}$ and $F=\pi(F_0)$ for some $F_0\in[G]^{<\omega}$, where $\pi:G\to G/H$ is the canonical projection. Then $\pi(T_n(\phi,F_0))= T_n(\overline{\phi},F)$ for every $n\in\N_+$. Therefore $h_{alg}(\phi)\geq H_{alg}(\phi,F_0)\geq H_{alg}(\overline{\phi}, F)$ and by the arbitrariness of $F$ this proves $h_{alg}(\phi)\geq h_{alg}(\overline{\phi})$.
\end{proof}

\begin{lemma}[Invariance under conjugation]\label{conj} Let $G$ and $H$ be groups, and $\phi:G\to G$ and $\psi:H\to H$ endomorphisms. If there exists an isomorphism $\xi : G \to H$, then $h_{alg}(\phi)=h_{alg}(\xi\phi\xi^{-1})$.
\end{lemma}
\begin{proof}
Let $\psi=\xi\phi\xi^{-1}$.
For $F\in[G]^{<\omega}$ and $n\in\N_+$, $T_n(\psi,\xi(F))=\xi(F)\cdot\xi(\phi(F))\cdot\ldots\cdot(\phi^{n-1}(F))$. Since $\xi$ is an isomorphism, $|T_n(\phi,F)|=|T_n(\psi,\xi(F))|$, and so $H_{alg}(\phi,F)=H_{alg}(\psi,\xi(F))$. This proves that $h_{alg}(\phi)=h_{alg}(\psi)$.
\end{proof}

\begin{proposition}[Logarithmic Law]\label{LL}
Let $G$ be a group and $\phi:G\to G$ an endomorphism. For every $k\in\N_+$, $h(\phi^k) = k h(\phi)$. If $\phi$ is an automorphism, then $h_{alg}(\phi^k) = |k|h_{alg}(\phi)$ for every $k\in\Z$, $k\neq 0$.
\end{proposition}
\begin{proof} Let $F\in[G]^{<\omega}$ and $n\in \N_+$. Let $F_1=T_k(\phi,F)$ and note that $T_n(\phi^k,F_1)=T_{kn}(\phi,F)$. Then 
$$\frac{h_{alg}(\phi^k)}{k}\geq \frac{H_{alg}(\phi^k,F_1)}{k}=\lim_{n\to\infty}\frac{\log|T_n(\phi^k,F_1)|}{kn} =\lim_{n\to\infty}\frac{\log|T_{kn}(\phi,F)|}{kn}=H_{alg}(\phi,F).$$ 
We can conclude that $h(\phi^k)\geq k h(\phi)$. 
By the previous chain of equalities, $$H_{alg}(\phi^k,F)\leq H_{alg}(\phi^k,F_1)=kH_{alg}(\phi,F)\leq k h_{alg}(\phi);$$
so we have also the converse inequality $h_{alg}(\phi^k)\leq k h_{alg}(\phi)$, hence equality holds.

Now assume that $\phi$ is an automorphism. It suffices to prove that $h_{alg}(\phi^{-1})=h_{alg}(\phi)$. Let $F\in[G]^{<\omega}$ and $n\in \N_+$. Note that $T_n(\phi^{-1},F)=\phi^{-n+1}(T_n(\phi,F))$; in particular, $|T_n(\phi^{-1}, F)|=|T_n(\phi,F)|$, as $\phi$ is an automorphism. This yields $H_{alg}(\phi^{-1},F)=H_{alg}(\phi,F)$, hence $h_{alg}(\phi^{-1})=h_{alg}(\phi)$.
\end{proof}

\begin{corollary}\label{0<->0}
Let $G$ be a group and $\phi:G\to G$ an endomorphism. Then:
\begin{itemize}
\item[(a)]$h_{alg}(\f)=0$ if and only if $h_{alg}(\f^k)=0$ for some $k\in\N_+$;
\item[(b)]$h_{alg}(\f)=\infty$ if and only if $h_{alg}(\f^k)=\infty$ for some $k\in\N_+$;
\item[(c)] if $\f^k = \f$ for some $k>1$, then either  $h_{alg}(\f)=0$ or $h_{alg}(\f)=\infty$; in particular, either  $h_{alg}(id_G)=0$ or $h_{alg}(id_G)=\infty$.  
\end{itemize}
\end{corollary}

\begin{proposition}[Continuity]\label{dirlim}
Let $G$ be a group and $\phi:G\to G$ an endomorphism. If $G$ is direct limit of $\phi$-invariant subgroups $\{G_i:i\in I\}$, then $h_{alg}(\phi)=\sup_{i\in I}h_{alg}(\phi\restriction_{G_i})$.
\end{proposition}

\begin{proof}
By Lemma \ref{restriction_quotient}, $h_{alg}(\f)\geq h_{alg}(\f\restriction_{G_i})$ for every $i\in I$ and so $h_{alg}(\f)\geq\sup_{i\in I}h_{alg}(\f\restriction_{G_i})$.

To prove the converse inequality let $F\in[G]^{<\omega}$. Since $G=\varinjlim\{G_i:i\in I\}$ and $\{G_i:i\in I\}$ is a directed family, there exists $j\in I$ such that $F\subseteq G_j$. Then $H_{alg}(\f,F)=H_{alg}(\f\restriction_{G_j},F)\leq h_{alg}(\f\restriction_{G_j})$. This proves that $h_{alg}(\f)\leq\sup_{i\in I}h_{alg}(\f\restriction_{G_i})$.
\end{proof}

%
%
%
%

\begin{theorem}[Weak Addition Theorem]\label{WAT}
Let $G_1$ and $G_2$ be groups, and $\phi_i:G_i\to G_i$ an endomorphism for $i=1,2$. Then $h_{alg}(\phi_1\times\phi_2)=h_{alg}(\phi_1)+h_{alg}(\phi_2)$.
\end{theorem}

\begin{proof}
Fix $F_i\in[G_i]^{<\omega}$, for $i=1,2$. Then, for every $n\in\N_+$,
$$T_n(\phi,F_1\times F_2)=T_n(\phi_1,F_1)\times T_n(\phi_2, F_2),$$ 
and so 
\begin{equation}\label{times-eq}
h_{alg}(\phi) \geq H_{alg}(\phi,F_1\times F_2)=H_{alg}(\phi_1,F_1)+H_{alg}(\phi_2,F_2).
\end{equation} 
Consequently, $h_{alg}(\phi)\geq h_{alg}(\phi_1)+h_{alg}(\phi_2)$.
Since every $F\in[G]^{<\omega}$ is contained in some $F_1\times F_2$, for $F_i\in[G_i]^{<\omega}$, $i=1,2$, so $H_{alg}(\phi,F)\leq H_{alg}(\phi,F_1\times F_2)$, hence \eqref{times-eq} proves also that $h_{alg}(\phi)\leq h_{alg}(\phi_1)+h_{alg}(\phi_2)$.
\end{proof}

\subsection{Examples, Normalizations, Addition Theorem and Uniqueness}\label{ab-sec}

The right Bernoulli shift is a fundamental example in the theory of algebraic entropy, in the same way as the left Bernoulli shift is in the theory of topological and measure entropy. We calculate its algebraic entropy in the next example.

\begin{example}[Bernoulli normalization]\label{beta}
For any group $K$, $$h_{alg}(\beta_K^\oplus)=\log|K|,$$ with the usual convention that $\log|K|=\infty$, if $|K|$ is infinite.

To compute this value, write $G=K^{(\N)}=\bigoplus_{i\in\N}K_i$, where $K_i=K$ for every $i\in\N$. To simplify the notations, in this example denote $\beta_K^\oplus$ simply by $\beta_K$.
Let $F\in [K_0]^{<\omega}$. For $n\in\N_+$, $T_n(\beta_K,F)=F\times\beta_K(F)\times\ldots\times \beta^{n-1}_K(F)=F^{n}\subseteq\bigoplus_{i=1}^nK_n$. Therefore $$H_{alg}(\beta_K,F)=\lim_{n\to\infty}\frac{\log|F|^n}{n}=\log|F|,$$ and so $$h_{alg}(\beta_K)\geq \sup\{H_{alg}(\beta_K,F):F\in[K_0]^{<\omega}\}=\sup\{\log|F|:F\in[K_0]^{<\omega}\}=\log|K|.$$
If $K$ is infinite, then $\log|K|=\infty$ and the wanted equality holds. So assume $|K|$ to be finite.
Every $F\in [G]^{<\omega}$ is contained in some finite subgroup of the form $F_k=K_1\oplus\ldots\oplus K_k$ with $k\in\N$, and so $$h_{alg}(\beta_K)=\sup\{H_{alg}(\beta_K,F_k):k\in\N\}.$$
Since $T_n(\beta_K,F_k)=K_0\oplus \ldots\oplus K_{k+n-1}=F_{k+n-1}$, we have $|T_n(\beta_K,F_k)|=|K|^{k+n}$ and so
$$H_{alg}(\beta_K,F_k)=\lim_{n\to\infty}\frac{\log|K|^{k+n}}{n}=\log|K|.$$
\end{example}

In item (a) of the next example we see that the identity map of an abelian group has zero algebraic entropy, as one would expect. On the other hand, the algebraic entropy of the identity map of a non-abelian group can be strictly positive (i.e., infinite -- see Corollary \ref{0<->0}(c)), as we see in Proposition \ref{h(id)>0} below.

\begin{example}\label{id-abelian}
\begin{itemize}
  \item[(a)] If $G$ is an abelian group, then $h_{alg}(\mbox{id}_G)=0$; indeed, $H(\mbox{id}_G,F)=0$ for every $F\in [G]^{<\omega}$, since $|T_n(id_G, F)|\leq (n+1)^{|F|}$ for every $n\in\N_+$. 
 \item[(b)]  If $G$ is a torsion abelian group and $\phi:G\to G$ an endomorphism, then $h_{alg}(\f)=0$ if and only if every element of $G$ is contained in a $\phi$-invariant finite subgroup \cite{DGSZ}. 
  \item[(c)] For $k\in\N_+$ the endomorphism $\varphi_k:\Z\to \Z$ given by the multiplication by $k$ has algebraic entropy $h_{alg}(\varphi_k)=\log k$. 
  More generally, assume that $G$ is an arbitrary group and 
  $\f: G \to G$ is a power endomorphism, i.e., all subgroups of $G$ are $\f$-invariant. Then for every $x\ne e_G$ in $G$ there exists an integer $m_x$ such that   $\f(x)= x^{m_x}$, hence $\f\restriction_{\langle x\rangle} = \varphi_{m_x}$. Therefore $h_{alg}(\f_{\langle x\rangle})= h_{alg}(\varphi_{m_x})=\log m_x$. From the monotonicity of $h_{alg}$ given by Lemma \ref{restriction_quotient} we deduce that $h_{alg}(\f) = \infty$ if $C= \sup_{x\in G} m_x$ is infinite; otherwise $h_{alg}(\f) \geq  \log C$. 
%
\end{itemize}
\end{example}

%
Item (b) of Example \ref{id-abelian} inspires the following

\begin{problem}\label{LaMadre_di_tutti_i_problemi}
If $G$ is a locally finite group and $\phi:G\to G$ an endomorphism, prove or disprove that $h_{alg}(\f)=0$ if and only if every element of $G$ is contained in a $\phi$-invariant finite subgroup of $G$ (i.e., $G$ is a direct limit of finite $\phi$-invariant subgroups).
 \end{problem}
 
\begin{example}\label{Exaaaample}  
For an abelian group $G$ let $G[m]=\{x\in G: mx=0\}$ for $m\in \N_+$. If $G$ is torsion, then 
$$G=\bigcup_{m\in\N_+}  G[m] \ = \ \varinjlim\{ G[m]:m\in\N_+\}, $$
so for $\phi:G\to G$ an endomorphism $h_{alg}(\phi)= \sup_{m\in\N_+}h_{alg}(\phi\restriction_{G[m]}).$
\begin{itemize}
  \item[(a)] In particular if $G$ is a $p$-group, then $h_{alg}(\phi)= \sup_{n\in\N_+}h_{alg}(\phi\restriction_{G[p^n]})$. Moreover $h_{alg}(\phi) >0$ if and only if $h_{alg}(\phi\restriction_{G[p]})>0$ (see \cite[Proposition 1.18]{DGSZ}). 
 \item[(b)]  Using (a), one can prove that if $G$ is also divisible and $h_{alg}(\phi) >0$, then $h_{alg}(\phi) = \infty$. Indeed, from $h_{alg}(\phi\restriction_{G[p]})>0$ and Example \ref{id-abelian}(b) deduce that the orbit $T(\langle x\rangle, \f)$ of some element $x \in G[p]$ is infinite, so isomorphic to the Bernoulli shift $\beta_{\Z(p)}$. For every $n\in \N_+$ find an $y_n\in G$ with $p^{n-1} y_n= x$,  then $T(\langle y_n\rangle, \f)$ is isomorphic to the Bernoulli shift $\beta_{\Z(p^n)}$, hence $h_{alg}(\phi\restriction_{G[p^n]})\geq h_{alg}(\beta_{\Z(p^n)}) = n \log p$. 
\end{itemize}
\end{example}

The next proposition shows that the study of the algebraic entropy for torsion-free abelian groups can be reduced to the case of divisible ones. It was announced for the first time by Yuzvinski \cite{juz67} (for a proof see \cite{DG}). 

\begin{proposition}\label{AA_} 
Let $G$ be a torsion-free abelian group, $\phi:G\to G$ an endomorphism and denote by $\widetilde\phi$ the extension of $\phi$ to the divisible hull $D(G)$ of $G$. Then $h(\phi)=h(\widetilde\phi)$.
\end{proposition}

The algebraic counterpart of the Kolmogorov-Sinai Formula was recently proved in \cite{V}.

\begin{theorem}[Algebraic Kolmogorov-Sinai Formula]\label{AKSF} 
Let $n\in\N_+$ and $\phi:\Z^n\to\Z^n$ an endomorphism. Then $$h_{alg}(\phi)=m(\phi).$$
\end{theorem}

The following is the algebraic counterpart of the Yuzvinski Formula \ref{YF}, which was given a direct proof recently in \cite{GV}. 

\begin{theorem}[Algebraic Yuzvinski Formula] \label{AYF} 
Let $n\in\N_+$ and $\phi:\Q^n\to\Q^n$ an endomorphism. Then $$h_{alg}(\phi)=m(\phi).$$
\end{theorem}

Let us see that the Algebraic Kolmogorov-Sinai Formula \ref{AKSF} follows from the Algebraic Yuzvinski Formula \ref{AYF}. Indeed, if $\phi:\Z^n\to\Z^n$ is an endomorphism, then its extension $\widetilde\phi: \Q^n \to \Q^n$ has the same algebraic entropy as $\f$ by Proposition \ref{AA_}. It remains to note that both endomorphisms have the same characteristic polynomial $p_\phi(t)=p_{\widetilde\phi}(t)$ over $\Z$, so the leading coefficient $s$ of $p_{\widetilde\phi}(t)$ is $1$.  

\begin{remark}\label{AYF0} 
The first attempt for a direct proof of the Algebraic Yuzvinski Formula \ref{AYF} was given by Zanardo in \cite{Z-yuz}. In these notes he considered an automorphism $\phi:\Q^n\to \Q^n$ whose characteristic polynomial over $\Q$ has integer coefficients. First he proved the inequality $h_{alg}(\phi)\leq m(\phi)$. He verified also the converse inequality in two particular cases, namely, when $\phi$ has a unique eigenvalue $\lambda$ such that $|\lambda|>1$ and when $n=2$. These are corollaries of an interesting result of the same paper, that is, 
\begin{equation}\label{geqmax}
h_{alg}(\phi)\geq\max\{\log|\lambda|:\lambda\ \text{eigenvalue of}\ \phi\}.
\end{equation}

The techniques used in \cite{Z-yuz} are from linear algebra, as well as those used in \cite{DKZ}, where the formula in \eqref{geqmax} was extended to all endomorphisms $\phi:\Q^n\to\Q^n$. This was used in \cite{DKZ} to prove the ``case zero'' of the Algebraic Yuzvinski Formula \ref{AYF}, that is, $h(\phi)=0$ if and only if $m(\phi)=0$ for every endomorphism $\phi:\Q^n\to \Q^n$.
\end{remark}

\smallskip
The next theorem is highly non-trivial. It was first proved for torsion abelian groups in \cite{DGSZ}. The proof of the general case given in \cite{DG} requires the Algebraic Yuzvinski Formula \ref{AYF}.

\begin{theorem}[Addition Theorem]\emph{\cite{DG}}\label{AT}
Let $G$ be an abelian group, $\phi:G\to G$ an endomorphism, $H$ a $\phi$-invariant normal subgroup of $G$ and $\overline\phi:G/H\to G/H$ the endomorphism induced by $\phi$. Then $$h_{alg}(\phi)=h_{alg}(\phi\restriction_H)+ h_{alg}(\overline\phi).$$
\end{theorem}

At this stage one could consider the Addition Theorem in full generality:

\begin{problem}\label{PAT}
Does the Addition Theorem hold for every group $G$, every endomorphism $\phi:G\to G$ and every $\phi$-invariant normal subgroup $H$ of $G$?
\end{problem}

Clearly the Addition Theorem holds for simple groups. Moreover one can start with the nilpotent case, which is close to the abelian one and where a natural weaker form can be isolated (see item (a)):

\begin{question}\label{Qnilpotent}
Let $G$ be a group and $\phi:G\to G$ an endomorphism.
\begin{itemize}
\item[(a)] Consider $\phi\restriction_{Z(G)}:Z(G)\to Z(G)$ and the endomorphism $\overline\phi:G/Z(G)\to G/Z(G)$ induced by $\phi$. Is it true that $h_{alg}(\phi)=h_{alg}(\phi\restriction_{Z(G)})+h_{alg}(\overline\phi)$?
\item[(b)] Does the Addition Theorem hold for nilpotent groups?
\item[(c)] Does a positive answer to (a) imply a positive answer to (b)?
\end{itemize}
\end{question}

We conjecture that conjugations of the nilpotent groups have zero algebraic entropy (this is an easy consequence of a positive answer to (a)).

\medskip
A counterpart of the above question (appropriate for solvable groups) can be obtained replacing the center by the derived subgroup:

\begin{question}\label{Qsolvable} Let $G$ be a group and $\phi:G\to G$ an endomorphism. 
\begin{itemize}
\item[(a)] Consider $\phi\restriction_{G'}:G'\to G'$ and the endomorphism $\overline\phi:G/G'\to G/G'$ induced by $\phi$. Is it true that $h_{alg}(\phi)=h_{alg}(\phi\restriction_{G'})+h_{alg}(\overline\phi)$?
\item[(b)] Does the Addition Theorem hold for solvable groups?
\item[(c)] Does a positive answer to (a) imply a positive answer to (b)?
\end{itemize}
\end{question}

This question suggests a weaker form of Problem \ref{PAT}, where $H$ is a {\em fully invariant} (so normal) subgroup. 

\smallskip
The next lemma shows a relation between Question \ref{Qnilpotent}(a) and Question \ref{Qsolvable}(a) for nilpotent groups of class $2$.

\begin{lemma}
For nilpotent groups of class $2$, a positive answer to Question \ref{Qsolvable}(a) implies a positive answer to Question \ref{Qnilpotent}(a).
\end{lemma}
\begin{proof}
Let $G$ be a group and $\phi:G\to G$ an endomorphism. Let $\overline\phi_{G/G'}:G/G'\to G/G'$ be the endomorphism induced by $\phi$.
\begin{equation*}
\xymatrix{
G' \ar[r]\ar[d]^{\phi\restriction_{G'}} & G \ar[r]\ar[d]^{\phi} & G/G' \ar[d]^{\overline\phi_{G/G'}} \\
G' \ar[r] & G \ar[r] & G/G'
}
\end{equation*}
By hypothesis 
\begin{equation}\label{eq1}
h_{alg}(\phi)=h_{alg}(\phi\restriction_{G'})+h_{alg}(\overline\phi_{G/G'}).
\end{equation}
Since $G$ is nilpotent of class $2$, $G/Z(G)$ is abelian and so $G'\subseteq Z(G)$. Consider the endomorphism $\overline\phi_{Z(G)/G'}:Z(G)/G'\to Z(G)/G'$ induced by $\phi$. 
\begin{equation*}
\xymatrix{
G' \ar[r]\ar[d]^{\phi\restriction_{G'}} & Z(G) \ar[r]\ar[d]^{\phi\restriction_{Z(G)}} & Z(G)/G' \ar[d]^{\overline\phi_{Z(G)/G'}} \\
G' \ar[r] & Z(G) \ar[r] & Z(G)/G'
}
\end{equation*}
Since $Z(G)$ is abelian, Addition Theorem \ref{AT} implies 
\begin{equation}\label{eq2}
h_{alg}(\phi\restriction_{Z(G)})=h_{alg}(\phi\restriction_{G'})+h_{alg}(\overline\phi_{Z(G)/G'}).
\end{equation}
Let $\overline\phi_{(G/G')/(Z(G)/G')}:(G/G')/(Z(G)/G')\to (G/G')/(Z(G)/G')$ be the endomorphism induced by $\phi$.
\begin{equation*}
\xymatrix{
Z(G)/G' \ar[r]\ar[d]^{\overline\phi_{G/G'}\restriction_{Z(G)/G'}} & G/G' \ar[r]\ar[d]^{\overline\phi_{G/G'}} & (G/G')/(Z(G)/G') \ar[d]^{\overline\phi_{(G/G')/(Z(G)/G')}} \\
Z(G)/G' \ar[r] & G/G' \ar[r] & (G/G')/(Z(G)/G')
}
\end{equation*}
Let $\overline\phi_{G/Z(G)}:G/Z(G)\to G/Z(G)$ be the endomorphism induced by $\phi$.
\begin{equation*}
\xymatrix{
Z(G)/G' \ar[r]\ar[d]^{\overline\phi_{Z(G)/G'}} & G/G' \ar[r]\ar[d]^{\overline\phi_{G/G'}} & G/Z(G) \ar[d]^{\overline\phi_{G/Z(G)}} \\
Z(G)/G' \ar[r] & G/G' \ar[r] & G/Z(G)
}
\end{equation*}
Since $\overline\phi_{G/G'}\restriction_{Z(G)/G'}$ is conjugated to $\overline\phi_{Z(G)/G'}$ and $\overline\phi_{(G/G')/(Z(G)/G')}$ is conjugated to $\overline\phi_{G/Z(G)}$, Lemma \ref{conj} applies; since the quotient $G/G'$ is abelian, together with Addition Theorem \ref{AT}, they give
\begin{align}\label{eq3}
h_{alg}(\overline\phi_{G/G'})&=h_{alg}(\overline\phi_{G/G'}\restriction_{Z(G)/G'})+h_{alg}(\overline\phi_{(G/G')/(Z(G)/G')}) \notag\\
&=h_{alg}(\overline\phi_{Z(G)/G'})+h_{alg}(\overline\phi_{G/Z(G)}).
\end{align}
Now \eqref{eq1}, \eqref{eq2} and \eqref{eq3} give $$h_{alg}(\phi)=h_{alg}(\phi\restriction_{Z(G)})+h_{alg}(\overline\phi_{G/Z(G)}).$$
\begin{equation*}
\xymatrix{
Z(G) \ar[r]\ar[d]^{\phi\restriction_{Z(G)}} & G \ar[r]\ar[d]^{\phi} & G/Z(G) \ar[d]^{\overline\phi_{G/Z(G)}} \\
Z(G) \ar[r] & G \ar[r] & G/Z(G)
}
\end{equation*}
This is precisely the positive answer to Question \ref{Qnilpotent}(a).
\end{proof}

Making use of a technique introduced by V\' amos \cite{V1} (based on length functions), one can see that uniqueness is available for the algebraic entropy in the category of all abelian groups:

\begin{theorem}[Uniqueness Theorem]\label{UT}\emph{\cite{DG}}
The algebraic entropy $h_{alg}:\Flow_\abg\to\R_{\geq0}\cup\{\infty\}$ is the unique function such that:
\begin{itemize}
\item[(a)] $h_{alg}$ is invariant under conjugation;
\item[(b)] $h_{alg}$ is continuous on direct limits;
\item[(c)] $h_{alg}$ satisfies the Addition Theorem;
\item[(d)] for $K$ a finite abelian group, $h_{alg}(\beta_K^\oplus)=\log|K|$;
\item[(e)] $h_{alg}$ satisfies the Algebraic Yuzvinski Formula.
\end{itemize}
\end{theorem}

In particular the Logarithmic Law follows from these five axioms (see \cite{DG} and \cite{Salce}).

\medskip
The problem of uniqueness can be considered also in a more general setting, but this depends on the solution of Problem \ref{PAT}. For example, if the Addition theorem holds for all solvable groups, then Uniqueness Theorem \ref{UT} obviously extends to flows on solvable groups. As far as larger classes of groups (containing non-solvable ones) are concerned, it is not clear if the above axioms can be sufficient to guarantee the uniqueness of the algebraic entropy.

\subsection{The growth of groups}\label{growth-sec}

Among the most important notions of growth in group theory is the one of growth of a finitely generated group; its definition was given independently by Schwarzc \cite{Sch} and Milnor \cite{M1}. It is a combinatorial form of the geometric notion of growth studied by Efremovich \cite{Efr} in the field of Riemannian manifolds. The connection between the two growth functions is given in \cite{Sch}. For a general reference on this topic see \cite{dlH} and \cite{CC}.

After the publication of \cite{M1}, the growth function of groups was intensively investigated, and several fundamental results were obtained by Wolf \cite{Wolf}, Milnor \cite{M2}, Bass \cite{Bass}, Tits \cite{Tits} and Adyan \cite{Ad}.

\medskip
Let $G$ be a finitely generated group and let $S$ be a finite set of generators of $G$. We can suppose without loss of generality that $S$ is symmetric, that is, $S=S^{-1}$, and that $e_G\not \in S$; so we assume these properties of $S$ for the whole \S \ref{growth-sec}. 

\smallskip
For every $g\in G\setminus \{e_G\}$, denote by $l_S(g)$ the length of the shortest word in the alphabet $S$ representing $g$; moreover let $l_S(e_G)=0$. For $n\in\N$, let $\gamma_S(n)$ be the number of elements $g\in G$ such that $l_S(g) \leq n$; in the notations of algebraic entropy we have that  $$\gamma_S(n)=|T_n(id_G,S\cup\{e_G\})|\ \text{for every}\ n\in\N_+.$$ The function $\gamma_S:\N\to\N$ is called \emph{growth function} of $G$ with respect to $S$.

\smallskip
If $G$ is a finite group and $S$ a set of generators of $G$, then $\gamma_S$ stabilizes at some $n_0\in\N_+$; moreover, choosing $S=G\setminus\{e_G\}$, in this case $\gamma_S(n)=\gamma_S(1)=|G|$ for all $n\in\N_+$.

On the other hand, for an infinite group $G$ generated by a finite subset $S$, the growth function $\gamma_S$ is monotone increasing, that is, $\gamma_S(n)<\gamma_S(n+1) $ for every $n\in\N$. Moreover, according to (the proof of) Lemma \ref{subadd}, $\gamma_S$ is submultiplicative: 

\begin{fact}\label{subgamma}
Let $S$ be a generating set of a group $G$. Then
$\gamma_S(m+n) \leq \gamma_S(m)\cdot \gamma_S(n)$ for every $m,n \in\N_+$. 
\end{fact}

Two maps $\gamma, \gamma': \N \to \N$ are in the relation $\gamma \preceq \gamma'$ if there exist $n_0,C\in\N_+$ such that $\gamma(n) \leq \gamma'(Cn)$ for every $n\geq n_0$. 
Moreover $\gamma$ and $\gamma'$ are \emph{equivalent} (write $\gamma \sim \gamma'$) if $\gamma\preceq\gamma'$ and $\gamma'\preceq\gamma$; indeed, $\sim$ is an equivalence relation.

\begin{example}
\begin{itemize}
\item[(a)] For every $\alpha, \beta\in\R_{\geq0}$, $n^\alpha\sim n^\beta$ if and only if $\alpha=\beta$.
\item[(b)] For every $a,b\in\R_{>1}$, $a^n\sim b^n$.
\end{itemize}
\end{example}

In particular, item (b) shows that all exponentials are equivalent with respect to $\sim$. 
So it is possible to give the following definitions.
A map $\gamma: \N \to \N$ is called 
\begin{itemize}
\item[(a)] \emph{polynomial} if $\gamma(n) \preceq n^d$ for some $d\in\N_+$;
\item[(b)] \emph{exponential} if $\gamma(n) \sim e^n$;
\item[(c)] \emph{intermediate} if $\gamma(n)\succ n^d$ for every $d\in\N_+$ and $\gamma(n)\prec e^n$.
\end{itemize}

\medskip
Fact \ref{subgamma} implies that $\gamma_S(n)\preceq\gamma_S(1)^n$ for every $n\in\N$, so the growth of $\gamma_S$ is at most exponential. Item (b) of the following fact shows that $\gamma_S$ is at least polynomial for an infinite group. These observations give sense to Definition \ref{growth-def} below.

\begin{fact}
Let $G$ be a finitely generated group and $S$ a finite set of generators of $G$. Then:
\begin{itemize}
\item[(a)] $\gamma_S(n)\sim 1$ if and only of $G$ is finite;
\item[(b)] $n\preceq \gamma_S(n)$ if and only if $G$ is infinite.
\end{itemize}
\end{fact}

It is easy to see that for two finite generating sets $S$ and $S'$ of a group $G$, the equivalence 
\begin{equation}\label{equiv}
\gamma_S\sim\gamma_{S'} 
\end{equation}
holds.
Define the \emph{growth function} $\gamma^G$ of $G$ as the equivalence class $[\gamma_S]_\sim$ of the growth function $\gamma_S$ of $G$ with respect to $S$.

\smallskip
The next basic property holds.

\begin{lemma}\label{finind-gr}
Let $G$ be a finitely generated group and $H$ a finite-index subgroup of $G$. Then $\gamma^G=\gamma^H$.
\end{lemma}

The  equivalence in \eqref{equiv} ensures the correctness of the following definition.

\begin{definition}\label{growth-def}
Let $G$ be a finitely generated group and let $S$ be a generating set of $G$. Then $G$ has
\begin{itemize}
\item[(a)] \emph{polynomial growth} if $\gamma_S$ is polynomial;
\item[(b)] \emph{exponential growth} if $\gamma_S$ is exponential;
\item[(c)] \emph{intermediate growth} if $\gamma_S$ is intermediate.
\end{itemize}
\end{definition}

\begin{example}\label{exgrowth}
\begin{itemize}
\item[(a)] The abelian groups have polynomial growth (see Example \ref{id-abelian}(a)); more precisely, for $G$ a finitely generated abelian group and $S$ a finite set of generators of $G$, $\gamma_S(n)\sim n^{r_0(G)}$ (this follows for example from the general Bass-Guivarch Formula \ref{BGF} below).
\item[(b)] The free groups with at least two generators have exponential growth.
\end{itemize}
\end{example}

Let $G$ be a finitely generated group and $S$ a finite set of generators of $G$. The \emph{growth exponent} of $G$ is $$\delta_G=\limsup_{n\to \infty}\frac{\log\gamma_S(n)}{\log n}.$$
It is easy to see that $\delta_G$ is finite if and only if $G$ has polynomial growth.
We see in Corollary \ref{grexp} below that $\delta_G$ does not depend on the choice of $S$ and that it is a limit. 
Moreover $\delta_G$ is a positive integer when it is finite.

\medskip
Milnor \cite{M3} proposed the following problem. We divide it in two parts as in \cite{Gri}.

\begin{problem}\label{Milnor-pb}\emph{\cite{M3}}
Let $G$ be a finitely generated group and $S$ be a finite set of generators of $G$.
\begin{itemize}
\item[(a)] Is the growth function $\gamma_S$ necessarily equivalent either to a power of $n$ or to the exponential function $2^n$?
\item[(b)] In particular, is the {growth exponent} $\delta_G$ either a well defined integer or infinity?
For which groups is $\delta_G$ finite?
\end{itemize}
\end{problem}

Part (a) of Problem \ref{Milnor-pb} was solved negatively by Grigorchuk in \cite{Gri1,Gri2,Gri3,Gri4}, where he constructed his famous examples of finitely generated groups $\mathbb G$ with intermediate growth.

\medskip
We discuss now part (b). Recall that a group $G$ is \emph{virtually nilpotent} if it contains a nilpotent finite-index subgroup. Milnor conjectured that $\delta_G$ is finite if and only if $G$ is virtually nilpotent. The same conjecture was formulated by Wolf \cite{Wolf} and Bass \cite{Bass}. Part (b) of Problem \ref{Milnor-pb} was solved by Gromov \cite{Gro}, confirming Milnor's conjecture:

\begin{theorem}[Gromov Theorem]\label{GT}\emph{\cite{Gro}}
A finitely generated group $G$ has polynomial growth if and only if $G$ is virtually nilpotent.
\end{theorem}

The easier implication that a nilpotent finitely generated group has polynomial growth was first proved by Wolf \cite{Wolf}. Moreover Guivarch \cite{Gui} and Bass \cite{Bass} computed, independently and with different proofs, the exact order of polynomial growth:

\begin{theorem}[Bass-Guivarch Formula]\label{BGF}
Let $G$ be a nilpotent finitely generated group with lower central series $$G = G_1 \supseteq G_2 \supseteq \ldots \supseteq G_k=\{e_G\}.$$ 
In particular the quotient group $G_m/G_{m+1}$ is a finitely generated abelian group for every $m=1,\ldots, k-1$. Then 
$\gamma^G=[n^d]_\sim$ for 
$$
d= \sum_{m=1}^{k-1} n\cdot r_0(G_m/G_{m+1}).
$$
\end{theorem}

The following is an important consequence of Gromov Theorem \ref{GT}, Bass-Guivarch Formula \ref{BGF} and Lemma \ref{finind-gr}.

\begin{corollary}\label{grexp}
Let $G$ be a finitely generated group. Then:
\begin{itemize}
\item[(a)] the $\limsup$ in the definition of growth exponent is a limit;
\item[(b)] if $G$ has polynomial growth, then $\delta_G\in\N$, otherwise $\delta_G=\infty$;
\item[(c)] $\delta_G$ does not depend on the choice of the finite set of generators of $G$.
\end{itemize}
\end{corollary}
%
%

By Fact \ref{subgamma} the sequence $\{\log\gamma_S(n)\}_{n\in\N}$ is subadditive, so an application of Fekete Lemma \ref{fekete} gives

\begin{lemma}\label{2nov} Let $G$ be a finitely generated group and $S$ a finite set of generators of $G$. Then the limit $$\lambda_S=\lim_{n\to \infty} \frac{\log \gamma_S(n)}{n}$$ always exists.
\end{lemma}

The number $\lambda_S$ is called \emph{growth rate} of $G$ with respect to $S$. 

\smallskip
In the same way as the growth exponent $\delta_G$ characterizes the finitely generated groups $G$ of polynomial growth, the growth rate $\lambda_S$ of $G$ with respect to a finite set of generators $S$ of $G$ characterizes the finitely generated groups $G$ of exponential growth in the following sense: 

\begin{remark}\label{exp<->gamma>0}
Let $G$ be a finitely generated group and let $S$ be a finite set of generators of $G$. 
\begin{itemize}
\item[(a)] The growth function $\gamma_S$ has exponential growth if and only if $\lambda_S>0$. For a proof see Proposition \ref{exp} which covers this case.
\item[(b)] Since the growth function $\gamma_S$ does not depend on the choice of $S$, 
the positivity of $\lambda_S$ does not depend on the choice of the finite set of generators $S$, it is a feature related to the group $G$.
\item[(c)] Let us see that in contrast with item (b) (and unlike $\delta_G$) $\lambda_S$ {\em depends} on the choice of $S$. Indeed, it may occur the case that $G$ is a finitely generated group, $S$ and $S'$ are finite subsets of generators of $G$, and $\lambda_S\neq\lambda_{S'}$.
For example let $G$ be the free group with two generators $a$ and $b$; then $S=\{a^{\pm 1},b^{\pm 1}\}$ gives $\gamma_S(n)= 1+ 2\cdot (3^n-1)$ and $S'=\{a^{\pm 1},b^{\pm 1},(ab)^{\pm 1}\}$ gives $\gamma_{S'}(n)=1+ 2\cdot (4^n-1)$, so $\lambda_S=\log 3$ and $\lambda_{S'}=\log 4$.
\end{itemize}
\end{remark}

\medskip
In the notations of algebraic entropy, for a finitely generated group $G$ and a finite set of generators $S$ of $G$, we have that 
$$
\lambda_S=H_{alg}(id_G,S\cup\{e_G\}).
$$
In particular in the context of algebraic entropy we can see immediately a deep difference between the abelian case and the non-abelian case. Indeed the identity map of an abelian group has always zero algebraic entropy by Example \ref{id-abelian}(a). In the non-abelian case, in view of Remark \ref{exp<->gamma>0}(a), for a finitely generated group $G$ the identity map has zero algebraic entropy when $G$ has either polynomial growth (i.e., it is virtually nilpotent by Gromov Theorem \ref{GT}) or intermediate growth. On the other hand $h_{alg}(id_G)>0$ when $G$ has exponential growth, and in this case $h_{alg}(id_G)=\infty$ (see Corollary \ref{0<->0}(c)).
In particular we have proved the following

\begin{proposition}\label{h(id)>0}
Let $G$ be a finitely generated group. Then:
\begin{itemize}
\item[(a)]$G$ has either polynomial growth or intermediate growth if and only if $h_{alg}(id_G)=0$;
\item[(b)]$G$ has exponential growth if and only if $h_{alg}(id_G)=\infty$.
\end{itemize}
\end{proposition}

So Grigorchuk group $\mathbb G$ is an example of a finitely generated group with $h_{alg}(id_{\mathbb G})=0$ and intermediate growth.

\smallskip
Let $\mathfrak P$ denote the class of all groups $G$ with $h_{alg}(id_G)=0$. By Lemma \ref{restriction_quotient} and Theorem \ref{WAT} the class $\mathfrak P$ is stable under taking subgroups, quotients and finite direct products (but not with respect to free products, see Example \ref{Exa}). Moreover a group belongs to $\mathfrak P$ precisely when all its finitely generated subgroups belong to $\mathfrak P$. Therefore the above proposition shows that $\mathfrak P$ is characterized also by the property that $G \in \mathfrak P$ if and only if all   finitely generated subgroups if $G$ have either polynomial growth or intermediate growth. This may allow us to extend the notion of having polynomial growth or intermediate growth also to all groups, not only the finitely generated ones. 
As a byproduct we find some natural permanence properties of this notion. 

Note that $\mathfrak P$ contains all locally nilpotent groups. Since both $\mathfrak P$ and the class of amenable groups (i.e., groups admitting a right invariant finitely additive probability measure) are stable under taking subgroups, quotients and finite products, and both do not contain the free groups, it make sense to consider the following

\begin{problem}
Compare $\mathfrak P$ with the class of amenable groups. Is one of them contained in the other?
\end{problem}

We do not know whether the class $\mathfrak P$ is stable under extensions.

\smallskip
We formulate the following

\begin{conjecture}
Let $G$ be a (finitely generated) group with $h_{alg}(id_G)=\infty$ (i.e., $G$ has exponential growth). Then:
\begin{itemize}
\item[(a)] $h_{alg}(\phi)=\infty$ for every automorphism $\phi:G\to G$.
\end{itemize}
One should start considering the following weaker form:
\begin{itemize}
\item[(b)] $h_{alg}(\phi)=\infty$ for every internal automorphism $\phi:G\to G$.
\end{itemize}
\end{conjecture}

Let $G$ be a finitely generated group and $S$ a finite set of generators of $G$. 
The exponential growth rate $\lambda_S$ of $G$ with respect to $S$ is called also \emph{entropy} by some authors and for instance in \cite{GroLP}. As explained in \cite{dlH}, the reason for this terminology is the following result in differential geometry. If $G$ is the fundamental group of a compact Riemannian manifold $M$ of diameter $d$ and $S$ is an appropriate finite generating set of $G$, then $\lambda_S$ is a lower bound for the topological entropy of the geodesic flow of the manifold multiplied by $2d$ \cite{Manni}.

An analogous result was given by Bowen \cite{Blast}. Indeed, he proved that the topological entropy $h_{top}(\f)$ of a continuous selfmap $\f:M\to M$ of a compact manifold $M$ satisfies $h_{top}(\f) \geq \log GR(\f_*)$, where $GR(\f_*)$ is the \emph{growth rate} of the endomorphism $\f_*:\pi_1 (M)\to \pi_1 (M)$ of  the fundamental group $\pi_1(m)$ of $M$ induced by $\f$. The growth rate $GR$ is briefly studied also in the recent manuscript \cite{Falconer}.  It presents properties typical of entropy functions, so one should start to compare  the growth rate $GR$ with the algebraic entropy $h_{alg}$ on endomorphisms of finitely generated groups. 
 
For a finitely generated group $G$, the \emph{uniform exponential growth rate} (or \emph{algebraic entropy}) of $G$ is defined as
$$
\lambda(G)=\inf\{\lambda_S:S\ \text{finite set of generators of}\ G\}
$$
(see for instance \cite{dlH-ue}). By Remark \ref{exp<->gamma>0}(a) the group $G$ has exponential growth if and only if $\lambda_S>0$. Moreover $G$ has \emph{uniform exponential growth} if $\lambda(G)>0$.
Gromov \cite{GroLP} asked whether every finitely generated group of exponential growth is also of uniform exponential growth, and many examples were found in this direction. It is worth to mention here that this problem was recently solved by Wilson \cite{Wilson}, who gave examples of groups of exponential growth but not of uniform exponential growth. 

De la Harpe \cite{dlH-ue}, in relation to Gromov's question, asked which are the finitely generated groups that realize their uniform exponential growth rate, that is, which finitely generated groups $G$ have a finite set of generators $S$ such that $\lambda_S=\lambda(G)$. For example free groups realize their uniform exponential growth rate, as well as several other classes of finitely generated groups. On the other hand Sambusetti \cite{Sam} gave an example of a finitely generated group $G$ such that $\lambda_S>\lambda(G)$ for every finite set of generators $S$ of $G$.

In this context some authors considered also the relations between the cardinality of a finite set of generators $S$ of a finitely generated group $G$ and the exponential growth rate $\lambda_S$ of $G$ with respect to $S$ (see for instance \cite{ArzLys} and \cite{Zuddas}).

\subsection{The growth of algebraic flows and the Pinsker subgroup}\label{Growth-sec}

Let $G$ be a group, $\phi:G\to G$ an endomorphism and $F\in[G]^{<\omega}$.
Consider the function $$\gamma_{\phi,F}:\N_+\to\N_+\ \text{defined by}\ \gamma_{\phi,F}(n)=|T_n(\phi,F)|\ \text{for every}\ n\in\N_+.$$
Since
\begin{equation}\label{2}
|F|\leq\gamma_{\phi,F}(n)\leq|F|^n\mbox{ for every }n\in\N_+,
\end{equation}
the growth of $\gamma_{\phi,F}$ is always at most exponential; moreover, $H(\phi,F)\leq \log |F|$.

\smallskip
The following notions were considered in \cite{DG0} in the abelian case.

\begin{definition}
Let $G$ be a group, $\phi:G\to G$ an endomorphism and $F\in[G]^{<\omega}$. We say that:
\begin{itemize}
\item[(a)] $\phi$ has \emph{polynomial growth at $F$} (denoted by $\f\in\pgp_F$) if $\gamma_{\phi,F}$ is polynomial;
\item[(b)] $\phi$ has \emph{exponential growth at $F$} (denoted by $\phi\in\mathrm{Exp}_F$) if $\gamma_{\phi,F}$ is exponential;
\item[(c)] $\phi$ has \emph{intermediate growth at $F$} if $\gamma_{\phi,F}$ is intermediate.
\end{itemize}
\end{definition}

This definition generalizes the classical one of growth of a finitely generated group recalled in Definition \ref{growth-def}. Indeed, if $G$ is a finitely generated group and $S$ is a finite symmetric set of generators of $G$, then $\gamma_S=\gamma_{id_G,S}$.
In other words, that $id_G$ has polynomial (respectively, exponential or intermediate) growth at $S\cup\{e_G\}$ means precisely that $G$ has polynomial (respectively, exponential or intermediate) growth.

So one can consider the following general problem.

\begin{problem}\label{gg-pb}
Let $G$ be a group, $\phi:G\to G$ an endomorphism and $F\in[G]^{<\omega}$. When does $\phi$ have polynomial (respectively, exponential or intermediate) growth at $F$?
\end{problem}

One could start, as a natural generalization of the classical problem, considering a finitely generated group $G$, an automorphism $\phi:G\to G$ and a finite symmetric set of generators $S$ of $G$. 

\bigskip
It is quite natural to study the relation of this growth problem with the algebraic entropy. We start from the following characterization of exponential growth, which was given in \cite{DG0} for the abelian case. It generalizes Remark \ref{exp<->gamma>0}(a) to arbitrary endomorphisms.

\begin{proposition}\label{exp}
Let $G$ be a group, $\phi:G\to G$ an endomorphism and $F\in[G]^{<\omega}$. Then $H_{alg}(\phi,F)>0$ if and only if $\phi\in\mathrm{Exp}_F$.
\end{proposition}

\begin{proof}  Note that both $H_{alg}(\phi,F)>0$ and $\f\in\mathrm{Exp}_F$ imply $|F|\geq2$; indeed, if $|F|=1$, then $\gamma_{\phi,F}(n)=1$ for every $n\in\N_+$.

Assume that $H_{alg}(\phi,F)=a>0$. Consequently there exists $m\in\N_+$ such that $\log\gamma_{\phi,F}(n)>n\cdot \frac{a}{2}$ for every $n>m$. Then $\gamma_{\phi,F}(n)>e^{n\cdot \frac{a}{2}}$ for every $n>m$. Since $|F|\geq2$, $\gamma_{\phi,F}(n)\geq2$ for every $n\in\N_+$; in particular, $\gamma_{\phi,F}(n)\geq(\sqrt[m]{2})^n$ for every $n\leq m$. For $b=\min\{\sqrt[m]{2},e^{\frac{a}{2}}\}$, we have $\gamma_{\phi,F}(n)\geq b^n$ for every $n\in\N_+$, and this proves that $\phi\in\mathrm{Exp}_F$.

Suppose now that $\phi\in\mathrm{Exp}_F$. Then there exists $b\in\R_{>1}$ such that $\gamma_{\phi,F}(n)\geq b^n$ for every $n\in\N_+$. Hence $H(\phi,F)\geq \log b>0$.
\end{proof}

On the other hand it is clear that:

\begin{lemma}\label{pol->0}
Let $G$ be a group, $\phi:G\to G$ an endomorphism and $F\in[G]^{<\omega}$. If $\phi\in\pgp_F$, then $H_{alg}(\phi,F)=0$.
\end{lemma}

In general the converse implication is not true even for the identity. Indeed, if $\phi$ has intermediate growth at $F$, then it has $H_{alg}(\phi,F)=0$ by Proposition \ref{exp}. So the identity map of Grigorchuk's group $\mathbb G$ has intermediate growth at the finite set $S$ of generators of $\mathbb G$, yet $H_{alg}(\phi,S)=0$. Nevertheless one can consider the following

\begin{problem}
For which groups $G$, endomorphisms $\phi:G\to G$ and $F\in[G]^{<\omega}$ does $H_{alg}(\phi,F)=0$ imply $\phi\in\pgp_F$? Is this true for nilpotent groups? And for the wider class $\mathfrak P$?
\end{problem}

If Problem \ref{LaMadre_di_tutti_i_problemi} can be resolved positively, then the above property is obviously available in all locally finite groups $G$. Indeed, $H_{alg}(\phi,F)=0$ for some  endomorphisms $\phi:G\to G$ and $F\in[G]^{<\omega}$ would imply that $F$ is contained in a finite $\f$-invariant subgroup, so that the function $\gamma_{\f,F}$ is actually a constant (so trivially $\phi\in\pgp_F$).

A motivation to this problem is given by the abelian case. Indeed, the following is the main result of \cite{DG0} and solves completely Problem \ref{gg-pb} in the abelian case.

\begin{theorem}[Dichotomy Theorem]\emph{\cite{DG0}}\label{DT}
Let $G$ be a group, $\phi:G\to G$ an endomorphism and $F\in[G]^{<\omega}$. Then $\phi$ has either exponential or polynomial growth at $F$.
\end{theorem}
 
In other words $H_{alg}(\phi,F)=0$ if and only if $\phi\in\pgp_F$.

\smallskip
The proof of this surprising dichotomy is based on the following notion, introduced imitating its counterparts for measure entropy and topological entropy.

\begin{definition}\cite{DG0}\label{pinsker-def}
Let $G$ be a group and $\phi:G\to G$ an endomorphism. The \emph{Pinsker subgroup} of $G$ with respect to $\phi$ is the greatest $\f$-invariant subgroup $\P(G,\f)$ of $G$ such that $h_{alg}(\f\restriction_{\P(G,\f)})=0$.
\end{definition}

It is clear, that the subgroup $\P(G,\f)$ is unique, if it exists. It is easy to see that $\P(G,\f)$ exists for an  endomorphism $\phi:G\to G$ if and only if whenever $H,K$ are $\phi$-invariant subgroups of $G$, with $h_{alg}(\phi\restriction_H)=0=h_{alg}(\phi\restriction_K)$ then also $h_{alg}(\phi\restriction_{\langle H,K\rangle})=0$. In case this property holds true, $\P(G,\f)$ is the subgroup of $G$ generated by all $\phi$-invariant subgroups $H$ of $G$ such that $h_{alg}(\phi\restriction_H)=0$.

In \cite{DG0} the existence of the subgroup $\P(G,\f)$ is established, as well as that of the greatest $\f$-invariant subgroup $\pgp(G,\f)$ of $G$ where the restriction of $\f$ has polynomial growth. Since $\phi\in\pgp_F$ immediately yields $H_{alg}(\phi,F)=0$ by Lemma \ref{pol->0}, this entails $\pgp(G,\f)\subseteq \P(G,\f)$. Actually, the equality $$\P(G,\f) = \pgp(G,\f)$$ holds for every abelian group $G$ and every endomorphism $\phi:G\to G$, and the proof of the Dichotomy Theorem \ref{DT} is essentially based on it. Moreover in the proof of this equality the Algebraic Yuzvinski Formula \ref{AYF} applies (more precisely, that its weaker form described in Remark \ref{AYF0} is enough for this proof).

\bigskip
The next example shows that the Pinsker subgroup may not exist in the non-abelian case even for the identity map.

\begin{example}\label{Exa} Let $G$ be the free non-abelian group of two generators $x_1$ and $x_2$. Let $H_i$ be the cyclic subgroup of $G$ generated by $x_i$, $i=1,2$. Then for $G = \langle H_1,H_2\rangle$ and for $\f = id_G$ one has $h_{alg}(\phi\restriction_{H_1})=0=h_{alg}(\phi\restriction_{H_2})$. Nevertheless, $\f$ has 
exponential growth, in particular $h_{alg}(\f) > 0$. Note, that the example can be used also to see that the counter-part of the similar property for polynomial growth fails as well (i.e., both $\phi\restriction_{H_1}$ and $\phi\restriction_{H_2}$ have polynomial growth, while $\f$ has exponential growth).  
\end{example}

Nevertheless the Pinsker subgroup exists for locally quasi-periodic endomorphisms $\phi$ of locally finite groups $G$ as $\mathbf{P}(G,\phi)=G$ in this case by Remark \ref{locfin}(c). On the other hand, one may impose some restraints on the group, so that {\em all} endomorphisms admit a Pinsker subgroup: 

\begin{question} For which groups $G$ every endomorphisms $\phi:G\to G$ admits a Pinsker subgroup? 
Is this true for nilpotent groups? What about the class $\mathfrak P$?
\end{question}

In this line, a positive solution of Problem \ref{LaMadre_di_tutti_i_problemi} would entail that for every endomorphism $\phi:G\to G$ of a locally finite group $G$ the subgroup generated by all finite $\f$-invariant subgroups of $G$ is precisely the Pinsker subgroup of $G$ with respect to $\f$.

\subsection{The algebraic e-spectrum of abelian groups}\label{aes-sec}

In analogy with the topological e-spectrum introduced for the topological entropy in \S \ref{tes-sec}, we give the following

\begin{definition} 
For a group $G$ the \emph{algebraic $e$-spectrum} of $G$ is $${\bf E}_{alg}(G)=\{h_{alg}(\phi): \phi\in \End(G)\} \subseteq \R_{\geq0}\cup \{\infty\}.$$
\end{definition}  

\begin{example}\label{exdgsz}\cite[Proposition 1.3]{DGSZ}
\begin{itemize} 
\item[(a)] If $G$ is a torsion abelian group or isomorphic to a subgroup of $\Q$, then 
$${\bf E}_{alg}(G)\subseteq \{\log n: n\in \N_+\}\cup \{\infty\},$$ 
so ${\bf E}_{alg}(G)\cap \R_{\geq0}$ is discrete in $\R_{\geq0}$ and $\inf ({\bf E}_{alg}(K)\setminus \{0\})=\log 2$.
\item[(b)] If $p$ is a prime and $G$ is an abelian $p$-group, then ${\bf E}_{alg}(G)\subseteq \{n\log p: n\in \N\}\cup \{\infty\},$  so ${\bf E}_{alg}(G)\cap \R_{\geq0}$ is {\em uniformly} discrete in $\R_{\geq0}$.
\end{itemize}
\end{example}

Let 
\begin{equation}\label{***}
{\bf E}_{alg}= \bigcup\{{\bf E}_{alg}(G):G\ \mbox{abelian group}\}.
\end{equation}

As in the topological case, the following natural question arises.

\begin{problem}\label{h>0-pb}
Is $\inf ({\bf E}_{alg}\setminus \{0\})=0$?
\end{problem}

By means of appropriate reductions, it is possible to restrict to endomorphisms of $\Q^n$:

\begin{theorem}\label{epsilon}\emph{\cite{DG}}
We have $\inf ({\bf E}_{alg}\setminus \{0\})=\inf(\bigcup_{n\in\N_+}\mathbf E_{alg}(\Q^n)\setminus\{0\})$.
\end{theorem}

It follows immediately from Theorem \ref{epsilon} and the Algebraic Yuzvinski Formula \ref{AYF} that 
$$
\mathfrak L=\inf ({\bf E}_{alg}\setminus \{0\});
$$
therefore Problem \ref{h>0-pb} is equivalent to Lehmer Problem \ref{L-pb}.

\medskip
We give now some properties of the set $\mathbf E_{alg}$.

\begin{lemma}\label{semiclosed}
The set ${\bf E}_{alg}$ is a submonoid of $\R_{\geq0}\cup\{\infty\}$. If $\{a_n\}_{n\in\N_+}\subseteq\mathbf E_{alg}\setminus\{0\}$, then $\sum_{n\in\N_+}a_n\in\mathbf E_{alg}$.
\end{lemma}
\begin{proof}
Clearly, $0\in\mathbf E_{alg}$. The Weak Addition Theorem \ref{WAT} implies that $\mathbf E_{alg}$ is a subsemigroup of $\R_{\geq0}\cup\{\infty\}$.

To prove the second statement, let $b_n=a_1+\ldots+a_n$ for every $n\in\N_+$ and $b=\lim_{n\to\infty} b_n$. Then $b\in\mathbf E_{alg}$. 
In fact, for every $n\in\N_+$ let $G_n$ be an abelian group and $\phi_n:G_n\to G_n$ an endomorphism such that $h_{alg}(\phi_n)=a_n$. Moreover for every $n\in\N_+$ let $\psi_n=\phi_1\times\ldots\times\phi_n:G_1\times\ldots\times G_n\to G_1\times\ldots\times G_n$. Let $G=\bigoplus_{n\in\N_+}G_n$ and $\phi=\bigoplus_{n\in\N_+}\phi_n$; then $h_{alg}(\phi)=\sup_{n\in\N_+}h_{alg}(\phi_n)$ by Proposition \ref{dirlim}. Since $h_{alg}(\phi)=b$, we can conclude that $b\in\mathbf E_{alg}$.
\end{proof}

The next proposition shows that the study of the monoid $\mathbf E_{alg}(G)$ for torsion-free abelian groups $G$ can be reduced to the case of divisible ones. It follows immediately from Proposition \ref{AA_}.

\begin{proposition}\label{AA}
Let $G$ be a torsion-free abelian group and $\phi:G\to G$ an endomorphism. Then $\mathbf E_{alg}(G)=\mathbf E_{alg}(D(G))$.
\end{proposition}

The following theorem gives equivalent conditions with respect to a positive answer to Problem \ref{h>0-pb}. It is the algebraic counterpart of Fact \ref{LD-top}.

\begin{theorem}\label{LD}
The following conditions are equivalent:
\begin{itemize}
  \item[(a)] ${\bf E}_{alg}=\R_{\geq0}\cup \{\infty\}$;
  \item[(b)] $\inf ({\bf E}_{alg}\setminus \{0\})=0$;
  \item[(c)] $\inf(\bigcup_{n\in\N_+}{\bf E}_{alg}(\Z^n)\setminus \{0\})=0$.
\end{itemize}
\end{theorem}

\begin{proof} (a)$\Rightarrow$(b) and (c)$\Rightarrow$(b) are obvious.

\smallskip
(b)$\Rightarrow$(a) First we verify that under our hypothesis, $\mathbf E_{alg}\setminus \{0\}$ is dense in $\R_{\geq0}$. To this end, consider 
an interval $(a,b)\subseteq \R_{\geq0}$ and let $\varepsilon=b-a>0$. By our hypothesis, there exists $r \in {\bf E}_{alg} \cap (0,\varepsilon)$.  
Since $h_{alg}(\phi)<\varepsilon=a-b$, there exists $k\in\N_+$ such that $k r\in(a,b)$. By Proposition \ref{LL}, $k r \in\mathbf E_{alg}$, and this proves our claim.

Let $r\in\R_{\geq0}$. If $r=0$ we are done, as $0\in\mathbf E_{alg}$. Assume that $r>0$. Let $a_0=b_0=0$.  We are going to define inductively a sequence $\{a_n\}_{n\in\N}\subseteq \mathbf E_{alg}$ such that $b_n= \sum_{m=0}^na_m$ satisfies
\begin{equation}\label{SsS}
b_k \in \left(r-\frac{1}{k},r\right) \mbox{ for all }k \in \N_+.
\end{equation}
Since \eqref{SsS} clearly implies $\sum_{n\in\N_+}a_n=r$, Lemma \ref{semiclosed} would imply $r\in\mathbf E_{alg}$, and this concludes the proof since obviously $\infty\in \mathbf E_{alg}$.

Assume that $n>0$ and the elements $a_1, \ldots, a_{n-1}$ are defined so that  (\ref{SsS}) holds for all $1 \leq k \leq n-1$ (obviously, for $n=1$ this condition becomes vacuous, so imposes no restraint on $b_0=0$). By the density of $\mathbf E_{alg}\setminus\{0\}$ in $\R_{\geq0}$ proved in the previous step of the proof, there exists $a_n \in \mathbf E_{alg} \cap (r- b_{n-1} -\frac{1}{n},r- b_{n-1})$. Clearly, $b_n= \sum_{k=0}^na_k$ satisfies \eqref{SsS}, so this ends up the inductive definition. 

\smallskip
(b)$\Rightarrow$(c) By Theorem \ref{epsilon}, $\inf ({\bf E}_{alg}\setminus \{0\})=\inf(\mathbf E_{alg}(\Q^n)\setminus\{0\})$. For an endomorphism $\phi:\Q^n\to\Q^n$, let $p_\phi(x)=sx^n+\ldots$ be the characteristic polynomial of $\phi$ over $\Z$, with 
 $s$ positive. The Algebraic Yuzvinski Formula \ref{AYF} gives $h_{alg}(\phi)=\log s+\sum_{|\lambda_i|>1}\log|\lambda_i|$, where $\{\lambda_i:i=1,\ldots,n\}$ is the the family of all eigenvalues of $\phi$. If $s\neq 1$, then $\log s\geq \log 2$. So one may assume without loss of generality that $s=1$. In this case, there exists a free $\phi$-invariant subgroup $F\cong\Z^n$ of $\Q^n$ and $h_{alg}(\phi)=h_{alg}(\phi\restriction_{\Z^n})$ in view of Proposition \ref{AA_}. This gives (c).
\end{proof}

Let $\mathbb A_0$ be the set of all algebraic numbers $\alpha$ with $ |\alpha|\geq1$.
Then $\mathbb A_0$ is countable and  $\alpha\in \mathbb A_0$ yields $\overline\alpha \in  \mathbb A_0$, so 
$|\alpha|\in\mathbb A_0 $ as well (since $|\alpha|=\sqrt{\alpha\overline\alpha}$); in particular, $|\alpha| \in\mathbb A_0 \cap \R_{\geq 1}$. Then 
$$
\mathbb A=\bigcup_{n\in\N_+}\{\log|\alpha|:\alpha\in\mathbb A_0 \}
$$ 
 is a dense countable submonoid of $(\R_{\geq0},0)$.

\smallskip
We see now that a negative answer to Problem \ref{h>0-pb} would imply that $\mathbf E_{alg}$ is countable.

\begin{theorem}
If $\inf(\mathbf E_{alg}\setminus\{0\})>0$, then ${\bf E}_{alg}\subseteq \mathbb A\cup\{\infty\}$, so in particular it is countable. Moreover this inclusion is proper.
\end{theorem}
\begin{proof} 
We verify that $\mathbf E_{alg}(G)\subseteq \mathbb A$ for every abelian group $G$. So let $G$ be an abelian group and $\phi:G\to G$ an endomorphism. By the Addition Theorem \ref{AT}, $h_{alg}(\phi)=h_{alg}(\phi\restriction_{t(G)})+h_{alg}(\overline\phi)$, where $\overline\phi:G/t(G)\to G/t(G)$ is the endomorphism induced by $\phi$. Now, $h_{alg}(\phi\restriction_{t(G)})\in\log\N\subseteq \mathbb A$. So it suffices to verify that $h_{alg}(\overline\phi)\in \mathbb A$; in other words, we can assume without loss of generality that $G$ is torsion-free. By Proposition \ref{AA} we can assume also that $G$ is divisible, that is, $G=V$ is a vector space over $\Q$.

Assume first that $V$ is a finite dimensional vector space over $\Q$, that is, $V\cong \Q^n$. By the Algebraic Yuzvinski Formula \ref{AYF}, $h_{alg}(\phi)=\log s+\sum_{|\lambda_i|>1}\log|\lambda_i|$, where $s$ is the leading coefficient of the characteristic polynomial of $\phi$ over $\Z$ and $\{\lambda_i:i=1,\ldots,n\}$ is the family of all the eigenvalues of $\phi$. Obviously, $h_{alg}(\phi)\in \mathbb A$.

Let now $V$ be an arbitrary vector space over $\Q$ of infinite dimension $\kappa$. If $h_{alg}(\phi)=\infty$ we are done, so we suppose that $h_{alg}(\phi)<\infty$. 

First we construct by transfinite induction an ordinal chain $\{V_\alpha\}_{\alpha<\kappa}$ of $\phi$-invariant subspaces of $V$ such that 
\begin{itemize}
\item[(i)] $V=\bigcup_{\alpha<\kappa}V_\alpha$; 
\item[(ii)] $\dim V_{\alpha + 1}/V_\alpha <\infty$ whenever $\alpha < \kappa$; and 
\item[(iii)] $V_\beta=\bigcup_{\alpha<\beta}V_\alpha$ whenever $\beta < \kappa$ is a limit ordinal.
\end{itemize}

Let $V_0 = 0$. Assume $0< \beta < \kappa$  so that all $V_\alpha$, for $\alpha<\beta$ are defined satisfying (ii) and (iii). If is $\beta$ a limit ordinal let $V_\beta=\bigcup_{\alpha<\beta}V_\alpha$. If $\beta=\alpha+1$ for some ordinal $\alpha$, let $\pi_\alpha:V\to V/V_\alpha$ be the canonical projection and let $\overline\phi_\alpha:V/V_\alpha\to V/V_\alpha$ be the endomorphism induced by $\phi$. 
Pick $x_\beta\in V\setminus V_\alpha$, so that  $\pi_\alpha(x_\beta)\in V/V_\alpha\setminus\{0\}$. Then $W_\beta=\langle T(\overline\phi_\alpha,\pi_\alpha(x_\beta))\rangle$ is finite-dimensional, since otherwise the flow $(W, \overline\phi_\alpha \restriction_{W_\beta})$ would be conjugated to the right Bernoulli shift $\beta_\Q^\oplus$, which would contradict $$h_{alg}(\overline\phi_\alpha \restriction_{W_\beta})\leq h_{alg}(\overline\phi_\alpha) \leq h_{alg}(\phi)<\infty = h_{alg}(\beta_\Q^\oplus)$$ (see Example \ref{beta} for the last equality). Then  $V_\beta = \pi_\alpha^{-1}(\langle T(\overline\phi_\alpha,\pi_\alpha(x_\beta))\rangle)$ is a subspace of $V$ containing $V_\alpha$ with $V_\beta/V_\alpha \cong W_\alpha$, so $\dim V_\beta/V_\alpha< \infty$. This ends up the inductive construction of the chain $\{V_\alpha\}_{\alpha<\kappa}$. 

For completeness, let $V_\kappa = V$ and for every $\alpha\leq \kappa$ let $\phi_\alpha=\phi\restriction_{V_\alpha}$. We prove by transfinite induction that $h_{alg}(\phi_\alpha)\in\mathbb A$. This is obvious for $\alpha = 0$ as $\f_0 =0$. Assume $0 < \beta \leq \kappa$ and $h_{alg}(\phi_\gamma)\in\mathbb A$ for all $\gamma < \beta$. Let $\overline \phi_{\beta,\gamma}:V_\beta/V_\gamma\to V_\beta/V_\gamma$ be the endomorphism induced by $\phi_\beta$ Êfor all $\gamma < \beta$. 
If $\beta=\gamma+1$ for some ordinal $\gamma$, then $h_{alg}(\phi_\beta)=h_{alg}(\phi_\gamma)+h_{alg}(\overline\phi_{\beta,\gamma})$ by the Addition Theorem \ref{AT}. As $h_{alg}(\phi_\gamma)\in\mathbb A$ by the inductive hypothesis, and $h_{alg}(\overline\phi_{\beta,\gamma})\in\mathbb A$ by $\dim V_\beta/V_\gamma<\infty $, we conclude that $h_{alg}(\phi_\beta)\in\mathbb A$ as well. 

If $\beta$ is a limit ordinal, then $V_\beta=\varinjlim_{\gamma<\beta} V_\alpha$ and $\phi_\beta=\varinjlim_{\gamma<\beta}\phi_\gamma$. So Proposition \ref{dirlim} gives $h_{alg}(\phi_\beta)=\sup_{\gamma<\beta}h_{alg}(\phi_\gamma)$; in particular $h_{alg}(\phi_\beta)$ is the limit of the non-decreasing net $\{h_{alg}(\phi_\gamma)\}_{\gamma<\beta}$. By the Addition Theorem \ref{AT}, $h_{alg}(\phi_\beta)=h_{alg}(\phi_\gamma)+h_{alg}(\overline\phi_{\beta,\gamma})$ for every $\gamma<\beta$. So, if the net $\{h_{alg}(\phi_\gamma)\}_{\gamma<\beta}$ converges properly to $h_{alg}(\phi_\beta)$, then the net $\{h_{alg}(\overline\phi_{\beta,\gamma})\}_{\gamma<\beta}$ converges properly to $0$, against the hypothesis that $\inf(\mathbf E_{alg}\setminus\{0\})>0$. Therefore there exists $\gamma<\beta$ such that $h_{alg}(\phi_\beta)=h_{alg}(\phi_\gamma)$. By inductive hypothesis $h_{alg}(\phi_\beta)\in\mathbb A$. Hence $h_{alg}(\phi)= h_{alg}(\phi_\kappa)\in\mathbb A$.

\smallskip
Finally, the inclusion ${\bf E}_{alg}\subseteq \mathbb A\cup\{\infty\}$ is proper since $\mathbb A$, contrary to ${\bf E}_{alg}\setminus \{\infty\}$, is dense in $\R_{\geq0}$. 
\end{proof}

If $\inf ({\bf E}_{alg}\setminus \{0\})>0$, then obviously $\log|\alpha|\not\in\mathbf E_{alg}$ whenever $\log|\alpha|< \inf({\bf E}_{alg}\setminus \{0\})$, so we have the following consequence.

\begin{corollary} The following conditions are equivalent:
\begin{itemize}
  \item[(a)] $\inf ({\bf E}_{alg}\setminus \{0\})>0$;
  \item[(b)] there exists an algebraic number $\alpha$ with $|\alpha| > 1$, such that $0< \log|\alpha|\not\in \mathbf E_{alg}$;
\end{itemize}
\end{corollary}

Recall that $\log|\alpha|\not\in\mathbf E_{alg}$ means that $\log|\alpha|$ cannot be the entropy of any abelian group endomorphism.

\medskip

Let us conclude with a variation of Problem \ref{h>0-pb} (equivalent to Lehmer's
problem) based on a different approach to \eqref{***}. 

For a class $\mathcal V$ of groups let 
$$
{\mathbf E}_{alg}^{\mathcal V} = \bigcup \{{\mathbf E}_{alg}(G): G \in \mathcal V \},
$$
so that ${\mathbf E}_{alg}^{\mathcal A} = {\mathbf E}_{alg}$ where ${\mathcal A}$ is the class of abelian groups. 
 By Example \ref{Exaaaample}(a), ${\mathbf E}_{alg}^{\mathcal{TA}} \setminus\{0\} = \log 2$ for  the class $\mathcal{TA}$ of torsion abelian groups. 

\begin{problem} Is there a class  ${\mathcal V}$ of groups containing $\mathcal{A}$ such that $\inf ({\mathbf E}_{alg}^{\mathcal V} \setminus\{0\} )= 0$ ? 
\end{problem}

\subsection{Algebraic entropy on locally compact groups}\label{alc-sec}

In \cite{Pet1} Peters gave a further generalization of the entropy he defined in \cite{P}, using Haar measure, for topological automorphisms of locally compact abelian groups. In \cite{V} Virili modified Peters's definition, in the same way as in \cite{DG} for the discrete case, obtaining a new notion of algebraic entropy for endomorphisms $\phi$ of locally compact abelian groups $G$. We give it now removing the hypothesis that the group has to be abelian.

Let $G$ be a locally compact group, $\mu$ a right Haar measure on $G$ and $U\in\mathcal C(G)$; the \emph{algebraic entropy of $\phi$ with respect to $U$} is $$H_{alg}(\phi,U)=\limsup_{n\to\infty}\frac{\log\mu(T_n(\phi,U))}{n},$$
and the \emph{algebraic entropy} of $\phi$ is $$h_{alg}(\phi)=\sup\{H_{alg}(\phi,U):U\in\mathcal C(G)\}.$$

\smallskip
The algebraic entropy defined in the discrete case is a particular case of this definition. Indeed, discrete groups are locally compact and in the discrete case it is possible to choose as Haar measure the cardinality of subsets. So we call this entropy still algebraic entropy and adopt the same notation $h_{alg}$ as in the discrete case.

\begin{example}\label{idG}
Let $G$ be a locally compact abelian group.
\begin{itemize}
\item[(a)]It was proved in \cite[Example 2.5]{V} that $h_{alg}(id_G)=0$ in case $G$ is an abelian group, as a consequence of \cite[Theorem 7.9]{EG} stating that, for $\mu$ a Haar measure on $G$, $U\in\C(G)$ and $n\in\N_+$, the map $\N_+\to \R_{\geq 0}\cup\{\infty\}$ given by $n\mapsto\mu(U_{(n)})$ is polynomial.
\item[(b)] If $\phi:G\to G$ is continuous endomorphism and $U\in\mathcal C(G)$ is such that $\phi(U)\subseteq U$, then $T_n(\f,U) = U_{(n)}$ for every $n\in\N_+$, so $H_{alg}(\phi,U)\leq H_{alg}(id_G,U)=0$ by item (a).
\end{itemize}
\end{example}

We recall in Lemma \ref{mon} and Proposition \ref{prop} the basic properties proved in \cite{V} and \cite{GV} in the abelian case, which hold also in the general case. We omit their proofs as they are the same as those given in the abelian case.

\begin{lemma}\label{mon}
Let $G$ be a locally compact group and $\phi:G\to G$ a continuous endomorphism.
\begin{itemize}
\item[(a)] If $C,C'\in\mathcal C(G)$ and $C\subseteq C'$, then $H_{alg}(\phi,C)\leq H_{alg}(\phi,C')$.
\item[(b)] If $\mathcal C$ is a cofinal subfamily of $\mathcal C(G)$, then $h_{alg}(\phi)=\sup\{H_{alg}(\phi,U):U\in\mathcal C\}$.
\item[(c)] If $H$ is an open $\phi$-invariant subgroup of $G$, then $H_{alg}(\phi\restriction_H,C)= H_{alg}(\phi,C)$ for every $C\in\c(H)$. In particular $h_{alg}(\phi\restriction_H)\leq h_{alg}(\phi)$.
\end{itemize}
\end{lemma}

Let  $G$ be a locally compact group, $\mu$ a right Haar measure on $G$, $\phi:G\to G$ a continuous endomorphism and $C\in \mathcal C(G)$. If the sequence $\left\{\frac{\log\mu(T_n(\phi,C))}{n}\right\}_{n\in\N_+}$ is convergent, we say that {\em the $\phi$-trajectory of $C$ converges}. In particular, if the $\phi$-trajectory of $C$ converges, then the $\limsup$ in the definition of $H_A(\phi,C)$ becomes a limit:
$$
H_{alg}(\phi,C)=\lim_{n\to\infty}\frac{\log\mu(T_n(\phi,C))}{n}.
$$

If $G$ is a compact or discrete group, then the $\phi$-trajectory of $C$ converges for every continuous endomorphism
$\phi:G\to G$ and $C\in\c(G)$. Indeed, the compact case is obvious as the values of the measure form a bounded subset of the reals (so the above sequence always converges to $0$). For the discrete case we refer to the previous section.

\begin{proposition}\label{prop}
 Let $G$ and $H$ be locally compact groups, $\phi:G\to G$ and $\psi:H\to H$ continuous endomorphisms. 
\begin{itemize}
\item[(a)]\emph{(Invariance under conjugation)} If there exists a topological isomorphism $\alpha : G \to H$ such that $\phi = \alpha^{-1}\psi\alpha$, then 
$H_{alg}(\phi,C)=H_{alg}(\psi,\alpha (C))$ for every $C\in\c(G)$. In particular $h_{alg}(\phi) = h_{alg}(\psi)$. 
\item[(b)] \emph{(Weak Addition Theorem)} For every $C_1\in\mathcal (G)$ and $C_2\in \mathcal C(H)$,
\begin{equation}\label{ATPP}
H_{alg}(\phi\times\psi,C_1\times C_2)\leq H_{alg}(\phi,C_1)+H_{alg}(\psi,C_2).
\end{equation}
In particular $h_{alg}(\phi \times \psi)\leq h_{alg}(\phi)+h_{alg}(\psi)$. Furthermore, if the $\phi$-trajectory of $C_1$ converges, then equality holds in \eqref{ATPP}.
 \item[(c)] Let $C\in\mathcal C(G)$. Then 
\begin{equation*}
H_{alg}(\phi\times \phi,C\times C)= 2H_{alg}(\phi,C).
\end{equation*}
In particular $h_{alg}(\phi\times \phi)=2 h_{alg}(\phi).$
\item[(d)] If $\phi$ is a topological automorphism, then $$h_{alg}(\phi^{-1} ) = h_{alg}(\phi)-\log (\mod_G(\phi)).$$ 
\item[(e)] \emph{(Continuity)} Suppose that $\{ H_i  :  i \in I\}$ is a directed system of open $\phi$-invariant subgroups of $G$ such that $G=\varinjlim H_i$. Then $h_{alg}(\phi)=\sup_{i\in I}h_{alg}(\phi\restriction_{H_i})$.
\end{itemize}
\end{proposition}

Item (d) shows that the equality $h_{alg}(\phi^{-1} ) = h_{alg}(\phi)$ may fail for an automorphism of a non-compact locally compact group. Peters's entropy $h_\infty(\phi)$, defined in  \cite{Pet1} for an automorphism $\f$ of a locally compact abelian group, is nothing else but $h_{alg}(\phi^{-1} )$ in our terms. Hence his notion substantially differs from the algebraic entropy considered here.

\medskip
Using absolutely similar arguments as in Proposition \ref{no-mu}, one obtains Proposition \ref{h_mu(U)-index1} giving a measure-free formula also for the algebraic entropy in totally disconnected locally compact groups. 

For a group $G$ and an endomorphism $\phi:G\to G$, the $n$-th $\phi$-trajectory $T_n(\phi,H)$ of a subgroup $H$ of $G$ need not be a subgroup. So we use the following generalization of the concept of index. For a group $G$, a subgroup $H$ of $G$ and a subset $Y$ of $G$ of the form $Y=HZ$ for some subset $Z$ of $G$, let $[Y:H]=|\{Hz:z\in Z\}|$ be the {\em index} of $H$ in $Y$.

Note that, if $G$ is a totally disconnected locally compact group, $\phi:G\to G$ a continuous endomorphism and $U\in\mathcal B(G)$, then $T_n(\phi,U)=U\phi(T_{n-1}(\phi,U))$ is a compact subset of $G$ while each coset $Ux$ is open, hence $[T_n(\phi,U):U]$ is finite.
After this observation we can state the following proposition.

\begin{proposition}\label{h_mu(U)-index1}
Let $G$ be a totally disconnected locally compact group, $\phi:G\to G$ a continuous endomorphism and $U\in\mathcal B(G)$. Then
$$H_{alg}(\phi,U)=\lim_{n\to\infty}\frac{\log [T_n(\phi,U):U]}{n}.$$
\end{proposition}

Similarly, one can remove the measure also in the definition of the modular function in a totally disconnected locally compact group $G$. 
Indeed, if $\f$ is an automorphism of $G$ and $U$ is an open compact subgroup of $G$, then $\mod(\f) = \frac{\mu(\f(U))}{\mu(U)}$. Here we can use the fact that $\f(U)$ is an open compact subgroup of $G$ as well, hence $U\cap \f(U)$ is an open subgroup in both $U$ and $\f(U)$. Since these subgroups are compact, both indexes $[U: U\cap \f(U)]$ and $[\f(U):U\cap \f(U)]$ are finite. Hence 
$$\mu(U) =[U: U\cap \f(U)] \cdot \mu(U\cap \f(U)) \mbox{ and }\mu(\f(U)) =[\f(U): U\cap \f(U)] \cdot \mu(U\cap \f(U)).$$
This gives 
$$\mod(\f) = \frac{\mu(\f(U))}{\mu(U)}=\frac{[\f(U): U\cap \f(U)]}{[U: U\cap \f(U)]}.$$



In the next example we compute the algebraic entropy of endomorphisms of $\Q_p$.

\begin{example}\label{ExaQ_p}
Let $p$ be a prime, $\xi \in \Q_p$ and let $\varphi_\xi: \Q_p \to \Q_p$ be defined by $\varphi_\xi(\eta) = \xi \eta$ for all $\eta \in \Q_p$. 

If $\xi \in \J_p$, then $h_{alg}(\varphi_\xi) =0$. Indeed, $\mathcal B(\Q_p)$ is given by the 
subgroups $p^k\J_p$ which are $\f$-invariant, and $\mathcal B(\Q_p)$ is cofinal in $\mathcal C(\Q_p)$, so Example \ref{idG}(b) and Lemma \ref{mon}(b) apply. 

If $\xi = p^{- k} \kappa$ for some $\kappa \in \J_p\setminus p\J_p$ and $k \in \N_+$, then $T_n(\varphi_\xi,\J_p) = p^{- k(n-1)}\J_p$ for every $n\in\N_+$. By Proposition \ref{h_mu(U)-index1}, $H_{alg}(\varphi_\xi,U) =\lim \frac{\log [T_n(\varphi_\xi,U):U]}{n}= k \log p$. Therefore $h_{alg}(\varphi_\xi)=k\log p$ by Lemma \ref{mon}(b).
\end{example}

The following example shows that 
that for $G$ a totally disconnected locally compact group and $\phi:G\to G$ a continuous endomorphism, the algebraic entropy of $\phi$ is in general not equal to $\sup\{H_{alg}(\phi,U):U\in\mathcal B(G)\}$, contrarily to the case of the topological entropy in Lemma \ref{B(G)}. By Lemma \ref{mon}(b) the equality holds true when $\mathcal B(G)$ is cofinal in $\mathcal C(G)$. When the locally compact group $G$ is abelian, this occurs precisely when $G$ has an open compact subgroup $K$ such that $G/K$ is (discrete) torsion (i.e., $G$ is a union of compact subgroups).

\begin{example}\label{noBG}
Let  $G = \Q_p \times D$, where $p$ is a prime and $D$ is $\Q_p$ endowed with the discrete topology. Let $\varphi_2:G\to G$ be the continuous endomorphism defined by $\varphi_2(x)=2x$ for every $x\in G$. Then $h_{alg} (\varphi_2\restriction_{\Q_p})=0$ by Example \ref{ExaQ_p}. 
Moreover $h_{alg}(\varphi_2\restriction_D)=\infty$. Then also $h_{alg}(\varphi_2)=\infty$ by Proposition \ref{prop}(b). 

On the other hand $\sup\{H_{alg}(\phi,U):U\in\mathcal B(G)\}=0$. Indeed, if $U$ is a compact open subgroup of $G$, then $\pi_2(U)$ is a compact subgroup of $D$, where $\pi_2:G\to D$ is the canonical projection, hence $\pi_2(U)=0$. Therefore $U\subseteq \Q_p$, that is, $H_{alg}(\phi,U)=0$, as 
$h_{alg} (\varphi_2\restriction_{\Q_p})=0$ (see above). 
\end{example}

In analogy with Example \ref{exlctop} one can consider the groups $\mathbb H_{\Q_p}$ and $
G=\begin{pmatrix}
\R_{>0} & \R  \\
0  & 1
\end{pmatrix}$,
 and compute the algebraic entropy of their conjugations. Again, in analogy with the topological case, one can consider the matrix group $G_K=\begin{pmatrix}
K^{*} & K  \\
0  & 1
\end{pmatrix}$
over a non-discrete locally compact field $K$, and calculate the algebraic entropy of the conjugations of $G_K$. More precisely one can ask the counterpart of the questions in Question \ref{conj2} for the algebraic entropy.

\medskip
The algebraic entropy for continuous endomorphisms of locally compact abelian groups was introduced in \cite{V} in order to prove the Algebraic Kolmogorov-Sinai Formula \ref{AKSF}. It was then used for the proof of the Algebraic Yuzvinski Formula \ref{AYF} in \cite{GV}. Indeed, in both cases one goes from the discrete case of endomorphisms of $\Z^n$ (respectively, $\Q^n$) to continuous endomorphisms of the locally compact abelian group $\R^n$ (respectively, $\Q_p^n$).

In particular the following result from \cite{V} holds.

\begin{theorem}\label{VF}
Let $n\in\N_+$ and $\phi:\R^n\to \R^n$ a continuous endomorphism. Then 
$$h_{alg}(\phi)=\sum_{|\lambda_i|>1}\log|\lambda_i| = h_U(\f),$$ 
where $\{\lambda_i:i=1,\ldots,n\}$ is the family of all eigenvalues of $\phi$.
\end{theorem}

\bigskip
For open problems see \cite{V} and \cite{GV}. Moreover we leave the following

\begin{problem}
Does the Addition Theorem for the algebraic entropy of continuous endomorphisms of locally compact (abelian) groups
 hold true ? 
\end{problem}

More details about the validity of the Addition Theorem for $h_{alg}$ and for $h_\infty$ can be found in \cite{DSV}.

\newpage
\section{Adjoint entropy}\label{adj-sec}

\subsection{Definition and basic properties}\label{adj-def}

In analogy to the algebraic entropy $\ent$, in \cite{DGS} the adjoint algebraic entropy of endomorphisms of abelian groups $G$ was introduced ``replacing" the family $\mathcal F(G)$ of all finite subgroups of $G$ with the family $\CC(G)$ of all finite-index subgroups of $G$. We give here the same definition in the more general setting of endomorphisms of arbitrary groups.

\smallskip
Let $G$ be a group and $N\in \CC(G)$. For an endomorphism $\phi:G\to G$ and $n\in\N_+$, the \emph{$n$-th $\phi$-cotrajectory of $N$}, as defined in \eqref{cotraj}, is $$C_n(\phi,N)=N\cap\phi^{-1}(N)\cap\ldots\cap\phi^{-n+1}(N).$$ The \emph{$\phi$-cotrajectory of $N$} is $$C(\phi,N)=\bigcap_{n\in\N}\phi^{-n}(N).$$ Then $C(\phi,N)$ is the maximum $\phi$-invariant subgroup of $N$.

Since the map induced by $\phi^n$ on the partitions $\{\phi^{-n}(N)g:g\in G\}\to \{Ng:g\in G\}$ is injective, it follows that $\phi^{-n}(N)\in\CC(G)$ for every $n\in\N$. Therefore $C_n(\phi,N)\in \CC(G)$ for every $n\in\N_+$, because $\CC(G)$ is closed under finite intersections. 

\begin{definition}
Let $G$ be a group and $\phi:G\to G$ an endomorphism.
\begin{itemize}
\item[(a)] The \emph{adjoint algebraic entropy of $\phi$ with respect to $N$} is 
\begin{equation}\label{H*}
H^\star(\phi,N)={\limsup_{n\to \infty}\frac{\log[G:C_n(\phi,N)]}{n}}.
\end{equation}
\item[(b)] The \emph{adjoint algebraic entropy of $\phi$} is $$\aent(\phi)=\sup\{H^\star(\phi,N):N\in\CC(G)\}.$$
\end{itemize}
\end{definition}

For a group $G$, denote by $\CC^\lhd(G)$ the family of all normal finite-index subgroups of $G$.

\begin{fact}\label{pullbacklhd}
Let $G$ be a group and $H$ a subgroup of $G$. 
\begin{enumerate}[(a)]
\item If $N\in\CC^\lhd(G)$, then $N\cap H\in\CC^\lhd(H)$.
\item If $\phi:G\to G$ is an endomorphism and $N\in\CC^\lhd(G)$, then $C_n(\phi,N)\in\CC^\lhd(G)$ for every $n\in\N_+$.
\end{enumerate}
\end{fact}

Recall that for a group $G$ and a subgroup $H$ of $G$, the \emph{heart} $H_G$ of $H$ in $G$ is the maximal normal subgroup of $G$ contained in $H$, that is, $H_G=\bigcap_{g\in G}g Hg^{-1}$.

\begin{fact}\label{heart} 
Let $G$ be a group and $H$ a subgroup of $G$. If $H\in\CC(G)$, then $H_G\in\CC^\lhd(G)$.
\end{fact}

In view of Fact \ref{heart}, we see in Lemma \ref{N<M->HN>HM} that to compute the adjoint algebraic entropy it suffices to consider normal finite-index subgroups.

\begin{lemma}\label{N<M->HN>HM}
Let $G$ be a group and $\phi:G\to G$ an endomorphism. Then $H^\star(\phi,-)$ is antimonotone, that is, $H^\star(\phi,N)\leq H^\star(\phi,M)$ if $M\subseteq N$ in $\CC(G)$.
Therefore $$\aent(\phi)=\sup\{H^\star(\phi,N):N\in\CC^\lhd(G)\}.$$
\end{lemma}
\begin{proof} 
The first part is clear, the second follows from Fact \ref{heart}.
\end{proof}

\begin{proposition}\label{C=C_n->ent*=0} Let $G$ be a group, $\phi:G\to G$ an endomorphism and $N\in\CC^\lhd(G)$. Then, for 
$D_n:=\frac{C_n(\phi,N)}{C_{n+1}(\phi,N)} $, the sequence $\alpha_n=\left|D_n\right|$ is stationary; more precisely, there exists a natural number $\alpha>0$ such that $\alpha_n= \alpha$ for all $n$ large enough. In particular,
\begin{itemize}
\item[(a)] the limit $H^\star(\phi,N)$ does exist and it coincides with $\log\alpha$, and
\item[(b)] $H^\star(\phi,N)=0$ if and only if $C(\phi,N)=C_n(\phi,N)$ for some $n\in\N_+$.
\end{itemize}
\end{proposition}
\begin{proof} Let $c_n=[G:C_n(\phi,N)]$ and $\gamma_{n}=\log c_n$ for every $n\in\N_+$. Since $C_{n+1}(\phi,N)$ is a normal subgroup of $C_n(\phi,N)$, it follows that 
$$
\frac{G}{C_{n}(\phi,N)}\cong\frac{\frac{G}{C_{n+1}(\phi,N)}}{D_n}.
$$ 
Then ${c_n}\alpha_{n}= c_{n+1}$, so $c_{n} | c_{n+1}$, for every $n\in\N_+$. 

We verify that
\begin{equation}\label{alphaeq}
\alpha_{n+1}|\alpha_n \ \text{for every}\ n\in\N,n>1.
\end{equation}
Fix $n\in\N$, $n>1$.  We intend to prove that $D_n$ is isomorphic to a subgroup of $D_{n-1}$ and so $\alpha_{n}|\alpha_{n-1}$.

First note that $D_n\cong \frac{C_n(\phi,N)\cdot\phi^{-n}(N)}{\phi^{-n}(N)}$.
From 
$$
C_n(\phi,N)=N\cap \phi^{-1}(C_{n-1}(\phi,N))\subseteq\phi^{-1}(C_{n-1}(\phi,N)),
$$ 
it follows that 
$$
\frac{C_n(\phi,N)\cdot\phi^{-n}(N)}{\phi^{-n}(N)}\subseteq A_n=\frac{\phi^{-1}(C_{n-1}(\phi,N))\cdot\phi^{-n}(N)}{\phi^{-n}(N)}.
$$
 Since the homomorphism $\tilde\phi:\frac{G}{\phi^{-n}(N)}\to \frac{G}{\phi^{-n+1}(N)}$, induced by $\phi$, is injective, also its restriction to $A_n$ is injective, and the image of $A_n$ is contained in $L_n= \frac{C_{n-1}(\phi,N)\cdot\phi^{-n+1}(N)}{\phi^{-n+1}(N)}$, which is isomorphic to $D_{n-1}$. Summarizing, 
 $$
D_n\cong\frac{C_n(\phi,N)\cdot\phi^{-n}(N)}{\phi^{-n}(N)}\leq A_n \rightarrowtail L_n\cong D_{n-1},$$
which concludes the proof of \eqref{alphaeq}.

Clearly $c_{n+1}=c_1\cdot\alpha_2\cdot\ldots\cdot\alpha_n$, as $c_{n+1}=\alpha_{n}\cdot c_n$, for every $n\in\N_+$. By \eqref{alphaeq} there exist $m\in\N_+$ and $\alpha\in\N_+$ such that $\alpha_n=\alpha$ for every $n\geq m$. Let $a_0=c_1\cdot \alpha_2\cdot \ldots\cdot \alpha_m$, so that $c_{n+1}=c_n\alpha$ 
(and consequently $\gamma_{n+1}=\gamma_n + \log \alpha$) for all $n\geq m$. Therefore
\begin{equation}\label{(*)}
c_n=a_0\cdot\alpha^{n-N}, \mbox{ and so } \gamma_n= \log a_0 + (n-N)\log \alpha,  \mbox{ for all } n\geq m.
\end{equation}
From \eqref{(*)} one can immediately see that $\lim_{n\to\infty} \frac{\gamma_n}{n} = \log \alpha$ (i.e., the limit in \eqref{H*} exists and coincides with $\log\alpha$). 

Note that in \eqref{(*)} either $\alpha> 1$ or $\alpha=1$; in the latter case the sequence $0<c_1\leq c_2\leq\ldots$ is stationary, or equivalently $C(\phi,N)=C_n(\phi,N)$ for some $n\in\N_+$.
\end{proof}

We start with some basic examples.

\begin{example}\label{examplestar}
\begin{itemize}
\item[(a)] For any group $G$, $\ent^\star(id_G)=\ent^\star(0_G)=0$.
\item[(b)] If an endomorphism $\phi:G\to G$ is nilpotent (i.e., there exists $n\in\N_+$ such that $\phi^n=0$) or periodic (i.e., there exists $n\in\N_+$ such that $\phi^n=id_G$), then $\ent^\star(\phi)=0$.
\item[(c)] If $\phi$ is quasi-periodic (i.e., there exist $n,m\in\N_+$, $n\neq m$, such that $\phi^n=\phi^m$), then $\ent^\star(\phi)=0$.
\item[(d)] If $D$ is a divisible group, then $D$ has no proper subgroups of finite index, and so $\ent^\star(\phi)=0$ for every endomorphism $\phi:D\to D$. Note that in the abelian case, $D$ is divisible if and only if  $\CC(D)=\{D\}$ (this fail in the non-abelian case, take for example $S(X)$ for some infinite set $X$).
\end{itemize}
\end{example}

Moreover the following general property holds.

\begin{lemma}\label{phi-inv}
Let $G$ be a group and $\phi:G\to G$ an endomorphism. If $N$ is a $\phi$-invariant normal subgroup of $G$ of finite index, then $H^\star(\phi,N)=0$.
\end{lemma}
\begin{proof}
Since $\phi(N)\subseteq N$, it follows that $C(\phi,N)=N$, hence the thesis.
\end{proof}

We give now the basic properties of the adjoint algebraic entropy. They were proved in \cite{DGS} in the abelian case.

\begin{lemma}[Invariance under conjugation]\label{cbi}
Let $G$ be a group and $\phi:G\to G$ an endomorphism. If $H$ is another group and $\xi:G\to H$ an isomorphism, then $\ent^\star(\xi\phi\xi^{-1})=\ent^\star(\phi)$.
\end{lemma}
\begin{proof}
Let $N$ be a normal subgroup of $H$ and call $\theta=\xi\phi\xi^{-1}$. Then $\xi^{-1}(H)$ is normal in $G$ and:
\begin{itemize}
\item[-] $H/N$ is finite if and only if $G/\xi^{-1}(N)$ is finite;
\item[-] $C_n(\theta,N)=\xi(C_n(\phi,\xi^{-1}(N))$ for every $n\in\N_+$, since $\theta^n=\xi\phi^n\xi^{-1}$ for every $n\in\N$.
\end{itemize}
Therefore for every $n\in\N_+$, 
$$\left|\frac{G}{C_n(\theta,N)}\right|=\left|{\frac{\xi(G)}{\xi(C_n(\phi,\xi^{-1}(N)))}}\right|=\left|{\frac{G}{C_n(\phi,\xi^{-1}(N))}}\right|.$$ Consequently $H^\star(\theta,N)=H^\star(\phi,\xi^{-1}(N))$ for every $N\in\CC^\lhd(H)$, and hence $\ent^\star(\theta)=\ent^\star(\phi)$.
\end{proof}

\begin{lemma}[Logarithmic Law]\label{ll}
\begin{itemize}
\item[(a)] Let $G$ be a group and $\phi:G\to G$ an endomorphism. Then $\ent^\star(\phi^k)=k\cdot \ent^\star(\phi)$ for every $k\in\N_+$.
\item[(b)] If $\phi\in\Aut(G)$, then $\ent^\star(\phi)=\ent^\star(\phi^{-1})$. Consequently $\ent^\star(\phi^k)=|k|\cdot \ent^\star(\phi)$ for every $k\in\Z$.
\end{itemize}
\end{lemma}
\begin{proof}
(a) For $N\in\CC^\lhd(G)$, fixed $k\in\N$, for every $n\in\N$ we have
\begin{equation}\label{Cnkf=Cnfk}
{C_{nk}(\phi,N)}=C_n(\phi^k,C_k(\phi,N)).
\end{equation}
Then
\begin{align*}
H^\star(\phi,N)&=\lim_{n\to\infty}\frac{\log[G:C_{nk}(\phi,N)]}{nk}=\lim_{n\to\infty}\frac{\log[G:C_n(\phi^k,C_k(\phi,N))]}{nk}=\\
&=\frac{H^\star(\phi^k,C_k(\phi,N))}{k}\geq \frac{H^\star(\phi^k,N)}{k},
\end{align*}
where in the second equality we have applied \eqref{Cnkf=Cnfk} and in the last inequality Lemma \ref{N<M->HN>HM}. 
The equality $k\cdot H^\star(\phi,N)=H^\star(\phi^k,C_k(\phi,N))$ implies $k\cdot \ent^\star(\phi)\leq\ent^\star(\phi^k)$, while the inequality $k\cdot H^\star(\phi,N)\geq H^\star(\phi^k,N)$ yields $k\cdot \ent^\star(\phi)\geq \ent^\star(\phi^k)$ and this concludes the proof.

\smallskip
(b) For every $n\in\N_+$ and every $N\in\CC^\lhd(G)$, we have
\begin{align*} 
\phi^{n-1}(C_n(\phi,N)) &=\phi^{n-1}(N\cap\phi^{-1}(N)\cap\ldots\cap\phi^{-n+1}(N))=\\
&=\phi^{n-1}(N)\cap\phi^{n-2}(N)\cap\ldots\cap \phi (N)\cap N=C_n(\phi^{-1},N),
\end{align*}
and so 
$$[G:C_n(\phi,N)]=\left|\frac{G}{C_n(\phi,N)}\right|=\left|{\frac{\phi^{n-1}(G)}{\phi^{n-1}(C_n(\phi,N))}}\right|=\left|\frac{G}{C_n(\phi^{-1},N)}\right|=[G:C_n(\phi^{-1},N)].$$ 
Therefore $H^\star(\phi,N)=H^\star(\phi^{-1},N)$, and hence $\ent^\star(\phi)=\ent^\star(\phi^{-1})$.
\end{proof}

Now we give various weaker forms of the Addition Theorem for the adjoint algebraic entropy.

\begin{lemma}[Monotonicity for quotients]\label{quotient}
Let $G$ be a group, $\phi:G\to G$ an endomorphism and $H$ a $\phi$-invariant normal subgroup of $G$. Then $\ent^\star(\phi)\geq \ent^\star(\overline\phi)$, where $\overline \phi:G/H\to G/H$ is the endomorphism induced by $\phi$.
\end{lemma}
\begin{proof}
Let $N/H\in\CC^\lhd(G/H)$; then $N\in\CC^\lhd(G)$. Fixed $n\in\N_+$,
$$C_n\left(\overline \phi, \frac{N}{H}\right)\supseteq \frac{C_n(\phi,N)\cdot H}{H},$$
and so 
$$\left|\frac{\frac{G}{H}}{C_n\left(\overline\phi,\frac{N}{H}\right)}\right|\leq 
\left|\frac{\frac{G}{H}}{\frac{C_n(\phi,N)\cdot H}{H}}\right|=\left|\frac{G}{C_n(\phi,N)\cdot H}\right|\leq \left|\frac{G}{C_n(\phi,N)}\right|.$$
This yields $H^\star(\overline\phi, N/H)\leq H^\star(\phi,N)$ and so $\ent^\star(\overline\phi)\leq\ent^\star(\phi)$.
\end{proof}

The adjoint entropy is preserved under taking some special quotients. More precisely let $R(G) = \bigcap_{N\in \CC^\lhd(G)}N$ (this is the smallest normal subgroup of $G$ such that $G/R(G)$ is residually finite). Now, if $H$ is a $\phi$-invariant normal subgroup of $G$
contained in $R(G)$, then $\ent^\star(\phi)= \ent^\star(\overline\phi)$, as every $N \in \CC(G)$ contains $R(G)$. This fact reduces the study of the adjoint algebraic entropy to endomorphisms of residually finite groups.

In general $\ent^\star$ fails to be monotone with respect to taking restrictions to $\phi$-invariant subgroups (i.e., if $G$ is an abelian group, $\phi:G\to G$ an endomorphism and $H$ a $\phi$-invariant subgroup of $G$, then the inequality $\ent^\star(\phi)\geq\ent^\star(\phi\restriction_H)$ may fail), as the following easy example shows. 

\begin{example}\label{noAT*}\cite{DGS}
Let $G$ be an abelian group that admits an endomorphism $\phi:G\to G$ with $\ent^\star(\phi)>0$. Then the divisible hull $D$ of $G$ has $\ent^\star(D)=0$ by Example \ref{examplestar}(d).
\end{example}

\begin{lemma}[Monotonicity for subgroups]\label{subgroup}
Let $G$ be a group, $\phi:G\to G$ an endomorphism and $H$ a $\phi$-invariant subgroup of $G$. If $H\in\CC(G)$, then $\ent^\star(\phi)=\ent^\star(\phi\restriction_H)$.
\end{lemma}
\begin{proof}
Let $N\in\CC^\lhd(H)$. Since $N$ has finite index in $H$ and $H$ has finite index in $G$, it follows that $N$ has finite index in $G$ as well. Then $N_G\in\CC^\lhd(G)$. Since $N_G\subseteq N$, Lemma \ref{N<M->HN>HM} yields $H^\star(\phi,N)\leq H^\star(\phi,N_G)$ and $H^\star(\phi\restriction_H,N)\leq H^\star(\phi\restriction_H,N_G)$. It is possible to prove by induction on $n\in\N_+$ that $$C_n(\phi\restriction_H,N_G)=C_n(\phi,N_G)\cap H.$$
Then, for every $n\in\N_+$,
\begin{equation*}
\frac{G}{C_n(\phi,N_G)} \geq\frac{H\cdot C_n(\phi,N_G)}{C_n(\phi,N_G)}\cong\frac{H}{C_n(\phi,N_G)\cap H}=\frac{H}{C_n(\phi\restriction_H,N_G)},
\end{equation*}
and so $\ent^\star(\phi)\geq H^\star(\phi,N_G)\geq H^\star(\phi\restriction_H,N_G)\geq H^\star(\phi\restriction_H,N)$, that implies $\ent^\star(\phi)\geq\ent^\star(\phi\restriction_H)$.

Let now $M\in\CC^\lhd(G)$. Then $N=M\cap H\in\CC^\lhd(H)$. For $c_n=[G:C_n(\phi,N_G)]$ and $c_n'= [H:C_n(\phi\restriction_H,N_G)]$ we proved in the first part of the proof that $c_n \geq c_n'$. On the other hand, one can easily see that $c_n/c_n'\leq [G:H]$ is bounded. Therefore 
$$
H^*(\phi\restriction_{H},N_G)=\lim_{n\to \infty}\frac{\log c_n'}{n}=\lim_{n\to\infty} \frac{\log c_n}{n}=H^\star(\phi,N_G)\geq H^\star(\phi,N)\geq H^\star(\phi,M).
$$
We can conclude that $\ent^\star(\phi)=\ent^\star(\phi\restriction_H)$.
\end{proof}

The next result shows the additivity of the adjoint algebraic entropy for finite direct products. 

\begin{lemma}[Weak Addition Theorem]\label{poorAT}
Let $G_1$ and $G_2$ be groups, and $\phi_i:G_i\to G_i$ an endomorphism for $i=1,2$.
Then $\ent^\star(\phi_1\times\phi_2) = \ent^\star(\f_1) + \ent^\star(\f_2)$.   
\end{lemma}
\begin{proof}
By Fact \ref{pullbacklhd}, every $N\in\CC^\lhd(G)$ contains $N'=N_1\times N_2$, where $N_i=N\cap G_i\in\CC^\lhd(G_i)$ for $i=1,2$. Moreover $C_n(\f, N')\cong C_n(\f_1,N_1)\oplus C_n(\f_2,N_2)$. 
\end{proof}

\subsection{Examples, Addition Theorem and Dichotomy}\label{adj-mp}

We collect in this subsection examples and fundamental results from \cite{DGS} about the adjoint algebraic entropy in the abelian case.

\begin{example}
Let $G$ be an abelian group.
\begin{itemize}
\item[(a)] Let $m\in\Z$ and $\varphi_m:G\to G$ the endomorphism of $G$ defined by $x\mapsto mx$ for every $x\in G$. Then $\ent^\star(\varphi_{m})=0$, since all subgroups of $G$ are $\varphi_{m}$-invariant and so Lemma \ref{phi-inv} applies.
\item[(b)] If $\End(G)\subseteq \mathbb Q$, then $\ent^\star(\phi)=0$ for every endomorphism $\phi:G\to G$.
\end{itemize}
\end{example}

The values of the adjoint algebraic entropy of the Bernoulli shifts were calculated in \cite[Proposition 6.1]{DGS} applying \cite[Corollary 6.5]{G0} and the Pontryagin duality; a direct computation can be found in \cite{G}. So, in contrast with what occurs for the algebraic entropy, we have:

\begin{example}\label{beta*}
For $K=\Z(p)$, where $p$ is a prime, $\ent^\star(\beta_K^\oplus)=\ent^\star( {}_K\beta^\oplus)=\ent^\star(\overline\beta_K^\oplus)=\infty.$
\end{example}

As an application of Example \ref{beta*} we can now give an example witnessing  the lack of continuity of the adjoint algebraic entropy
under taking inverse limits. 
 
\begin{example}\label{no-lim}\cite{DGS} 
Let $p$ be a prime, $G=\Z(p)^\N$ and consider $\beta_{\Z(p)}:G\to G$. For every $i\in\N$, let $$H_i=\underbrace{0\times\ldots\times 0}_i\times\Z(p)^{\N\setminus\{0,\ldots,i-1\}}\subseteq G.$$ Each subgroup $H_i$ is $\beta_{\Z(p)}$-invariant. The induced endomorphism $\overline{\beta_{\Z(p)}}_i:G/H_i\to G/H_i$ has $\ent^\star(\overline{\beta_{\Z(p)}}_i)=0$, since $G/H_i$ is finite. Moreover $G=\displaystyle\lim_{\longleftarrow}G/H_i$, as $\{(G/H_i,p_i)\}_{i\in\N}$ is an inverse system, where $p_i:G/H_{i+1}\cong\Z(p)^{i+1}\to G/H_{i}\cong \Z(p)^i$ is the canonical projection for every $i\in\N$.
By Example \ref{beta*} $\ent^\star(\beta_{\Z(p)})=\infty$, while $\sup_{i\in\N}\ent^\star(\overline{\beta_{\Z(p)}}_i)=0$.
\end{example}

\begin{theorem}[Addition Theorem]\label{AT*} 
Let $G$ be a bounded abelian group, $\phi:G\to G$ an endomorphism, $H$ a $\phi$-invariant subgroup of $G$ and $\overline\phi:G/H\to G/H$ the endomorphism induced by $\phi$. Then $$\ent^\star(\phi)=\ent^\star(\phi\restriction_H)+\ent^\star(\overline\phi).$$
\end{theorem}

For the adjoint algebraic entropy, the Weak Addition Theorem holds in general by Lemma \ref{poorAT}. On the other hand, the Monotonicity for invariant subgroups fails even for torsion abelian groups by Example \ref{noAT*}; in particular, the Addition Theorem fails in general. Moreover the argument used in \cite{DGS} in the proof of Theorem \ref{AT*} does not work out of the class of bounded abelian groups. 

\medskip
While the algebraic entropy may take finite positive values, the adjoint algebraic entropy takes values only in $\{0,\infty\}$ as shown by \cite[Theorem 7.6]{DGS}:

\begin{theorem}[Dichotomy Theorem]\label{dichotomy}
Let $G$ be an abelian group and $\phi:G\to G$ an endomorphism. Then 
\begin{center}
either $\ent^\star(\phi)=0$ or $\ent^\star(\phi)=\infty$.
\end{center}
\end{theorem}

We leave open the problem of extending this result out of the realm of abelian groups:

\begin{problem}
Does the Dichotomy Theorem for the adjoint algebraic entropy hold also in the general case of arbitrary groups?
What about nilpotent groups? And the groups belonging to $\mathfrak P$?
\end{problem}

\subsection{Topological adjoint entropy}\label{top-adj}

The particular ``binary behavior'' of the values of the adjoint algebraic entropy (given by Theorem \ref{dichotomy}) seems to be caused by the fact that the family of finite-index subgroups can be very large. So in \cite{G} only a part of it was taken making recourse to an appropriate topology in the following way. For a topological group $(G,\tau)$, consider the subfamily $$\CC_\tau(G)=\{N\in\CC(G):N\ \text{$\tau$-open}\}$$ of $\CC(G)$ consisting of all $\tau$-open finite-index subgroups of $G$. 

\begin{definition}
For a topological group $(G,\tau)$ and a continuous endomorphism $\psi:(G,\tau)\to (G,\tau)$, the \emph{topological adjoint entropy} of $\psi$ with respect to $\tau$ is
\begin{equation}\label{5}
\ent_{\tau}^\star(\psi)=\sup \{ H^\star(\psi,N):N\in\CC_\tau(G)\}.
\end{equation}
\end{definition}

Roughly speaking, $\aent_{\tau}$ is a variant of $\aent$ but taken only with respect to \emph{some} finite-index subgroups, namely, the $\tau$-open ones. Clearly, for $(G,\tau)$ a topological group and $\psi:(G,\tau)\to(G,\tau)$ a continuous endomorphism, $\aent(\psi)\geq\aent_\tau(\psi)$.
For the discrete topology $\delta_G$ of $G$, we have $\aent_{\delta_G}(\psi)=\aent(\psi)$, so the notion of topological adjoint entropy extends that of adjoint algebraic entropy. 

\medskip
For a compact group $(K,\tau)$, let $\CC_\tau^\lhd(K)$ be the subfamily of $\CC_\tau(K)$ consisting of all normal $\tau$-open subgroups of $K$.
If $H$ is a $\tau$-open subgroup of $K$, then $H_K$ is $\tau$-open as well. Therefore if $H\in\CC_\tau(K)$, then $H_K\in\CC_\tau^\lhd(K)$.
This observation, in view of the antimonotonicity of $H^\star(\psi,-)$ proved in Lemma \ref{N<M->HN>HM}, gives the possibility to consider only normal $\tau$-open subgroups for the computation of the topological adjoint entropy:

\begin{lemma}\label{normal}
Let $(K,\tau)$ be a compact group and $\psi:(K,\tau)\to (K,\tau)$ a continuous endomorphism. Then 
$$
\aent_\tau(\psi)=\sup\{H^\star(\psi,N):N\in\CC_\tau^\lhd(K)\}.
$$
\end{lemma}

\medskip
As pointed out in \cite{G}, the topological adjoint entropy appeared in some non-explicit way in the proof of the classical Bridge Theorem of Weiss.
More precisely, from the proof of that theorem one can deduce the following

\begin{theorem}\label{ent*=htop}\emph{\cite{G}}
If $(G,\tau)$ is a totally disconnected compact abelian group and $\psi:(G,\tau)\to(G,\tau)$ is a continuous endomorphism, then $$\aent_\tau(\psi)=h_{top}(\psi).$$
\end{theorem}

In particular the topological adjoint entropy takes every value in $\log\N_+\cup\{\infty\}$. Therefore the topological adjoint entropy comes as an alternative form of topological entropy for continuous endomorphisms of topological groups, with the advantage that it is defined for every continuous endomorphism of every topological group, and not only in the compact case.

\newpage
\section{Connections among entropies}\label{conn-sec}

\subsection{Measure entropy and topological entropy}\label{mes-top}

There is a nice connection between the measure entropy and topological entropy in compact metric spaces.
More precisely, if $X$ is a compact metric space and $\psi: X \to X$ is a continuous surjective selfmap, then the set $M(X,\psi)$ of all $\psi$-invariant
Borel probability measures $\mu$ on $X$ (i.e., making $\psi:(X,\mu) \to (X,\mu)$ measure preserving) is non-empty; this fact is known as Krylov-Bogolioubov Theorem \cite{BK} (see also \cite[Corollary 6.9.1]{Wa}). 

Denote by $h_\mu$ the measure entropy with respect to $\mu\in M(X,\psi)$.
The inequality $h_{\mu}(\psi)\leq h_{top}(\psi)$ for every $\mu \in M(X,\psi)$ is due to Goodwyn \cite{Goo}. Moreover the {\em variational principle} (\cite[Theorem 8.6]{Wa}) holds true:
$$
h_{top}(\psi)=  \sup \{h_{\mu}(\psi):  \mu\in M(X,\psi)\}.
$$
The surjectivity of $\psi$ is important, since the continuous image $\psi(X)$ of the compact space $X$ is compact, so closed in $X$. 
If $\psi(X) \ne X$, then for the open non-empty set $U=X \setminus \psi(X)$ of $X$ and $\mu \in M(X,\psi)$ one would obtain $\mu(U) = 0$, as
$1 = \mu(X) = \mu(\psi(X))$; this is a contradiction, since $\mu(U)>0$ by the regularity of $\mu$. 

\smallskip
Passing to the case of groups, it is well known that every compact group has its unique Haar measure. Halmos \cite{Halmos} noticed that in this case surjectivity is not only a necessary, but also a sufficient condition for a continuous endomorphism to be measure preserving: 

\begin{fact}[Halmos' Paradigm]\label{Halmos}
Let $K$ be a compact group and $\psi:K\to K$ a surjective continuous endomorphism. Then $\psi$ is measure preserving.
\end{fact}

The proof easily follows from the uniqueness of the Haar measure of $K$. 

\smallskip
According to Fact \ref{Halmos}, both $h_{top}$ and $h_{mes}$ are available for surjective continuous endomorphisms of compact groups.
It was proved that they coincide by Berg \cite{Berg} for surjective endomorphisms of tori, by Aoki \cite{Aoki} for automorphisms of compact abelian groups. The proof of the general case was obtained by Stoyanov \cite{S} as an application of the Uniqueness Theorem \ref{UT-top}.

\begin{theorem}\emph{\cite{S}}\label{A-S}
Let $K$ be a compact group and $\psi:K\to K$ a surjective continuous endomorphism. Then $h_{mes}(\psi)=h_{top}(\psi)$.
\end{theorem}

For other results in this direction see the paper of Berg \cite{Berg}, showing that Haar measure maximizes the measure theoretic entropy of a continuous automorphism of a compact metrizable group, and under conditions of finiteness and ergodicity it does so uniquely.

\subsection{Topological entropy and algebraic entropy}\label{top-alg}

Recall that the Pontryagin dual of a torsion abelian group is a totally disconnected compact abelian group, and that the Pontryagin dual of a countable abelian group is a metrizable compact abelian group. Weiss \cite{We} and Peters \cite{P} proved respectively the two items of the following result.

\begin{theorem}\emph{\cite{W,P}}\label{BT-WP}
If $\phi:G\to G$ is an endomorphism of an abelian group $G$, then $h_{alg}(\phi)=h_{top}(\widehat{\phi})$ provided one of the following conditions hold: 
\begin{itemize}
\item[(a)] $G$ is torsion; or 
\item[(b)] $G$ is countable and $\phi$ is an automorphism.  
\end{itemize}
\end{theorem}

In \cite{DG-bridge} this theorem was generalized to all endomorphisms of all abelian groups by means of a direct proof, using the properties of the topological entropy and the algebraic entropy. The core of the proof is the reduction to the case of an automorphism of $\Q^n$, where the Yuzvinski Formula \ref{YF} and the Algebraic Yuzvinski Formula \ref{AYF} apply. Using the Uniqueness Theorem \ref{UT} for the algebraic entropy one can obtain an alternative proof:

\begin{theorem}[Bridge Theorem]\emph{\cite{DG-bridge}}\label{BT}
Let $G$ be an abelian group and $\phi:G\to G$ an endomorphism. Then $$h_{alg}(\phi)=h_{top}(\widehat \phi).$$
\end{theorem}

\begin{proof}
Define $h_G$ on $\End(G)$ by letting $h_G(\f)= h_{top}(\widehat \phi)$ for every $\f \in \End(G)$. In order to apply Uniqueness Theorem \ref{UT} for $h_{alg}$, that gives the required equality, we have to check for $h_G$ the five axioms characterizing $h_{alg}$. This is clear as in Section \ref{top-sec} it is proved that $h_{top}$ satisfies Invariance under conjugation, Continuity on inverse limits and Addition Theorem, as well as the Yuzvinski Formula and Bernoulli normalization, so Pontryagin duality applies.
\end{proof}

At this stage one has to consider the following open question, asking for a generalization of the Bridge Theorem to all continuous endomorphisms of locally compact abelian groups.

\begin{problem}
Does the Bridge Theorem extend to all locally compact abelian groups?
\end{problem}

Theorem \ref{BF} and Theorem \ref{VF} give immediately that the Bridge Theorem holds for every continuous endomorphism of $\R^n$.

\begin{example}\label{ExaBTQ_p}
Let us check now that the Bridge Theorem holds for endomorphisms of the group $\Q_p$. Indeed, let $\f: \Q_p \to \Q_p$ be a continuous endomorphism. Then there exists some $\xi \in \Q_p$  such that $\f(\eta) = \xi \eta$ for all $\eta \in \Q_p$. By Example \ref{ExaQ_p}(a), $h_{alg}(\f) =0$ if $\xi\in\J_p$. Otherwise, if $\xi = p^{- k} \kappa$ for some $\kappa \in \J_p\setminus p\J_p$ and $k \in \N_+$, then $h_{alg}(\f) =k \log p$. Note that in the first case one can deduce also $k(\f) = 0$ by Remark \ref{Base}(c).Thus, $h_{alg}(\f)= k(\f) = 0$ in this case. In the second case, one can use the fact that $\f^{-1}(\J_p) = p^k\J_p$ to compute $C_n(\f, \J_p) = p^{k(n-1)}\J_p$, so $[\J_p: C_n(\f, \J_p)]= p^{k(n-1)}$. This gives $k(\f)=k(\f, \J_p) = k\log p$. So, the equality $h_{alg}(\f)= k(\f)$ holds true always. Since $\Q_p$ is autodual, i.e., $\widetilde \Q_p \cong \Q_p$, and $\widetilde \f$ is conjugated to $\f$, this shows that the Bridge Theorem holds for $\Q_p$. 
\end{example}

This example can be extended to $\Q_p^n$ for every $n\in\N_+$, in view of \cite[Fact A]{GV} for the algebraic entropy and \cite[Theorem 2]{LW} for the topological entropy. With this motivation 
and using the formulas in Proposition \ref{no-mu} and Proposition \ref{h_mu(U)-index1}, as well as the nice properties of Pontryagin duality, and arguing as in the proof of Weiss Bridge Theorem (see \cite{G} or \cite{DSV}), it is possible to prove the following  

\begin{theorem}\emph{\cite{DG-bridge}}
Let $G$ be a totally disconnected locally compact abelian group such that $\widehat G$ is totally disconnected as well, and let $\psi:G\to G$ be a continuous endomorphism. Then $k(\psi)=h_{alg}(\widehat\psi)$. 
\end{theorem}

\begin{remark} 
As already mentioned above, Peters introduced in \cite{Pet1} an entropy for a locally compact abelian group $G$
with a Haar measure $\mu$ and  a topological automorphism $\psi:G\to G$, letting 
$$
h_\infty(\psi)=\limsup \left\{\lim_{n\to\infty}\frac{\log\mu(T_n(\psi^{-1},C))}{n}:C\in\c(G)\right\}.
$$
In other words $h_\infty(\psi)=h_{alg}(\psi^{-1})$. According to Proposition \ref{prop}(d),  
$h_{alg}(\psi)\neq h_{\infty}(\psi)$ for a topological automorphism $\psi$ of a locally abelian group $G$ precisely when $\mod_G(\psi)\neq 1$.
Therefore,  in the general case of continuous endomorphisms $\f$ of locally compact abelian groups (i.e., with $\mod_G(\psi)\neq 1$), the Bridge Theorem cannot simultaneously hold true for the algebraic entropy and for Peters's entropy.
%
%
%
For more details on this discussion see \cite{DSV}.
\end{remark}

As an application of the Bridge Theorem we prove Theorem \ref{Cartagena}. 

\medskip
\noindent {\bf Proof of Theorem \ref{Cartagena}.} A complete proof of this theorem will appear in \cite{ADS}. Here we provide a brief sketch of the proof for reader's convenience. 

Item (a) and the equivalence (b$_1$)$\Leftrightarrow$(b$_2$) can be deduced from \cite{DGSZ} as follows. 
According to \cite[Corollary 1.20]{DGSZ}, if a torsion abelian group admits an endomorphism with positive algebraic entropy, then  it admits also an endomorphism with infinite algebraic entropy. Now the Bridge Theorem  \ref{BT} applies to give (a). 

As far as the  implication (b$_1$)$\Rightarrow$(b$_2$) is concerned, one can 
note first that $K = \prod_p K_p$, where each $K_p$ is a pro-$p$-group, so $G_p=\widehat{K}_p$ is a discrete $p$-group. The assumption ${\bf E}_{top}(K)\ne  \{0,\infty\}$, along with the Addition Theorem \ref{AT-top} and the Continuity on inverse limits \ref{invlim}, implies that  ${\bf E}_{top}(K_p)\ne  \{0,\infty\}$
for some prime $p$. By the Bridge Theorem  \ref{BT}, the group $G_p$ has an endomorphism $\phi$ with finite positive algebraic entropy $\ent(\phi)$. 
Let $D$ be the maximal divisible subgroup of $G_p$, so that $G_p = D \oplus R$, where $R$ is a reduced $p$-group. Since $D$ is fully invariant, 
$\phi$ induces and endomorphism $f = \phi\restriction _D$. As a divisible group, $D$ satisfies ${\bf E}_{alg}(D)\subseteq  \{0,\infty\}$ (see Example \ref{Exaaaample}(b)). Hence, 
$\ent(\phi)< \infty$ implies $\ent(f)=0$. Let $\bar \phi: R \to R$ be the induced endomorphism. The Addition Theorem \ref{AT} implies $\ent( \bar \phi) = \ent(\phi) > 0$, finite. According to \cite[Theorem 1.19]{DGSZ}, a reduced $p$-group with an endomorphism of finite positive entropy has an infinite bounded direct summand. Hence $R$ (so also $G_p$) has a direct summand of the form $\Z(p^n)^{(\N)}$ for some $n\in \N$ and some prime $p$, so $K_p$ (hence, also $K$) has a direct summand of the form $\Z(p^n)^\N$. 

The implication  (b$_2$)$\Rightarrow$(b$_1$) uses the fact that the class $\mathfrak A$ of compact abelian groups $K$ with ${\bf E}_{top}(K)\ne  \{0,\infty\}$
has the property that if a group $L$ has a (topological) direct summand from $\mathfrak A$, then $L \in \mathfrak A$. Since $\Z(p^n)^\N \in \mathfrak A$, we are done. 

(b$_4$)$\Leftrightarrow$(b$_2$) follows from Pontryagin duality, while (b$_3$)$\Rightarrow$(b$_1$) is trivial.  For (b$_2$)$\Rightarrow$(b$_3$) note that ${\bf E}_{top}(\Z(p^n)^\N)= \{m\log p: m\in \N\}$. 
$\Box$

\subsection{The entropy of generalized shifts and set-theoretic entropies}\label{genshift}

For a set $X$, a selfmap $\lambda:X\to X$ and a fixed non-trivial group $K$, the \emph{generalized shift}  $\sigma_\lambda:K^X \to K^X$ is defined by $\sigma_\lambda(f) = f\circ \lambda $ for $f\in K^X$ \cite{AKH} (see also \cite{AADGH,FD}). Then $K^{(X)}$ is $\sigma_\lambda$-invariant precisely when $\lambda$ is finitely many-to-one. We denote $\sigma_\lambda\restriction_{K^{(X)}}$ by $\sigma_\lambda^\oplus$. 

\smallskip
Both the left and the two-sided Bernoulli shifts can be obtained as generalized shifts, while the right Bernoulli shift requires a little adjustment:

\begin{example}\label{bernoulli}
Let $K$ be a non-trivial finite abelian group, and consider the Bernoulli shifts $\beta_K,\ _K\beta:K^\N\to K^\N$ and $\overline\beta_K:K^\Z\to K^\Z$. Then:
\begin{itemize}
\item[(a$_1$)] $_K\beta=\sigma_{\rho}$, with $\rho:\N\to\N$ defined by $n\mapsto n+1$ for every $n\in\N$;
\item[(a$_2$)] $\overline\sigma_K=\sigma_{\overline\varsigma}$, with $\overline\sigma:\Z\to\Z$ defined by $n\mapsto n-1$ for every $n\in\Z$;
\item[(a$_3$)] let $\varsigma:\N\to\N$ be defined by $n\mapsto n-1$ for every $n\in\N_+$ and $0\mapsto 0$; then $\sigma_{\varsigma}:K^\N\to K^\N$ is given by
$(x_0,x_1,x_2,\ldots)\mapsto (x_0,x_0,x_1,\ldots)$. Moreover $h_{alg}(\beta_K)=h_{alg}(\sigma_{\varsigma})$, since $\beta_K\restriction_{K^{\N_+}}=\sigma_{\varsigma}\restriction_{K^{\N_+}}$ and $K^\N/K^{\N_+}\cong K$ is finite, so it is possible to apply the Addition Theorem \ref{AT}. Then also $h_{alg}(\beta_K^\oplus)=h_{alg}(\sigma_{\varsigma}^\oplus)$.
\end{itemize}
\end{example}

\medskip
For a finite subset $Y$ of $X$, let $G_Y=K^{(Y)}\subseteq K^{(X)}$. We denote by $\Delta G_Y$ the diagonal subgroup of $G_Y$ (so $|\Delta G_Y|=|K|$).

\begin{lemma}\label{invariant} Let $X$ be a set, $\lambda:X\to X$ a finitely many-to-one selfmap, $K$ a finite group and $F\in[X]^{<\omega}$. Then: 
\begin{itemize}
\item[(a)] $G_F$ is $\sigma^\oplus_\lambda$-invariant if and only if $F$ is $\lambda^{-1}$-invariant.
\item[(b)] $(\sigma_\lambda^\oplus)^n(G_F)=\bigoplus_{i\in F}\Delta G_{\lambda^{-n}(i)}\subseteq G_{\lambda^{-n}(F)}$ for every $n\in\N_+$;
\item[(c)] if $F$ is an antichain, then $H_{alg}(\sigma_\lambda^\oplus,G_F)=|F|\log|K|$.
\end{itemize}
\end{lemma}

\begin{proof} 
(a) and (b) are clear and  proved in {\cite{AADGH}}.

\smallskip
(c) If $k\in\N_+$, then $(\sigma_\lambda^\oplus)^k(G_F)=\bigoplus_{i\in F}\Delta G_{\lambda^{-k}(i)}$ by (b); in particular $|(\sigma_\lambda^\oplus)^k(G_F)|=|K|^{|F|}$. Moreover $\lambda^{-n}(F)\cap \mathfrak T_n^*(\lambda,F)=\emptyset$ for every $n\in\N_+$ by Lemma \ref{h*:E->x}(a), and so $T_n(\sigma_\lambda^\oplus,G_F)=G_F\oplus G_{\lambda^{-1}(F)}\oplus\ldots\oplus G_{\lambda^{-n+1}(F)}$ for every $n\in\N_+$. Therefore $|T_n(\sigma_\lambda^\oplus,G_F)|=|K|^{n|F|}$, and hence $H(\sigma_\lambda^\oplus,G_F)=|F|\log|K|$.
\end{proof}

There is a simple relation between the algebraic entropy of the generalized shift  $\sigma_\lambda^\oplus$ and the contravariant set-theoretical entropy $\mathfrak h^*(\lambda)$ introduced in Section \ref{stent-sec}:

\begin{theorem}\label{genshiftsum} 
Let $X$ be a set, $\lambda:X\to X$ a finitely many-to-one selfmap and $K$ a finite group. Then
$$h_{alg}(\sigma_\lambda^\oplus)= \mathfrak h^*(\lambda) \log |K|.$$
\end{theorem}

\begin{proof} We can assume without loss of generality that $\lambda$ is surjective. Indeed $h_{alg}(\sigma_\lambda^\oplus)=h_{alg}(\sigma_{\lambda\restriction_{\mathrm{sc}(\lambda)}}^\oplus)$ was proved in \cite{AADGH}. Moreover $\mathfrak h^*(\lambda)=\mathfrak h^*(\lambda\restriction_{\mathrm{sc}(\lambda)})$ by definition.

\smallskip
Every $H\in [K^{(X)}]^{<\omega}$ is contained in $G_F=K^F\subseteq K^{(X)}$ for some $F\in[X]^{<\omega}$. So 
$$h_{alg}(\sigma_\lambda^\oplus)=\sup\{H_{alg}(\sigma_\lambda^\oplus,G_F):F\in[X]^{<\omega}\},$$ 
and it suffices to compute $H_{alg}(\sigma_\lambda^\oplus,G_F)$. So fix $F\in[X]^{<\omega}$ and let $n\in\N_+$. Then, by Lemma \ref{invariant},
\begin{align*}
T_n(\sigma_\lambda^\oplus,G_F)&=G_F+ \sigma_\lambda^\oplus(G_F)+\ldots+(\sigma_\lambda^\oplus)^{n-1}(G_F)\\
&\subseteq G_{F\cup\lambda^{-1}(F)\cup\ldots\cup\lambda^{-n+1}(F)}\\
&=G_{\mathfrak T_n^*(\lambda,F)}.
\end{align*}
Therefore 
$$
|T_n(\sigma_\lambda^\oplus,G_F)|\leq|G_{\mathfrak T_n^*(\lambda,F)}|=|K|^{|\mathfrak T_n^*(\lambda,F)|}.
$$
In view of the definitions, this gives 
$H_{alg}(\sigma_\lambda^\oplus,G_F)\leq\mathfrak h^*(\lambda,F)\cdot\log|K|\leq h^*(\lambda)\cdot\log|K|$, so
$$
h_{alg}(\sigma_\lambda)=\sup\{H_{alg}(\sigma_\lambda^\oplus,G_F):F\in[X]^{<\omega}\}\leq
\mathfrak h^*(\lambda)\cdot\log|K|;
$$
consequently $h_{alg}(\sigma_\lambda^\oplus)\leq \mathfrak h^*(\lambda) \log |K|.$

\smallskip
To prove the converse inequality we have to verify that $\mathfrak h^*(\lambda,E)\log|K|\leq h_{alg}(\sigma_\lambda^\oplus)$ for every $E\in[X]^{<\omega}$. In view of \eqref{antieq} we can assume without loss of generality that $E\subseteq X\setminus \mathrm{Per}(\lambda,X)$ and that $E$ is an antichain. By Lemma \ref{h*:E->x}(c$_3$), 
$$
\mathfrak h^*(\lambda,E)=\sum_{x\in E}\mathfrak h^*(\lambda,\{x\}).
$$ 
Since $\uparrow\! x$ is inversely $\lambda$-invariant for every $x\in E$, Lemma \ref{invariant} implies that $G_{\uparrow\! x}$ is $\sigma_\lambda^\oplus$-invariant for every $x\in E$. Moreover $G_{\uparrow\! x}\cap G_{\uparrow\! y}=0$ for every $x,y\in E$ with $x\neq y$, since $\uparrow\! x\cap \uparrow\! y=\emptyset$ by Lemma \ref{h*:E->x}(a). Therefore 
$$
h_{alg}(\sigma_\lambda^\oplus\restriction_{G_{\uparrow\! E}})=\sum_{x\in E} h_{alg}(\sigma_\lambda^\oplus\restriction_{G_{\uparrow\! x}}).
$$
So, it suffices to prove that, for every $x\in E$,
\begin{equation}\label{hx-eq}
\mathfrak h^*(\lambda,\{x\})\log|K|\leq h_{alg}(\sigma_\lambda^\oplus\restriction_{G_{\uparrow\! x}}).
\end{equation}
Indeed, this implies $\mathfrak h^*(\lambda,E)\log|K|\leq h_{alg}(\sigma_\lambda^\oplus\restriction_{G_{\uparrow\! E}})\leq h_{alg}(\sigma_\lambda^\oplus)$. So the last part of the proof is dedicated to verify \eqref{hx-eq} for a fixed $x\in E$.

First assume that there exist infinitely many ramification points of $\lambda$ over $x$. By Proposition \ref{ram-rec} for every $m\in\N_+$ there exists an antichain $F_m\in[X]^{<\omega}$ such that $F_m\subseteq \uparrow\! x$ and $|F_m|\geq m$.
By Lemma \ref{invariant}(c) 
$$
H_{alg}(\sigma_\lambda^\oplus\restriction_{G_{\uparrow\! x}}, G_{F_m})=H_{alg}(\sigma_\lambda^\oplus, G_{F_m})=|F_m|\log|K|\geq m\log|K|.
$$ 
Since this holds for every $m\in\N_+$, we can conclude that $h_{alg}(\sigma_\lambda^\oplus\restriction_{G_{\uparrow\! x}})=h_{alg}(\sigma_\lambda^\oplus)=\infty$, so in this case \eqref{hx-eq} holds true.

Assume now that there exist finitely many ramification points of $\lambda$ over $x$. Then $\mathfrak h^*(\lambda,\{x\})<\infty$ according to Proposition \ref{ram-rec}, so by Corollary \ref{max-a-chain}, there exists a stratifiable antichain $F\subseteq \uparrow\! x$ of $\{x\}$, hence $\mathfrak h^*(\lambda,\{x\})=\mathfrak h^*(\lambda,F)=|F|$. By Lemma \ref{invariant}(c) $$H_{alg}(\sigma_\lambda^\oplus\restriction_{G_{\uparrow\! x}}, G_{F})=H_{alg}(\sigma_\lambda^\oplus, G_{F})=|F|\log|K|=\mathfrak h^*(\lambda,\{x\})\log|K|.$$ In particular this gives the required inequality in \eqref{hx-eq}.
\end{proof}

Since $\mathfrak h^*(\lambda)=\mathfrak s(\lambda)$ in view of Theorem \ref{ConsetE}, Theorem \ref{genshiftsum} is a generalization to not necessarily abelian groups of the main theorem of \cite{AADGH}, stating that, for $\lambda:X\to X$ a finitely many-to-one selfmap of $X$ and $K$ a finite abelian group, $h_{alg}(\sigma_\lambda^\oplus)=\mathfrak s(\lambda)\log|K|$.

\medskip
In the sequel $X$ is an infinite set, $\lambda:X\to X$ an arbitrary selfmap and $K$ a finite set. Then $K^X$ is a compact space considered with the product topology and the generalized shift $\sigma_\lambda:K^X\to K^X$ is continuous. Theorem \ref{genshiftprod} below computes the topological entropy $h_{top}(\sigma_\lambda)$ of  $\sigma_\lambda$ in terms of the set-theoretical entropy $\mathfrak h(\lambda)$ (i.e, the  infinite orbit number $\mathfrak o(\lambda)$) recalled in Section \ref{cstent-sec}. This theorem was proved in a slightly different way in \cite{FD}. 

\begin{theorem}\label{genshiftprod} {\rm  \cite{FD}}
Let $X$ be a set, $\lambda:X\to X$ a selfmap and $K$ a finite non-empty set. Consider $K^X$ endowed with the product topology. Then 
\begin{equation}\label{BT_}
h_{top}(\sigma_\lambda)=\mathfrak h(\lambda)\log|K|= \mathfrak o(\lambda)\log|K|.
\end{equation}
\end{theorem}
\begin{proof} 
According to \S \ref{bowen-sec}, $h_{top}(\sigma_\lambda)=h_U(\sigma_\lambda)=k(\sigma_\lambda)$, so we are going to prove that $k(\sigma_\lambda)= \mathfrak h(\lambda)\log|K|$. We may assume without loss of generality that $K$ is a group. Then $\sigma_\lambda$ is an endomorphism
and for a subset $Y\ne \emptyset $ of $X$ we may identify the subproduct $K^Y$ with a subgroup of $K^X$ in an obvious way. 
The next lemma is needed in the proof of Theorem \ref{genshiftprod}.

\begin{lemma}\emph{\cite{G0}}
Let $X$ be a set, $\lambda:X\to X$ a selfmap, $K$ a finite group, $F\in[X]^{<\omega}$ and $n\in\N_+$. Denoting $U_F=K^{X\setminus F}$, then
$$\sigma_\lambda^{-n}(U_F)=U_{\lambda^n(F)}.$$
\end{lemma}

To continue the proof of Theorem \ref{genshiftprod} note that the family $\{U_F:F\in[X]^{<\omega}\}$, with $U_F=K^{X\setminus F}$, is a base of the neighborhoods of the identity in $K^X$ consisting of open compact subgroups. So $k(\sigma_\lambda)=\sup\{k(\sigma_\lambda,U_F):F\in[X]^{<\omega}\}$, and it suffices to compute $k(\sigma_\lambda,U_F)$ using Proposition \ref{no-mu}. So fix $F\in[X]^{<\omega}$ and let $n\in\N_+$. Then 
\begin{align*}
C_n(\sigma_\lambda,U_F)&=U_F\cap \sigma_\lambda^{-1}(U_F)\cap\ldots\cap\sigma_\lambda^{-n+1}(U_F)\\
&=U_{F\cup\lambda(F)\cup\ldots\cup\lambda^{n-1}(F)}=\; U_{\mathfrak T_n(\lambda,F)}.
\end{align*}
Therefore 
$$
[K^X:C_n(\sigma_\lambda,U_F)]=[K^X:U_{\mathfrak T_n(\lambda,F)}]=|K|^{|\mathfrak T_n(\lambda,F)|}.
$$
By the definitions, $k(\sigma_\lambda,U_F)=\mathfrak h(\lambda,F)\cdot\log|K|$, so
$$
k(\sigma_\lambda)=\sup\{k(\sigma_\lambda,U_F):F\in[X]^{<\omega}\}=\sup\{\mathfrak h(\lambda,F)\cdot\log|K|:F\in[X]^{<\omega}\}=\mathfrak h(\lambda)\cdot\log|K|;
$$
this concludes the proof in view of Theorem \ref{setE}.
\end{proof}

\begin{remark}\label{LaastRemark}
\begin{itemize}
\item[(a)]  The algebraic entropy of the generalized shift $\sigma_\lambda$ considered on the product $K^X$ was calculated in \cite{G0}; it is either zero or infinity, depending on the properties of $\lambda$, and more precisely on the string number, the infinite orbit number and other two analogous functions called ladder number and periodic ladder number.
\item[(b)] The Bridge Theorems as well as Theorems \ref{genshiftsum} and \ref{genshiftprod} can be viewed as results on {\em preservation of entropy along functors}. Indeed, in the last case one takes the functor $F_K: \mathbf{Set}\to \mathbf{Comp}$  defined by $F_K(X) = K^X$ and for $\lambda: X \to Y$ the map $\sigma_\lambda: K^Y \to K^X$ defined by $ \sigma_\lambda(f) =f  \circ  \lambda $ for $f \in K^Y$. 
Then the entropy $\mathfrak h$ in $\mathbf{Set}$ and the entropy $h_{top}$ in $\mathbf{Comp}$ are connected by (\ref{BT_}), which can be interpreted as  preservation of the 
entropies up to a multiplicative constant $\log |K|$. A similar interpretation can be given to Theorem \ref{genshiftsum}, but in this case the category $\mathbf{Set}$ must be replaced by its 
non-full subcategory $\mathbf{Set}_{\mathrm{fin}}$ having as objects all sets and as morphisms only the finitely many-to one maps.  
\end{itemize}
\end{remark}

Motivated by item (b) of Remark \ref{LaastRemark} one can consider the usual Stone-\v Cech compactification functor $\beta: \mathbf{Set}\to \mathbf{Comp}$ (as usual, every set $X$ is considered as a discrete topological space) and ask:

\begin{question}
Is there an appropriate entropy function in $\mathbf{Set}$ so that the functor $\beta$ preserves the entropies?  
\end{question}

Let us mention that the counterpart of this problem in the case of the one-point Alexandroff compactificiation $A(X)$ of an infinite discrete space $X$ and the 
assignment $\lambda \mapsto A(\lambda)$,  where $A(\lambda): A(X) \to A(X)$ is the unique continuous extension of $\lambda$ to $A(X)$, has been resolved
in \cite{FD} as follows. First of all, a necessary condition for the continuity of $A(\lambda)$ is $\lambda$ to be finitely many-to-one (i.e., $\lambda: X \to X$ is in $\mathbf{Set}_\mathrm{fin}$).
In such a case, one has always $h_{top}(A(\lambda))=0$, so the unique possible entropy function in $ \mathbf{Set}_\mathrm{fin}$ such that $\lambda \mapsto A(\lambda)$ preserves entropy is the constant zero.   

\begin{remark} 
Let $X$ be a set and $\lambda:X\to X$ a selfmap. If one considers a finite abelian group $K$, then one may want to apply the Pontryagin duality and the Bridge Theorem in order to calculate the algebraic entropy of the dual $\widehat{\sigma_\lambda}$ of $\sigma_\lambda$ instead of the topological entropy of $\sigma_\lambda$. But $\widehat{\sigma_\lambda}$ is not a generalized shift in general.
Indeed, one can prove that:
\begin{itemize}
\item[(a)] If $\lambda$ is a bijection, then $\widehat{\sigma_\lambda^\oplus}=\sigma_{\lambda^{-1}}$ is a generalized shift.
\item[(b)] Assume now that $\lambda$ is a finitely many-to-one selfmap. Let $\chi=(\chi_i)_{i\in X}\in K^X\cong \widehat{K^{(X)}}$. Then 
$$\widehat{\sigma_\lambda^\oplus}(\chi)=\chi\circ\sigma_\lambda^\oplus=\left(\sum_{j\in\lambda^{-1}(i)}\chi_j\right)_{i\in X}.$$
\item[(c)] Consider $\sigma_\lambda:K^X\to K^X$ and let $\chi\in \widehat{K^X}\cong K^{(X)}$. Then there exists a smallest finite subset $F_\chi$ of $X$ such that $\chi(K^{X\setminus F_\chi})=0$. Since $\widehat{\sigma_\lambda}(\chi)=\chi\circ\sigma_\lambda$, $\widehat{\sigma_\lambda}(\chi)(K^{X\setminus F_\chi})=0$ as well. Identify $\chi=(\chi_i)_{i\in F_\chi}$. Then $$\widehat{\sigma_\lambda}(\chi)=\left(\sum_{j\in F_\chi\cap\lambda^{-1}(i)}\chi_j\right)_{i\in F_\chi}.$$
\end{itemize}
Hence taking the dual of a generalized shift in items (b) and (c) we do not obtain a generalized shift.
\end{remark}

\subsection{Adjoint entropy vs topological entropy and algebraic entropy}\label{LAAAAST}

Using Proposition \ref{no-mu}, it is possible to deduce the following result giving a relation between the topological adjoint entropy and the topological entropy. It generalizes Theorem \ref{ent*=htop} to totally disconnected compact groups that are not necessarily abelian.

\begin{theorem}
Let $(K,\tau)$ be a totally disconnected compact group and $\psi:(K,\tau)\to (K,\tau)$ a continuous endomorphism. Then $$\aent_\tau(\psi)=k(\psi)=h_{top}(\psi).$$
\end{theorem}
\begin{proof}
In this case $\mathcal B(K,\tau)=\CC_\tau(K)$. By Proposition \ref{no-mu}, if $U\in\mathcal B(K,\tau)$, then $k(\psi,U)=H^\star(\psi,U)$. Lemma \ref{B(G)} gives $$\aent_\tau(\psi)=\sup\{H^\star(\psi,U):U\in\CC_\tau(K)\}=\sup\{k(\psi,U):U\in\mathcal B(K,\tau)\}=k(\psi).$$
The last equality follows from the known fact that $k$ and $h_{top}$ coincide on continuous endomorphisms of compact groups.
\end{proof}

The next result from \cite{DGS} connects the adjoint algebraic entropy with the algebraic entropy $\ent$ through Pontryagin duality. We recall that for an abelian group $G$ and an endomorphism $\phi:G\to G$ we have $\ent(\phi)=h_{alg}(\phi\restriction_{t(G)})$ (see Remark \ref{locfin}(d)).

\begin{theorem}[Bridge Theorem]\label{ent*=ent^}\emph{\cite{DGS}}
Let $K$ be a compact abelian group and $\psi:K\to K$ an endomorphism. Then $\ent(\psi)=\ent^\star(\widehat\psi)$.
\end{theorem}

A consequence of this theorem is the following exotic criterion for determining the continuity of an endomorphism of a compact abelian group.

\begin{corollary}
Let $K$ be a compact abelian group and $\psi:K\to K$ an endomorphism. If $\ent(\psi)$ is finite and positive, then $\psi:K\to K$ is not continuous.
\end{corollary}

The following Bridge Theorem \ref{BT*}, relating the topological adjoint entropy and the algebraic entropy through the Pontryagin duality, extends Weiss Bridge Theorem \ref{BT-WP} from totally disconnected compact abelian groups to arbitrary compact abelian groups. On the other hand, this is not possible using the topological entropy, so in some sense the adjoint topological entropy may be interpreted as the true ``dual entropy'' of the algebraic entropy $\ent$, as well as the topological entropy is the ``dual entropy'' of the algebraic entropy $h_{alg}$ in view of Theorem \ref{BT}. 

\begin{theorem}[Bridge Theorem]\label{BT*}
Let $(K,\tau)$ be a compact abelian group and $\psi:(G,\tau)\to (G,\tau)$ a continuous endomorphism. Then $$\aent_\tau(\psi)=\ent(\widehat\psi).$$
\end{theorem}

We leave the following open question.

\begin{problem}
Is it possible to extend the Bridge Theorem \ref{BT*} to all continuous endomorphisms of locally compact abelian groups?
\end{problem}

\newpage
\section{Entropy in a category}\label{cav-sec}

In this section we briefly describe a first categorical approach to entropy from \cite{DG1}, tailored to cover the constriction of the Pinkser subgroup
from a categorical point of view. 

\subsection{Covariant entropy function in an abelian category}

We start recalling the definition of (covariant) entropy function in an abelian category. The required set of three axioms is clearly inspired by the properties of the algebraic entropy and it is minimal in order to ensure the construction of the Pinsker radical  below (see Definition \ref{DefiPin}).

\begin{definition}\label{h-def}
Let $\mathfrak M$ be a well-powered cocomplete abelian category. 
An \emph{entropy function} $h$ of $\mathfrak M$ is a function $h:\mathfrak M\to \R_{\geq0}\cup\{\infty\}$ such that:
\begin{itemize}
\item[{(A1)}] $h(0) = 0$ and $h(M)=h(N)$ if $M$ and $N$ are isomorphic objects in $\mathfrak M$;
\item[{(A2)}] $h(M)=0$ if and only if $h(N)=0=h(Q)$ for every exact sequence $0\longrightarrow N\longrightarrow M \longrightarrow Q\longrightarrow0$ in $\mathfrak M$;
\item[{(A3)}] for a set $\{M_j:j\in I\}$ of objects of $\mathfrak M$, $h\left(\bigoplus_{j\in J} M_j\right) = 0$ if and only if $h(M_j) = 0$ for all $j\in J$.
\end{itemize}
An entropy function $h$ of $\mathfrak M$ is \emph{binary} if it takes only the values $0$ and $\infty$.
\end{definition}

The fundamental and inspiring example is given by the algebraic entropy:

\begin{example} Let us first recall that the category $\Flow_\abg$ of flows of the category of abelian groups
is an abeilan category. Indeed, every flow $\f: G \to G$ defines a structure of a $\Z[X]$-module on the abelian group $G$
(where   $\Z[X]$ is the ring of polynomials with integer coefficients). In this way the category $\Flow_\abg$ is isomorphic to the 
category $\mathfrak M$ of $\Z[X]$-module, that is an abelian category. 
Now it suffices to note that the algebraic entropy $h_{alg}$ is an entropy function of the category 
$\mathfrak M$ in view of the Invariance under conjugation Axiom, the Continuity Axiom and the Addition Theorem. 
\end{example}

Extending the notion of Pinsker subgroup, defined for the algebraic entropy (see Definition \ref{pinsker-def}), we have the following general

\begin{definition}\label{DefiPin}
Let $h$ be an entropy function of $\mathfrak M$. The \emph{Pinsker radical}  with respect to $h$ is defined for every object $M$ of $\mathfrak M$ as the largest subobject $\P_h(M)$ of $M$ such that $h(\P_h(M)) =0$. 
\end{definition}

Clearly, such a subobject is uniquely determined by its definition. Its existence easily follows from the properties (A1) -- (A3). Indeed, 
the join 
$$
\P_h(M)=\sum\{N_j\subseteq M: h(N_j)=0\}
$$
has the desired properties. 

Using the properties (A1) -- (A3) one can see that the subobject $\P_h(M)$ of $M$ is functorial. In other words, every morphism $f:M \to N$ in $\mathfrak M$ carries $\P_h(M)$ into $\P_h(N)$ (i.e., $f(\P_h(M))\subseteq \P_h(N)$). Therefore,  one has a 
functor $\P_h:\mathfrak M\to\mathfrak M$ assigning to each $M\in \mathfrak M$ its Pinsker radical $\P_h(M)$. We keep the 
notation $\P_h$ and the name radical for this functor too.

It is natural to consider the class $$\mathcal T_h=\{M\in\mathfrak M: h(M)=0\}$$ for a given entropy function $h$ of $\mathfrak M$.
In terms of Pinsker radical $\P_h$ these are precisely the objects $M$ with $\P_h(M)=M$. 

Moreover we say that an object $M$ of $\mathfrak M$ has \emph{completely positive entropy} if $h(N)>0$ for every non-zero subobject $N$ of $M$, and we denote this by $h(M)>\!\!>0$. As a natural counterpart of $\mathcal T_h$, we define the class of all objects in $\mathfrak M$ with completely positive entropy, that is, $$\mathcal F_h=\{M\in\mathfrak M: h(M)>\!\!>0\}.$$
Obviously, $M \in \mathcal F_h$ if and only if $\P_h(M)=0$. 

Let $\mathfrak t_h=(\mathcal T_h,\mathcal F_h)$. Since $\mathcal T_h=\{M\in\mathfrak M: \P_h(M)=M\}$ and $\mathcal F_h=\{M\in\mathfrak M: \P_h(M)=0\}$, one has the following: 

\begin{theorem}
Let $\mathfrak M$ be a well-powered cocomplete abelian category.  If $h$ is an entropy function of $\mathfrak M$, then $\P_h$ is a hereditary radical of $\mathfrak M$ and $\mathfrak t_h$ is a hereditary torsion theory in $\mathfrak M$.
\end{theorem}

In particular, this theorem can be applied in the known case of the algebraic entropy.

One can define a preorder $\preceq$ of the class $\H(\mathfrak M)$ of all entropy functions of $\mathfrak M$ induced by the order of $\R_{\geq0}\cup\{\infty\}$. It makes $(\H(\mathfrak M),\preceq)$ a complete lattice as well as its sublattice $(\HB(\mathfrak M),\preceq)$ of all binary entropy functions of $\mathfrak M$. The assignment $h\mapsto \mathfrak t_h$
between entropy functions of $\mathfrak M$ and hereditary torsion theories in $\mathfrak M$ is order preserving, but generally
many-to-one. Its restriction to binary entropy functions defines a bijective order preserving correspondence 
between binary entropy functions of $\mathfrak M$ and hereditary torsion theories in $\mathfrak M$. 
 So, there may be information in an entropy function which is not captured by the corresponding hereditary torsion theory, and binary entropy functions are simply those which do not contain any additional information.

\subsection{Contravariant entropy functions}

The topological entropy is not a (covariant) entropy function, so the point is to understand how to deal with this different situation. 
More specifically, the topological entropy $h_{top}$ is continuous with respect to inverse limits when considered in 
the category $\mathbf{CompGrp}$ of all compact groups (see Proposition \ref{invlim}), while the algebraic entropy is continuous with respect to direct limits (see Proposition \ref{dirlim}). This is the substantial difference between the two entropies (see also Remark \ref{Berny}).

The distinction between both types of entropy is well visible also in the case of the ``normalization axiom" that imposes a specific value of the entropy function at the Bernoulli shifts. 

\begin{remark}\label{Berny} For $K\in\abg$, the left Bernoulli shift ${}_K\beta^\oplus$ of $K^{(\N)}$ has algebraic entropy $0$, and the right Bernoulli shift $\beta_K^\oplus$ of $K^{(\N)}$ has algebraic entropy $\log |K|$. 
Conversely, for $K\in\mathbf{CompGrp}$, the left Bernoulli shift ${}_K\beta$ of $K^\N$ has topological entropy $\log |K|$, while the right Bernoulli shift $\beta_K$ of $K^\N$ has topological entropy $0$. 
\end{remark}

This gives a good motivation to split the abstract notion of entropy functions in two dual notions, say \emph{covariant entropy functions} (precisely those of Definition \ref{h-def}) and \emph{contravariant entropy functions} following the pattern of $h_{top}$ on $\mathbf{CompGrp}$. 

\begin{definition}\label{hc-def}
Let $\mathfrak N$ be a complete abelian category. A \emph{contravariant entropy function} of $\mathfrak N$ is an entropy function $h:\mathfrak N^{op}\to \R_{\geq0}\cup\{\infty\}$, where $\mathfrak N^{op}$ is the opposite category of $\mathfrak N$.
\end{definition}

The main example is the topological entropy in the category of compact abelian groups. 

The difference between the two notions (covariant and and contravariant entropy) relies in the continuity property. Indeed, both the covariant and contravariant entropy functions must be invariant under conjugation and satisfy the Addition Theorem (or some weaker version of the Addition Theorem), while the continuity property must be imposed in a selective way: the covariant entropy functions must be continuous with respect to direct limits in $\mathfrak N$, while the contravariant entropy functions must be continuous with respect to inverse limits in $\mathfrak N$. 

\medskip 
This approach works for the category of compact abelian groups which is abelian, while the category $\textbf{CompGrp}$ is semiabelian but not abelian. This suggests to generalize the setting of Definition \ref{hc-def} at least to semiabelian categories to include as example the topological entropy for the larger category $\mathbf{CompGrp}$
Hence we leave the following: 

\begin{problem}\emph{\cite{DG1}}
Develop the theory of covariant (respectively, contravariant) entropy functions in semiabelian cocomplete (respectively, complete) categories.
\end{problem}

\end{document}